\documentclass[10pt]{amsart}

\usepackage{graphicx}
\usepackage{color}
\usepackage{tikz-cd}
\usepackage{mathrsfs}

\definecolor{refblue}{RGB}{0, 0, 153}
\definecolor{citegreen}{RGB}{0, 115, 0}
\definecolor{linkred}{RGB}{191, 26, 61}

\usepackage[marginparwidth=75pt]{geometry}
\usepackage[all]{xy}
\usepackage[colorlinks,linkcolor=refblue,citecolor=citegreen,urlcolor=linkred,bookmarks=false,hypertexnames=true]{hyperref} 
\usepackage{amssymb}
\usepackage{setspace}
\usepackage{stmaryrd}
\usepackage[nameinlink]{cleveref}
\usepackage{hyperref}
\usepackage{microtype}

\newtheorem{theorem}{Theorem}[section]
\newtheorem{lemma}[theorem]{Lemma}
\newtheorem{proposition}[theorem]{Proposition}

\newtheorem{corollary}[theorem]{Corollary}
\newtheorem{maintheorem}{Theorem}

\theoremstyle{definition}
\newtheorem{definition}[theorem]{Definition}
\newtheorem{example}[theorem]{Example}

\theoremstyle{remark}
\newtheorem{remark}[theorem]{Remark}

\numberwithin{equation}{section}

\newcommand\SmallMatrix[1]{{%
  \tiny\arraycolsep=0.3\arraycolsep\ensuremath{\begin{pmatrix}#1\end{pmatrix}}}}

\newcommand{\stackcite}[1]{\cite[\href{https://stacks.math.columbia.edu/tag/#1}{#1}]{stk}}

\newcommand\atopnew[2]{\genfrac{}{}{0pt}{}{#1}{#2}}

\newcommand{\bbA}{\mathbb A}

\newcommand{\bbC}{\mathbb C}
\newcommand{\bbD}{\mathbb D}

\newcommand{\bbF}{\mathbb F}
\newcommand{\bbG}{\mathbb G}

\newcommand{\bbN}{\mathbb N}

\newcommand{\bbQ}{\mathbb Q}

\newcommand{\bbX}{\mathbb X}

\newcommand{\bbZ}{\mathbb Z}

\newcommand{\scrT}{\mathscr T}

\newcommand{\calA}{\mathcal A}

\newcommand{\calC}{\mathcal C}
\newcommand{\calD}{\mathcal D}
\newcommand{\calE}{\mathcal E}
\newcommand{\calF}{\mathcal F}

\newcommand{\calI}{\mathcal I}

\newcommand{\calM}{\mathcal M}

\newcommand{\calO}{\mathcal O}

\newcommand{\calS}{\mathcal S}

\newcommand{\calY}{\mathcal Y}

\newcommand{\frakA}{\mathfrak A}

\newcommand{\frakX}{\mathfrak X}

\newcommand{\frakg}{\mathfrak g}

\newcommand{\frakm}{\mathfrak m}

\newcommand{\frakp}{\mathfrak p}

\newcommand{\gl}{\mathfrak{gl}}

\renewcommand{\sp}{\mathfrak{sp}}
\renewcommand{\sl}{\mathfrak{sl}}

\newcommand{\GL}{\operatorname{GL}} 
\newcommand{\SL}{\operatorname{SL}} 
\newcommand{\SO}{\operatorname{SO}} 
\newcommand{\Sp}{\operatorname{Sp}} 
\newcommand{\GSp}{\operatorname{GSp}} 
\newcommand{\OO}{\operatorname{O}} 
\newcommand{\GO}{\operatorname{GO}} 

\newcommand{\Adm}{\operatorname{\mathrm{Adm}}} 
\newcommand{\Aff}{\operatorname{Aff}} 
\newcommand{\An}{\operatorname{\mathrm{An}}} 
\newcommand{\Art}{\operatorname{\frakA}} 
\newcommand{\FSch}{\operatorname{\mathrm{FSch}}} 
\newcommand{\Rep}{\operatorname{\mathrm{Rep}}} 
\newcommand{\cRep}{\operatorname{\mathrm{cRep}}} 
\newcommand{\Set}{\operatorname{\mathrm{Set}}} 
\newcommand{\CAlg}{\mathrm{CAlg}} 

\newcommand{\Ext}{\operatorname{Ext}} 
\newcommand{\spa}{\operatorname{Spa}} 
\newcommand{\Spec}{\operatorname{Spec}} 
\newcommand{\Spf}{\operatorname{Spf}} 
\newcommand{\Tor}{\operatorname{Tor}} 

\newcommand{\ad}{\operatorname{ad}} 
\newcommand{\Aut}{\operatorname{Aut}} 
\newcommand{\chara}{\operatorname{char}} 
\newcommand{\coker}{\operatorname{coker}} 
\newcommand{\colim}{\mathop{\mathstrut\rm colim}\limits} 
\newcommand{\Det}{\mathrm{Det}}
\newcommand{\cDet}{\mathrm{cDet}}
\newcommand{\Gal}{\operatorname{Gal}} 
\newcommand{\Hom}{\operatorname{Hom}} 
\newcommand{\id}{\operatorname{id}} 
\newcommand{\im}{\operatorname{im}} 
\newcommand{\ind}{\operatorname{ind}} 
\newcommand{\Jac}{\operatorname{Jac}} 
\newcommand{\map}{\operatorname{Map}} 
\newcommand{\Nil}{\operatorname{\mathrm{Nil}}} 
\newcommand{\op}{\operatorname{op}} 
\newcommand{\PC}{\operatorname{PC}} 
\newcommand{\cPC}{\operatorname{cPC}} 
\newcommand{\pr}{\operatorname{pr}} 
\newcommand{\Res}{\operatorname{\mathrm{Res}}} 
\newcommand{\Spin}{\operatorname{Spin}} 
\newcommand{\Sym}{\operatorname{Sym}} 
\newcommand{\tr}{\operatorname{tr}} 
\newcommand{\ps}{\operatorname{ps}} 
\newcommand{\irr}{\operatorname{irr}} 
\newcommand{\spcl}{\operatorname{spcl}} 
\newcommand{\nspcl}{\operatorname{nspcl}} 
\newcommand{\sing}{\operatorname{sing}} 
\newcommand{\pair}{\operatorname{pair}} 
\newcommand{\red}{\operatorname{red}} 
\newcommand{\dec}{\operatorname{dec}} 
\newcommand{\Def}{\operatorname{Def}} 
\newcommand{\univ}{\operatorname{univ}} 
\newcommand{\FM}{\operatorname{FM}} 
\newcommand{\FG}{\operatorname{FG}} 
\newcommand{\Ran}{\operatorname{Ran}} 
\newcommand{\simil}{\operatorname{sim}} 
\newcommand{\jj}{\operatorname{j}} 
\newcommand{\lnad}{\operatorname{lnad}} 
\newcommand{\rig}{\operatorname{rig}} 
\newcommand{\Specmax}{\operatorname{Specmax}} 
\newcommand{\Gm}{\bbG_{\mathrm{m}}} 
\newcommand{\Thetabar}{\overline{\Theta}} 
\newcommand{\Dbar}{\overline{D}} 
\newcommand{\rhobar}{\overline{\rho}} 
\newcommand{\kappabar}{\overline{\kappa}}
\newcommand{\RpsThetabar}{R^{\ps}_{\overline\Theta}} 
\newcommand{\semi}{\mathrm{ss}} 
\newcommand{\et}{\mathrm{\acute et}} 

\begin{document}

\title{Deformations of $G$-valued pseudocharacters}

\author{Julian Quast}

\address{Department of Mathematics, Universität Duisburg-Essen, Thea-Leymann-Straße 9, 45127 Essen}
\email{julian.quast@uni-due.de, me@julianquast.de}

\subjclass[2000]{Primary 11F80, 14J10; Secondary 20G25, 11S25, 14M35}

\date{March 28, 2024}

\keywords{Pseudorepresentation, pseudocharacter, determinant law, deformation theory, Galois representation}

\begin{abstract} We define a deformation space of V. Lafforgue's $G$-valued pseudocharacters of a profinite group $\Gamma$ for a possibly disconnected reductive group $G$. We show that this definition generalizes Chenevier's construction. We show that the universal pseudodeformation ring is noetherian and that the functor of continuous $G$-pseudocharacters on affinoid $\bbQ_p$-algebras is represented by a quasi-Stein rigid analytic space, whenever $\Gamma$ is topologically finitely generated. We also show that the pseudodeformation ring is noetherian when $\Gamma$ satisfies Mazur's condition $\Phi_p$ and $G$ satisfies a certain invariant-theoretic condition. For $G = \Sp_{2n}$ we describe three types of obstructed loci in the special fiber of the universal pseudodeformation space of an arbitrary residual pseudocharacter and give upper bounds for their dimension.
\end{abstract}

\footnote{Department of Mathematics, Universität Duisburg-Essen, Thea-Leymann-Straße 9, D-45127 Essen. Gefördert durch die Deutsche Forschungsgemeinschaft (DFG) - Projektnummer 444845124 - TRR 326 und Projektnummer 517234220.}

\maketitle

\tableofcontents

\newpage

\sloppy

\section{Introduction}

If $\rho : \Gamma \to \GL_d(k)$ is a representation of a group $\Gamma$ which takes values in an algebraically closed field $k$ of characteristic $0$, then it is well-known that the character
$$ \tr \rho : \Gamma \to k, \quad \gamma \mapsto \tr(\rho(\gamma)) $$
determines $\rho$ up to semisimplification.
To study congruences of Galois representations Wiles \cite{Wiles1988} and Taylor \cite{MR1115109} introduced the notion of an $n$-dimensional \emph{pseudocharacter}, which captures the essential properties of traces of representations axiomatically. An $n$-dimensional pseudocharacter over $k$ is always the trace of a semisimple representation.

If the characteristic of $k$ is positive and $\leq n$, then the semisimplification of $\rho$ is not determined by the trace anymore.
To remedy the problem, building on work of Procesi, Chenevier \cite{MR3444227} introduced the notion of an \emph{$n$-dimensional determinant law} which keeps track of all coefficients of characteristic polynomials of $\rho(\gamma)$ for all $\gamma \in \Gamma$, and works over arbitrary base rings. Now, an $n$-dimensional determinant law over an algebraically closed field of arbitrary characteristic is always associated to a semisimple representation.

Our main results extend the work of Chenevier to possibly disconnected reductive group schemes $G$ over the ring of integers $\calO$ of a $p$-adic local field $L$.
Instead of generalizing Chenevier's notion of determinant law to arbitrary reductive groups we work with \emph{$G$-pseudocharacters} introduced by V. Lafforgue in his work on the local Langlands correspondence for function fields.
When $G = \GL_n$, then Emerson and Morel \cite{emerson2023comparison} have shown that $n$-dimensional determinant laws coincide with $\GL_n$-pseudocharacters.
Lafforgue's $G$-pseudocharacters also play an important role in the geometrization of the local Langlands correspondence for $p$-adic fields due to Fargues and Scholze \cite{fargues2021geometrization}.

Any $G$-valued representation $\rho$ gives rise to a $G$-pseudocharacter $\Theta_{\rho}$.
Our first result is a \emph{reconstruction theorem}, which relates $G$-pseudocharacters over algebraically closed fields to conjugacy classes of semisimple representations.

\begin{maintheorem}[\Cref{reconstructiongeneral}, \Cref{continuousreconstruction}]\label{ThmRec} Let $\calO$ be a noetherian commutative ring, let $G$ be a generalized reductive $\calO$-group scheme, let $\Gamma$ be a group, let $k$ be an algebraically closed field over $\calO$ and let $\Theta$ be a $G$-pseudocharacter of $\Gamma$ over $k$. Then there is a $G$-completely reducible representation $\rho : \Gamma \to G(k)$ with $\Theta_{\rho} = \Theta$, which is unique up to $G^0(k)$-conjugation. If $k$ is an algebraic closure of a local field, $\Gamma$ is a profinite group and $\Theta$ is continuous, then $\rho$ is continuous.
\end{maintheorem}

The proof of \Cref{ThmRec} relies on results in geometric invariant theory due to Bate, Martin and Röhrle \cite{BMR, MartinGeneratingTuples}.
In \Cref{compprop} we obtain a similar result, which allows us to compare deformation functors of $G$-pseudocharacters with unframed deformation functors of representations.

Let $\Thetabar$ be a continuous $G$-pseudocharacter of a profinite group $\Gamma$ with values in a finite discrete field $\kappa$.
We introduce the universal deformation ring $\RpsThetabar$ for $\Thetabar$, replacing Chenevier's determinant laws by Lafforgue's $G$-pseudocharacters as introduced in \cite[§11]{Laf}. This deformation problem has been considered for connected reductive groups over $\bbZ$ in \cite{BHKT}. Our main result is that these rings are noetherian when $\Gamma$ is topologically finitely generated and for certain $G$, when $\Gamma$ satisfies Mazur's $p$-finiteness condition.

\begin{maintheorem}[\Cref{tfgfingen}, \Cref{Phipmainthm}, \Cref{H1mainthm}]\label{ThmA} Let $G$ be a generalized reductive $\calO$-group scheme, let $\Gamma$ be a profinite group and let $\Thetabar$ be a continuous $G$-pseudocharacter of $\Gamma$ over $\kappa$. Denote by $\kappa[G^m]^{G^0}$ the invariants under the action of $G^0$ on $G^m$ by simultaneous conjugation.
\begin{enumerate}
    \item If $\Gamma$ is topologically finitely generated, then the universal deformation ring $\RpsThetabar$ of $\Thetabar$ is noetherian. \label{B1}
    \item Assume that $G$ admits a representation $\iota : G \to \GL_d$, such that for all $m \geq 1$ the map $\kappa[\GL_d^m]^{\GL_d} \to \kappa[G^m]^{G^0}$ is surjective and assume that $\Gamma$ satisfies Mazur's condition $\Phi_p$. Then the canonical map $R^{\ps}_{\iota(\Thetabar)} \to \RpsThetabar$ is surjective and $R^{\ps}_{\Thetabar}$ is noetherian.
    \item Assume that $G$ admits a representation $\iota : G \to \GL_d$, such that for all $m \geq 1$ the map $\kappa[\GL_d^m]^{\GL_d} \to \kappa[G^m]^{G^0}$ is surjective. Assume further, that $\Thetabar$ comes from a continuous representation $\rhobar : \Gamma \to G(\kappa)$, that $\dim_{\kappa} H^1(\Gamma, \ad(\iota(\rhobar))) < \infty$ and $p > d$. Then $R^{\ps}_{\Thetabar}$ is noetherian.
\end{enumerate}
\end{maintheorem}

Our strategy for proving (1) relies only on the definitions and the assumption that $\Gamma$ is topologically finitely generated.
In particular, (1) applies, when $\Gamma$ is the absolute Galois group of a $p$-adic local field, since these groups are topologically finitely generated, see \cite[Satz 3.6]{Jannsen1982}.
Parts (2) and (3) are deduced from similar results of Chenevier.
We verify surjectivity of $\kappa[\GL_d^m]^{\GL_d} \to \kappa[G^m]^{G^0}$ for ${G \in \{\SL_n, \GL_n, \Sp_{2n}, \GSp_{2n}, \SO_{2n+1}, \OO_{2n+1}, \GO_n\}}$ under the additional assumption $p > 2$ in the orthogonal cases for $\iota$ the standard representation (\Cref{prtorsfingen}).

Chenevier introduced and studied the generic fiber of his universal deformation rings \cite{MR3444227, SurLaVariete}.
The functor of points of the generic fiber has a simple description in terms of continuous determinant laws and its geometric points correspond to conjugacy classes of continuous $G$-completely reducible representations. We obtain analogous results for general $G$.

\begin{maintheorem}[\Cref{rigspaceG}, \Cref{Lbarpoints}]\label{ThmB}
    If $\Gamma$ is topologically finitely generated, the functor $X_G : \Adm_L \to \Set$, that associates to every affinoid $L$-algebra $A$ the set of continuous $G$-pseudocharacters $\cPC_{G}^{\Gamma}(A)$ is representable by a quasi-Stein rigid-analytic space. The $\overline L$-points of $X_G$ are in canonical bijection with $G^0(\overline L)$-conjugacy classes of continuous $G$-completely reducible representations $\Gamma \to G(\overline L)$.
\end{maintheorem}

The proof of \Cref{ThmB} relies on \Cref{ThmRec} and \Cref{ThmA}.
This is a $p$-adic analog of the GIT quotient of the space of Langlands parameters studied in the $\ell \neq p$ case in \cite{DHKM, Zhu, fargues2021geometrization}.

Let $F/\bbQ_p$ be a $p$-adic local field, fix an algebraic closure $F \hookrightarrow \overline F$ and let $\Gamma := \Gamma_F := \Gal(\overline F/F)$ be the absolute Galois group of $F$. 
In the special case when $G=\Sp_{2n}$ we study certain obstructed loci $\overline X_{\Thetabar}^{\dec}$, $\overline X_{\Thetabar}^{\pair}$ and $\overline X_{\Thetabar}^{\spcl}$ (see \Cref{defsubloci}) of the special fiber $\overline X_{\Thetabar}$ of $\Spec(R_{\Thetabar}^{\ps})$ analogous to \cite[§3.4]{BIP} and \cite{BJ_new}.
Here, the geometric points of $\overline X_{\Thetabar}^{\dec}$ are given by representations decomposable as symplectic representations, $\overline X_{\Thetabar}^{\pair}$ are representations not decomposable as symplectic representations but decomposable as representations, these are called of pair type, and $\overline X_{\Thetabar}^{\spcl}$ are representations which are irreducible, but with nonvanishing obstruction group $H^2$.
These three obstructed loci will contain all points $x \in \overline X_{\Thetabar}$ of dimension $1$, where $H^2(\Gamma_F, \sp_{2n,\kappa(x)}) \neq 0$.

\begin{maintheorem}[\Cref{estspcl}, \Cref{estred}, \Cref{endcor}]\label{ThmC} Let $\Thetabar$ be a continuous $\Sp_{2n}$-pseudocharacter of $\Gamma_F$ over $\kappa$.
\begin{enumerate}
    \item $\dim \overline X_{\Thetabar}^{\dec} \leq n(2n+1)[F : \bbQ_p] - 4(n-1)[F : \bbQ_p]$.
    \item $\dim \overline X_{\Thetabar}^{\pair} \leq n^2[F : \bbQ_p] + 1$.
    \item $\dim \overline X_{\Thetabar}^{\spcl} \leq 2n^2[F : \bbQ_p] + 1$.
    \item $\dim \overline X_{\overline \Theta} \leq n(2n+1)[F : \bbQ_p]$.
\end{enumerate}
If $\Thetabar$ comes from an absolutely irreducible representation, then in (4) equality holds and $\overline X_{\Thetabar}^{\spcl} \subsetneq \overline X_{\overline \Theta}$.
\end{maintheorem}

With these results at hand one could prove the analog of the main result of \cite{BIP} for $G=\Sp_{2n}$.
In joint work with Vytautas Pa\v{s}k\={u}nas we will generalize the results of \cite{BIP} by replacing $\GL_d$ with a generalized reductive subgroup scheme.
Essential ingredients in the proof will be \Cref{ThmRec} and \Cref{ThmA}.

In joint work with Mohamed Moakher \cite{MQ} we introduced an analog of Chenevier's determinant laws for $\Sp_{2n}$ and $\GSp_{2n}$, which we call \emph{symplectic determinant laws}. They can be seen as an alternative approach to pseudocharacters for representations valued in symplectic groups which is based on algebras with involution.
We expect that symplectic determinant laws allow for a detailed analysis of deformation problems when the residual representation is reducible and multiplicity-free.

\textbf{Organization of the paper.} In \Cref{secgrpsch} we introduce the notion of a \emph{generalized reductive group scheme} $G$, which also allows for $G$ to be disconnected. We recall the notion of a \emph{good filtration} of a $G$-module in the connected case in \Cref{subsecGF} and prove a version of Mathieu's tensor product theorem over a principal ideal domain in \Cref{subsecMathieu}. We use these results in \Cref{secOGm} to show that (under technical hypotheses) the modules $\calO[G^m]$ have a good filtration under the conjugation action of $G^0$. Serre's notions of \emph{$G$-complete reducibility} and \emph{$G$-semisimplification} are extended in \Cref{subsecGss} to generalized reductive group schemes.
In \Cref{secGPC} we develop the foundations of the theory of $G$-pseudocharacters over arbitrary base rings.
We show in \Cref{secRecthm}, that $G$-pseudocharacters over algebraically closed fields $k$ are in bijection with $G$-completely reducible representations up to $G^0(k)$-conjugation, establishing \Cref{ThmRec}.
In \Cref{secsumtensor} we introduce direct sums, duals and tensor products of $\GL_n$-pseudocharacters and analogs for the symplectic groups $\Sp_{2n}$.
In \Cref{secRep} we remove the dependence of the definition of $G$-pseudocharacter on the base ring of $G$.
In \Cref{secinvthy} we calculate generators of invariant algebras for certain classical groups.
In \Cref{secdefGPC} we define the deformation problem for $G$-pseudocharacters and prove \Cref{ThmA} (1) in \Cref{subsecnoethTFG} and \Cref{ThmA} (2) and (3) in \Cref{subsecnoethPhip} building on a result of Chenevier.
Our deformation problem is compared in \Cref{subsecCompChen} with Chenevier's deformation problem for determinant laws using the comparison theorem of Emerson and Morel.
In \Cref{subsecdefComp} we give criteria under which there is an isomorphism between the deformation ring of a $G$-valued representation and the deformation ring of the associated $G$-pseudocharacter. In \Cref{secpadicspace} we define the rigid-analytic space of continuous $G$-pseudocharacters and establish \Cref{ThmB} building on \Cref{ThmRec} and \Cref{ThmA}.
In \Cref{secsubdiv} we explain our stratification of the special fiber and prove \Cref{ThmC} in \Cref{secboundsforSp}.

\textbf{Acknowledgments.} I would like to thank my doctoral thesis advisor Gebhard Böckle for suggesting the topic and continuous support and advice during the writing of this paper. I would also like to thank Vytautas Pa\v{s}k\={u}nas for insightful discussions about \cite{BIP} and comments on earlier versions of this article. I thank Sophie Morel for helpful conversations on the results of Kathleen Emerson's thesis \cite{Emerson2018ComparisonOD} and \cite{emerson2023comparison}. Finally I thank Mohamed Moakher, Ariel Weiss and Stephen Donkin for helpful conversations about classical invariant theory and good filtrations.

\textbf{Conventions.} Let $A$ be a topological commutative ring and let $X$ be an affine scheme of finite type over $A$. We endow $X(A)$ with the subspace topology of $X(A) \hookrightarrow \bbA^n(A) = A^n$ for a closed immersion $i : X \hookrightarrow \bbA^n$, where $A^n$ carries the product topology. This topology is independent of $i$, makes every map of $A$-schemes continuous and is compatible with fiber products \cite[Proposition 2.1]{ConradTopologies}.
We also need a topology in the following situation: Let $\kappa$ be a topological field, and let $A$ be a finite-dimensional local $\kappa$-algebra with residue field $\kappa$, equipped with the product topology induced by an isomorphism $A \cong \kappa^n$ of $\kappa$-vector spaces. If $X$ is an affine $A$-scheme of finite type, the map $X(A) \to X(\kappa)$ is continuous. Let $Z \subseteq X(A)$ be the preimage of a Zariski-closed subset $Y(\kappa) \subseteq $ $X(\kappa)$ for some closed $A$-subscheme $Y \subseteq X$. We define a topology on $Z$ as follows: The functor $T \mapsto X(A \otimes_{\kappa} T)$ is representable by an affine $\kappa$-scheme $\Res_{\kappa}^A X$ with $(\Res_{A/\kappa} X)(\kappa) = X(A)$, and the projection $X(A) \to X(\kappa)$ comes from a morphism of $\kappa$-schemes $\Res_{A/\kappa} X \to X_{\kappa}$. We then obtain $Z$ as the $\kappa$-points of the scheme-theoretic preimage of $Y_{\kappa}$ in $\Res_{A/\kappa} X$.

\section{Generalized reductive group schemes}
\label{secgrpsch}

\subsection{Definitions}
\label{secgengrpsch}

Working with deformations of representations valued in other algebraic groups $G$ than $\GL_n$, we have to decide which groups we want to allow for $G$.
Our group $G$ shall be defined over the coefficient ring of some deformation problem, such as the ring of integers of a $p$-adic local field.
The following definition of reductive and semisimple group schemes over arbitrary base schemes follows \cite[XIX, 2.7]{MR2867622} and \cite[Definition 3.1.1]{bcnrd}.

\begin{definition} A \emph{reductive (semisimple) group scheme} over a scheme $S$ is a smooth $S$-affine $S$-group scheme $G$, such that the geometric fibers of $G$ are connected reductive (semisimple) groups.
\end{definition}

When $S$ is the spectrum of a field, we recover the classical notion of connected reductive group.

An \emph{$S$-torus} is an $S$-group scheme of multiplicative type with smooth connected fibers, \cite[Definition 3.1.1]{bcnrd}.
Equivalently an $S$-torus is an $S$-group scheme which becomes étale-locally isomorphic to a power of $\Gm$, \cite[Proposition B.3.4]{bcnrd}.
If $G$ is a reductive $S$-group scheme, then a \emph{maximal torus} of $G$ is an $S$-torus $T \subseteq G$, such that for each geometric point $\overline s$ of $S$, $T_{\overline s}$ is a maximal torus of $G_{\overline s}$.
Étale-locally $G$ admits a maximal torus, \cite[Corollary 3.2.7]{bcnrd}.
For the slightly technical definition of a \emph{split reductive group} over $S$ we refer to \cite[Definition 5.1.1]{bcnrd}.
If $S = \Spec(\bbZ)$ and $G$ admits a maximal torus, then $G$ is split \cite[Example 5.1.4]{bcnrd}.

The following three sets are canonically in bijection, \cite[Theorem 1.4]{Conrad2014NONSPLITRG}.
\begin{enumerate}
    \item Split reductive group schemes over $\bbZ$ up to isomorphism of $\bbZ$-group schemes.
    \item Split connected reductive groups over $\bbQ$ up to isomorphism of $\bbQ$-group schemes.
    \item Root data up to isomorphism.
\end{enumerate}
We will refer to split reductive group schemes over $\bbZ$ as \emph{Chevalley groups}. By a \emph{Chevalley group} over a different base than $\bbZ$ we mean the base change of a Chevalley group over $\bbZ$.

Every split connected reductive group $G$ over the fraction field $L$ of a domain $\calO$ admits a model over $\calO$, which is a Chevalley group, \cite[Theorem 1.2]{Conrad2014NONSPLITRG}. If $\calO$ is a principal ideal domain, then every $\calO$-model of $G$ is a Chevalley group, \cite[Proposition 1.3]{Conrad2014NONSPLITRG}.

\begin{example} We discuss the main examples we will consider.
\begin{enumerate}
    \item The symplectic group $\Sp_{2n}$ over $\bbZ$ is the scheme-theoretic automorphism group of the standard symplectic bilinear form on $\bbZ^{2n}$. It is a semisimple Chevalley group with almost-simple connected geometric fibers.
    \item The orthogonal group $\OO_{n}$ over $\bbZ[\tfrac{1}{2}]$ is the automorphism group of the standard symmetric bilinear form on $\bbZ^n$. It is a smooth affine $\bbZ[\tfrac{1}{2}]$-group scheme with non-connected almost-simple geometric fibers. The special orthogonal group $\SO_n$, which is defined as the kernel of the determinant homomorphism $\det : \OO_n \to \Gm$, is a reductive $\bbZ[\tfrac{1}{2}]$-group scheme.
\end{enumerate}
\end{example}

The definition of $G$-pseudocharacters (\Cref{LafPC}) also allows for $G$ to be disconnected.
So we will need to extend the notion of reductivity to group schemes with possibly disconnected geometric fibers.

Suppose $G$ is a smooth affine group scheme over a commutative ring $\calO$. There is a unique open subgroup scheme $G^0 \subseteq G$, such that $(G^0)_s \cong (G_s)^0$ for all $s \in \Spec(\calO)$, \cite[Corollaire 15.6.5]{PMIHES_1966__28__5_0}. We say that $G^0$ is the \emph{identity component} of $G$.
Each $G^0_s$ is geometrically connected \cite[Exercise 1.6.5]{bcnrd} and it follows that the formation of the identity component $(-)^0$ commutes with any base change.
In particular, if the geometric fibers $G_{\overline s}$ for $s \in \Spec(\calO)$ are reductive groups, then their identity components $G_{\overline s}^0$ are connected reductive groups, $G^0$ is an open and closed $\calO$-subgroup scheme of $G$ and the quotient $G/G^0$ exists as a separated étale $\calO$-group scheme of finite presentation, \cite[Proposition 3.1.3]{bcnrd}.
In general $G/G^0$ does not have to be finite, \cite[Example 3.1.4]{bcnrd}.

This leads to the following definition, which includes the orthogonal groups $\mathrm O_n$ when $2$ is invertible in $\calO$.

\begin{definition}\label{defgenredgpsch} A \emph{generalized reductive group scheme} over a commutative ring $\calO$ is a smooth affine group scheme, such that $G^0$ is a reductive group scheme and $G/G^0$ is finite.
\end{definition}

It follows from \cite[Proposition 3.1.3]{bcnrd}, that for a generalized reductive $\calO$-group scheme $G$ the quotient $G/G^0$ is étale. Sean Cotner has shown in \cite{cotner} that if $\calO$ is noetherian then a smooth affine $S$-group scheme $G$ is generalized reductive if and only if it is geometrically reductive in the sense of \cite[Definition 9.1.1]{alper}. The notion has also been used in \cite[Definition 2.1]{Friedlander1988} and \cite{DHKM}.

\subsection{Good filtrations}\label{subsecGF} Let $\calO$ be a commutative ring, let $G$ be an affine $\calO$-group scheme and let $V$ be an $\calO[G]$-comodule. We will speak of an \emph{algebraic representation} of $G$ or a \emph{$G$-module}. We denote by $V^G$ the $\calO$-submodule of \emph{fixed points} or \emph{algebraic invariants} of $V$ under $G$, see \cite[§I.2.10]{Jantzen2003}.
If $\calO'$ is an arbitrary commutative $\calO$-algebra, then the natural map $V^G \otimes_{\calO} \calO' \to (V \otimes_{\calO} \calO')^{G_{\calO'}}$ is not always an isomorphism. The next lemma gives sufficient conditions.

\begin{lemma}\label{critchange} Let $G$ be a flat affine group scheme over a Dedekind domain $\calO$, let $V$ be a $G$-module and let $\calO'$ be a commutative $\calO$-algebra. Assume, that $\calO'$ is $\calO$-flat or that $H^1(G,V) = 0$. Then the natural map $V^G \otimes_{\calO} \calO' \to (V \otimes_{\calO} \calO')^{G_{\calO'}}$ is an isomorphism.
\end{lemma}

\begin{proof} By the universal coefficient theorem \cite[I.4.18 Proposition (a)]{Jantzen2003}, there is a short exact sequence
$$ 0 \to V^G \otimes_{\calO} \calO' \to (V \otimes_{\calO} \calO')^{G_{\calO'}} \to \Tor_1^{\calO}(H^1(G, V), \calO') \to 0 $$
Under either of the assumptions the claim follows.
\end{proof}

If $G$ is a Chevalley group over a principal ideal domain $\calO$ with fiber-wise maximal $\bbZ$-torus $T$ and Borel subgroup $B$, we define $H^0(\lambda) := \ind_B^G \lambda$ and $V(\lambda) := H^0(-w_0\lambda)^*$ for every dominant weight $\lambda \in X(T)_+$ and the longest element $w_0$ of the Weyl group. The $G$-module $V(\lambda)$ is free of finite rank over $\calO$. We say that an ascending filtration $V = \bigcup_{n \geq 0} V_n$ is \emph{good}, if for all $n \geq 0$ the module $V_{n+1}/V_n$ is isomorphic to $H^0(\lambda)$ for some $\lambda \in X(T)_+$. Since $V_n$ has finite rank, we have $H^i(G, V) = \Ext^i_G(V(0), V) = 0$ for all $i > 0$ by \cite[B.9 Lemma]{Jantzen2003}.

\subsection{Mathieu's tensor product theorem}
\label{subsecMathieu}

Mathieu's tensor product theorem states, that the tensor product of two modules with good filtration over a connected reductive group over an algebraically closed field admits a good filtration. We prove an integral version of this theorem.
Recall the universal coefficient theorem for $\Ext$ groups.

\begin{proposition}\label{UCTExt} Let $G$ be a flat affine group scheme over a Dedekind domain $\calO$ and let $\calO'$ be a commutative $\calO$-algebra. Then for every $\calO$-flat $G$-module $N$ and every finitely generated projective $G$-module $V$, we have a short exact sequence
$$ 0 \to \Ext_G^n(V,N) \otimes_{\calO} \calO' \to \Ext_{G_{\calO'}}^n(V \otimes_{\calO} \calO', N \otimes_{\calO} \calO') \to \Tor_1^{\calO}(\Ext_G^{n+1}(V,N),\calO') \to 0 $$
of $\calO'$-modules.
\end{proposition}

\begin{proof} By \cite[I.4.4 Lemma]{Jantzen2003} and \cite[I.4.2]{Jantzen2003}, there is a natural identification $\Ext_G^n(V,N) = \Ext_G^n(\calO,V^* \otimes_{\calO} N) = H^n(G,V^* \otimes_{\calO} N)$ and similarly for the middle term. The claim follows from the universal coefficient theorem \cite[I.4.18 Proposition (a)]{Jantzen2003}.
\end{proof}

\begin{theorem}\label{MathieuTPT} Let $G$ be a Chevalley group over a principal ideal domain $\calO$.
Let $M$ and $N$ be $G$-modules with good filtration. Then $M \otimes_{\calO} N$ is a $G$-module with good filtration.
\end{theorem}

\begin{proof} Choose a split and fiber-wise maximal $\calO$-torus $T \subseteq G$ and a Borel subgroup $B \subseteq G$ containing $T$. We first assume, that $M$ and $N$ are free of finite rank. By \cite[B.9 Lemma]{Jantzen2003} and \Cref{UCTExt} $M$ has a good filtration if and only if for any maximal ideal $\frakm$ of $\calO$ with residue field $\kappa := \calO/\frakm$ the $G_{\overline\kappa}$-module $M_{\overline\kappa} := M \otimes_{\calO} \overline\kappa$ has a good filtration. So by Mathieu's tensor product theorem \cite{MathieuTensorProductTheorem}, which holds for connected reductive groups over algebraically closed fields, see \cite[Proposition II.4.21]{Jantzen2003} or \cite[Theorem 4.4.3]{vanderKallen}, $M_{\overline\kappa} \otimes_{\overline\kappa} N_{\overline\kappa}$ has a good filtration. We conclude, that $M \otimes_{\calO} N$ is a $G$-module with good filtration. Now let $M$ and $N$ be arbitrary with good filtrations $M = \bigcup_{i=1}^{\infty} M_i$ and $N = \bigcup_{j=1}^{\infty} N_j$. Then $M \otimes_{\calO} N = \bigcup_i \bigcup_j M_i \otimes_{\calO} N_i$ by \stackcite{00DD}. Choosing a diagonal sequence, we obtain a good filtration of $M \otimes_{\calO} N$.
\end{proof}

\subsection{\texorpdfstring{$\calO[G^m]$}{O[G^m]}}
\label{secOGm}

We generalize \cite[Corollary VIII.5.6]{fargues2021geometrization} to principal ideal domains. We denote by $\FG(n)$ a free group on $n$ generators.

\begin{proposition}\label{scholzePID}
    Let $G$ be a Chevalley group over a principal ideal domain $\calO$.
    For any homomorphism $\FG(n) \to \Aut(G)$, the $G$-module $\calO[Z^1(\FG(n), G)]$ admits a good filtration.
\end{proposition}

\begin{proof}
    As a scheme $Z^1(\FG(n), G)$ is $G^n$ with an action by twisted conjugation, i.e. $g \cdot (h_1, \dots, h_n) = (\varphi_1(g)h_1g^{-1}, \dots, \varphi_n(g)h_ng^{-1})$ for $\varphi_1, \dots, \varphi_n \in \Aut(G)$ determined by the map $\FG(n) \to \Aut(G)$. The coordinate ring $\calO[G^n]$ has a good filtration as a $G^{2n}$-module with action given by left- and inverse right multiplication, see \cite[Lemma A.15]{Jantzen2003}.
    Restriction along the group automorphism $(g_1, \dots, g_{2n}) \mapsto (\varphi_1(g_1), g_2, \dots, \varphi_n(g_{2n-1}), g_{2n})$ preserves induced modules and hence good filtrations, so $\calO[Z^1(\FG(n), G)]$ admits a good filtration as a $G^{2n}$-module.
    The module remains good after restricting to the diagonal subgroup $G \hookrightarrow G^{2n}$, since induced modules $\ind_{B^{2m}}^{G^{2m}} \lambda$ are tensor products of induced modules of the individual factors and hence have a good filtration by \Cref{MathieuTPT}.
\end{proof}

In the connected case the following Corollary is proved in \cite[Lemma 2.7]{emerson2023comparison}.

\begin{corollary}\label{gfG}
    Let $G$ be a generalized reductive group over a Dedekind domain $\calO$.
    Then for any $\calO$-algebra $\calO'$ and any $\calO'$-algebra $\calO''$, the canonical map $\calO'[G^m]^{G^0} \otimes_{\calO'} \calO'' \to \calO''[G^m]^{G^0}$ is an isomorphism.
    If $\calO$ is a principal ideal domain, $G^0$ is split, $G/G^0$ is constant and there exists a scheme-theoretic splitting $G/G^0 \to G$, then $\calO[G^m]$ equipped with an action induced by the action of $G^0$ on $G^m$ by diagonal conjugation has a good filtration.
\end{corollary}

\begin{proof}
    For the first part we are by \Cref{critchange} allowed to replace $\calO$ by a faithfully flat $\calO$-algebra. 
    We will do this in a first step in such a way that $G^0$ becomes split and the algebra is a Dedekind domain.
    Recall that $G^0$ is a reductive group scheme.
    By \cite[Lemma 5.1.3]{bcnrd} there is a faithfully étale $\calO$-algebra $A$ over which $G^0$ becomes split.
    Let $K$ be the fraction field of $\calO$ and fix an algebraic closure $\overline x : \calO \hookrightarrow \overline K$.
    Then $K \otimes_{\calO} A$ is a finite-dimensional $K$-vector space and the map $A \hookrightarrow K \otimes_{\calO} A$ is injective.
    Every idempotent of $A$ gives an idempotent of $K \otimes_{\calO} A$, so there are only finitely many orthogonal idempotents $e_1, \dots, e_n$ in $A$.
    Given an idempotent $e_i \in A$, the ring $A_i := e_iA$ is an integral domain. Note, that $G^0$ becomes split over $A_i$.
    Let $S^{-1}\calO$ be the Dedekind domain associated with the image of $\Spec(A_i) \to \Spec(\calO)$.
    The fraction field $L_i$ of $A_i$ is a finite separable extension of $K$, hence the integral closure $\widetilde A_i$ of $A_i$ in $L_i$ is a finitely generated $S^{-1}\calO$-module and a Dedekind domain. We can apply \Cref{critchange} to the map $\calO \to \prod_i \widetilde A_i$ and assume from now on that $G^0$ becomes split over $\calO$.
    
    All $\calO$-homomorphisms $\calO[G/G^0] \to \overline K$ land in a common integrally closed integral extension of $\calO$ over which $G/G^0$ is a constant scheme.
    Indeed, we can see this by considering $G/G^0$ as finite $\pi_1^{\et}(\Spec(\calO), \overline x)$-set using \stackcite{0BND} and restricting to an open subgroup.
    We will assume, that $G/G^0$ is constant by further extending $\calO$.

    Let $C$ be a connected component of $G$. Then $C$ is a $G^0$-torsor. Since $C$ is of finite type over $\calO$, there is a homomorphism $\calO[C] \to \overline K$, which takes values in a suitable integrally closed finite extension of $\calO$. It follows, that $C$ is split over such an extension. By localization we may assume that $\calO$ is a principal ideal domain. By \Cref{critchange} the first claim reduces to the second claim.

    We now prove the second claim.
    By \cite[Proposition 1.3]{Conrad2014NONSPLITRG} $G^0$ is a Chevalley group.
    We have a disjoint decomposition $G = \bigsqcup_r rG^0$, where $r$ varies over a set of coset representatives of $G/G^0$. Conjugation by $r$ preserves $G^0$, so $\theta_r(g) := r^{-1}gr$ defines an automorphism $\theta_r : G^0 \to G^0$.
    The natural conjugation action of $G^0$ on $G$ preserves $rG^0$, as $g^{-1} r x g = r (r^{-1} g^{-1} r) x g = r \theta_r(g)^{-1} x g$ for generic $g, x \in G^0$.
    So $\calO[G] = \bigoplus_r \calO[rG^0] = \calO[Z^1(\FG(|G/G^0|), G^0)]$ as a $G^0$-module.
    The second claim now follows from \Cref{scholzePID}.
\end{proof}

\begin{remark} It also follows from \Cref{gfG}, that $\calO[G^m]^G \otimes_{\calO} \calO' = \calO'[G^m]^G$ for any $\calO$-algebra $\calO'$. In fact, a coset representative $g$ of $G/G^0$ acts trivially on $\calO[rG^0]^{G^0}$ if $[g] \in C_{G/G^0}([r])$, so the $G/G^0$-orbit of $\calO[rG^0]^{G^0}$ is an induced representation and it follows that $\calO[G^m]^{G^0}$ is an induced representation of $G/G^0$.
\end{remark}

\subsection{$G$-semisimplification}\label{subsecGss} We introduce a notion of $G$-semisimplification of homomorphisms $\Gamma \to G(k)$ for an algebraically closed field $k$ over $\calO$.
We use the dynamic definition of parabolic and Levi subgroups from \cite[§2]{Richardson1988ConjugacyCO}, which also applies to disconnected $G$.
In \cite[§3.2.1]{SerreCompleteReducibility} Serre introduced the notion of $G$-complete reducibility in the connected case.
We extend this definition to generalized reductive group schemes.

\begin{definition} A subgroup $H$ of $G(k)$ is \emph{$G$-completely reducible} if for every parabolic $P \subseteq G$ with $H \subseteq P(k)$, there exists a Levi subgroup $L \subseteq P$ with $H \subseteq L(k)$. We say that a homomorphism $\Gamma \to G(k)$ is \emph{$G$-completely reducible}, if its image is $G$-completely reducible.
\end{definition}

\begin{lemma}\label{PcapQcontainsmaxT} Let $P$ and $Q$ be parabolic subgroups of $G$. Then $P \cap Q$ contains a maximal torus of $G$.
\end{lemma}

\begin{proof} By \cite[Rmk. 5.3]{MartinGeneratingTuples}, $P^0$ and $Q^0$ are parabolic subgroups of $G^0$. By \cite[§2.4]{Borel1965} $P^0 \cap Q^0$ contains a maximal torus $T$ of $G^0$. So $T$ is also a maximal torus of $G$.
\end{proof}

\begin{lemma}\label{commonLeviLemma} Let $P$ and $Q$ be parabolic subgroups of $G$. Assume, that $Q$ contains a Levi of $P$ and $P$ contains a Levi of $Q$. Then $P$ and $Q$ have a common Levi.
\end{lemma}

\begin{proof} Let $T$ be a maximal torus of $G$ contained in $P \cap Q$ (\Cref{PcapQcontainsmaxT}). Let $L$ be a Levi of $P$, which contains $T$ and let $M$ be a Levi of $Q$, which contains $T$. Existence and uniqueness of $L$ and $M$ follow from \cite[Cor. 6.5]{BMR}. Since by assumption $P \cap Q$ contains a Levi of $P$ as well as a Levi of $Q$, the maps $P \cap Q \to P/R_u(P)$ and $P \cap Q \to Q/R_u(Q)$ are surjective. Since by \cite[Lem. 6.2 (iii)]{BMR}, $P \cap Q = (L \cap M) R_u(P \cap Q)$, we obtain surjections $L \cap M \to P/R_u(P)$ and $L \cap M \to Q/R_u(Q)$.
Hence $P = (L \cap M)R_u(P)$. Since $P = L \cdot R_u(P)$ and $L \cap R_u(P) = 1$, we have $L \cap M = L$. Similarly, we have $L \cap M = M$.
\end{proof}

\begin{lemma}\label{commonRLevi} Let $H$ be a closed subgroup of $G$. Let $P$ and $Q$ be parabolic subgroups of $G$, both minimal among parabolic subgroups containing $H$. Then $P$ and $Q$ have a common Levi.
\end{lemma}

\begin{proof} The group $(P \cap Q)R_u(Q)$ contains $H$ and is contained in $Q$ and by \cite[Cor. 6.9]{BMR}, $(P \cap Q)R_u(Q)$ is parabolic. By minimality of $Q$, we have $Q = (P \cap Q)R_u(P)$. Again by \cite[Cor. 6.9]{BMR}, $P$ contains a Levi subgroup of $Q$. Similarly $Q$ contains a Levi of $P$. We can apply \Cref{commonLeviLemma}.
\end{proof}

\begin{definition}\label{defGss} Let $\rho : \Gamma \to G(k)$ be a homomorphism. Let $P$ be a parabolic of $G$, such that $\rho(\Gamma) \subseteq P(k)$ and such that $P$ is minimal among all parabolic subgroups with this property. Let $L$ be a Levi of $P$. We have a canonical surjective homomorphism $c_{P,L} : P \to L$. We define the \emph{$G$-semisimplification} $\rho^{\semi}$ of $\rho$ with respect to $P$ and $L$ as the composition $\Gamma \overset{\rho}{\to} P(k) \overset{c_{P,L}}{\to} L(k) \to G(k)$.
\end{definition}

If $G=\GL_n$ we recover the usual notion of semisimplification, which is defined as the direct sum of the Jordan-Hölder factors of $\rho$.
By definition $\rho^{\semi}(\Gamma)$ is a $G$-semisimplification of the subgroup $\rho(\Gamma)$ in the sense of \cite[Definition 4.1]{Bate_2020}.
It is immediate, that $\rho^{\semi}$ is $G$-completely reducible.
If $\Gamma$ is a topological group, $k$ is a topological field and $\rho$ is continuous, then $\rho^{\semi}$ is continuous, but the converse is false in general.

\begin{proposition}\label{indepofPL} The $G^0(k)$-conjugacy class of $\rho^{\semi}$ is independent of the choice of $P$ and $L$.
\end{proposition}

\begin{proof} Let $\rho^{\semi, i} = c_{P_i, L_i} \circ \rho$ for $i=1,2$ be two semisimplifications of $\rho$ with respect to a parabolic $P_i$ and a Levi $L_i$ respectively. We first assume $P := P_1 = P_2$. Then there exists some $u \in R_u(P)$, such that $uL_1u^{-1} = L_2$. Since the square in the following diagram commutes, we obtain $uc_{P,L_1}u^{-1} = c_{P,L_2}$ and thus in particular $u\rho^{\semi, 1}u^{-1} = \rho^{\semi, 2}$.
\begin{center}
    \begin{tikzcd}[row sep=2.5em, column sep=2.5em]
        P \arrow[r] \arrow[dr] \arrow[rr, bend left=35, "c_{L_1}"] \arrow[drr, bend right=60, "c_{L_2}"'] & P/R_u(P) \arrow[d, equal] & L_1 \arrow[l, "\sim"'] \arrow[d, "u(-)u^{-1}"] \\
         & P/R_u(P) & L_2 \arrow[l, "\sim"']
    \end{tikzcd}
\end{center}
If $P_1 \neq P_2$, we can apply \Cref{commonRLevi} to find a common Levi $L$ of $P_1$ and $P_2$.
By \cite[Lemma 6.2 (iii)]{BMR}, we have $P_1 \cap P_2 = L \cdot (R_u(P_1) \cap R_u(P_2))$.
It follows, that the following diagram commutes.
\begin{center}
    \begin{tikzcd}
        P_1 \cap P_2 \arrow[r] \arrow[d] & P_1 \arrow[d, "c_{P_1,L}"] \\
        P_2 \arrow[r, "c_{P_2,L}"'] & L
    \end{tikzcd}
\end{center}
This implies $c_{P_1, L} \circ \rho = c_{P_2, L} \circ \rho$.
We obtain from the first step, that there are $u_i \in R_u(P_i)$ with $u_i(c_{P_i,L_i} \circ \rho) u_i^{-1} = c_{P_i, L} \circ \rho$.
\end{proof}

\section{$G$-pseudocharacters}
\label{secGPC}

Fix a commutative ring $\calO$, a generalized reductive $\calO$-group scheme $G$ and an arbitrary group $\Gamma$.
We are interested in the case, that $\calO$ is the ring of integers of a $p$-adic field.
By the datum of $G$, the datum of $\calO$ is given and we will drop $\calO$ from notations.
A $G$-pseudocharacter will be defined depending a priori on both the coefficient ring $\calO$ and a commutative $\calO$-algebra $A$, which corresponds to the base ring $A$ in \Cref{secdeterminants}.
We will later be able to remove the dependence on $\calO$ (see \Cref{basechange}).

\subsection{Definition}\label{subsecLafPC} The definition of $G$-pseudocharacter we use is slightly more general than Lafforgue's original definition \cite[§11]{Laf}, in that we work over arbitrary base rings $\calO$.
We introduce special notation for substitutions, which will be particularly important in \Cref{LafPC} and the proofs of \Cref{repofPC} and \Cref{decisivefiniteness}.

Let $\FG(m)$ be the free group on $m$ generators $x_1, \dots, x_m$.
We will use the same letters $x_1, x_2, \dots$ to denote the generators of $\FG(m)$ for every $m \geq 0$.
Let $\Gamma$ be an arbitrary group. To a tuple $\gamma = (\gamma_1, \dots, \gamma_m) \in \Gamma^m$ corresponds a unique homomorphism $f_{\gamma} : \FG(m) \to \Gamma$, so we have a bijection $\Gamma^m \cong \Hom(\FG(m), \Gamma), ~\gamma \mapsto f_{\gamma}$.

Let $\alpha : \FG(m) \to \FG(n)$ be a group homomorphism.
We define $(-)_{\alpha}$ to be the induced map $\Gamma^n \cong \Hom(\FG(n), \Gamma) \to \Hom(\FG(m), \Gamma) \cong \Gamma^m$.
For a tuple $\delta = (\delta_1, \dots, \delta_n) \in \Gamma^n$ the homomorphism $f_{\delta} : \FG(n) \to \Gamma, ~x_i \mapsto \delta_i$ satisfies $f_{\delta}(\alpha(x_j)) = (\delta_{\alpha})_j$ for all $j \in \{1, \dots, m\}$.

Similarly we obtain an induced map $(-)_{\alpha} : G^n \to G^m$.
The group $G^0$ acts on $G^m$ by $g \cdot (g_1, \dots, g_m) = (gg_1g^{-1}, \dots, gg_mg^{-1})$. 
This induces an algebraic representation of $G^0$ on the affine coordinate ring $\calO[G^m]$ of $G^m$. 
The submodule $\calO[G^m]^{G^0} \subseteq \calO[G^m]$ is defined as the module of algebraic invariants of the $G^0$-representation $\calO[G^m]$.
It is an $\calO$-subalgebra, since $G^0$ acts by $\calO$-algebra automorphisms. 
The map $(-)_{\alpha} : G^n \to G^m$ is $G^0$-equivariant, so there is an induced homomorphism $(-)^{\alpha} : \calO[G^m]^{G^0} \to \calO[G^n]^{G^0}$.
In the special case that $\alpha$ is induced by a map of sets $\zeta : \{1, \dots, n\} \to \{1, \dots, m\}$, such that $\alpha(x_i) = x_{\zeta(i)}$, we also write $\delta_{\zeta} := \delta_{\alpha}$ for $\delta \in \Gamma^n$ and $f^{\zeta} := f^{\alpha}$ for $f \in \calO[G^m]^{G^0}$.

\begin{definition}[$G$-pseudocharacter]\label{LafPC} Let $A$ be a commutative $\calO$-algebra. A \emph{$G$-pseudocharacter} $\Theta$ of $\Gamma$ over $A$ is a sequence of $\calO$-algebra maps 
$$\Theta_m : \calO[G^m]^{G^0} \to \map(\Gamma^m,A)$$ 
for each $m \geq 1$, satisfying the following conditions:
\begin{enumerate}
    \item For all $n,m \geq 1$, each map $\zeta : \{1, \dots, m\} \to \{1, \dots,n\}$, every $f \in \calO[G^m]^{G^0}$ and all $\gamma_1, \dots, \gamma_n \in \Gamma$, we have
    $$ \Theta_n(f^{\zeta})(\gamma_1, \dots, \gamma_n) = \Theta_m(f)(\gamma_{\zeta(1)}, \dots, \gamma_{\zeta(m)}) $$
    where $f^{\zeta}(g_1, \dots, g_n) = f(g_{\zeta(1)}, \dots, g_{\zeta(m)})$. \label{subst_1}
    \item For all $m \geq 1$, for all $\gamma_1, \dots, \gamma_{m+1} \in \Gamma$ and every $f \in \calO[G^m]^{G^0}$, we have
    $$ \Theta_{m+1}(\hat f)(\gamma_1, \dots, \gamma_{m+1}) = \Theta_m(f)(\gamma_1, \dots, \gamma_m\gamma_{m+1}) $$
    where $\hat f(g_1, \dots, g_{m+1}) = f(g_1, \dots, g_mg_{m+1})$. \label{subst_2}
\end{enumerate}
We denote the set of $G$-pseudocharacters of $\Gamma$ over $A$ by $\PC_{G}^{\Gamma}(A)$.
When $\Gamma$ and $A$ are equipped with a topology, we say that $\Theta$ is \emph{continuous}, if $\Theta_m$ takes values in the subset $\calC(\Gamma^m,A) \subseteq \map(\Gamma^m,A)$ of continuous maps for all $m \geq 1$. We write $\cPC_{G}^{\Gamma}(A)$ for the subset of continuous $G$-pseudocharacters.
\end{definition}

If $f : A \to B$ is a homomorphism of $\calO$-algebras, then there is an induced map $- \otimes_A B : \PC_{G}^{\Gamma}(A) \to \PC_{G}^{\Gamma}(B)$.
For $\Theta \in \PC_{G}^{\Gamma}(A)$, the image $\Theta \otimes_A B$ is called the \emph{scalar extension} of $\Theta$.
This notion of scalar extension shall not be confused with change of the base ring $\calO$ of $G$, which will be addressed in \Cref{basechange}.

If $H$ is another generalized reductive $\calO$-group scheme and $\iota : G \to H$ is a homomorphism, we define an $H$-pseudocharacter $\iota(\Theta)$ by letting $\iota(\Theta)_m$ be the composition of $\Theta_m$ with the induced map $\calO[H^m]^{H^0} \to \calO[G^m]^{G^0}$. In \cite[Definition 4.1]{BHKT} a $G$-pseudocharacter is defined only for Chevalley groups over $\bbZ$.

If $\Delta \to \Gamma$ is a group homomorphism the \emph{restriction} $\Theta|_{\Delta}$ of $\Theta$ to $\Delta$ is defined by letting $(\Theta|_{\Delta})_m$ be the composition of $\Theta_m$ with induced map $\map(\Gamma^m, A) \to \map(\Delta^m, A)$.

A representation $\rho : \Gamma \to G(A)$ gives rise to a $G$-pseudocharacter $\Theta_{\rho}$, which depends only on $\rho$ up to $G^0(A)$-conjugation. Here $(\Theta_{\rho})_m : \calO[G^m]^{G^0} \to \map(\Gamma^m, A)$ is defined by 
$$ (\Theta_{\rho})_m(f)(\gamma_1, \dots, \gamma_m) := f(\rho(\gamma_1), \dots, \rho(\gamma_m)) $$
We let $\Rep^{\Gamma, \square}_G(A) := \Hom(\Gamma, G(A))$ and thereby get a natural map $\Rep^{\Gamma, \square}_G(A) \to \PC_G^{\Gamma}(A)$.
When $A$ is a topological ring and $\rho : \Gamma \to G(A)$ is a continuous homomorphism with $G(A)$ topologized as in \cite[Proposition 2.1]{ConradTopologies}, then $\Theta_{\rho}$ is a continuous $G$-pseudocharacter. We write $\cRep^{\Gamma, \square}_G(A) \subseteq \Rep^{\Gamma, \square}_G(A)$ for the subset of continuous representations and we have a natural map $\cRep^{\Gamma, \square}_G(A) \to \cPC_G^{\Gamma}(A)$.

The following density principle is the analog of \cite[Ex. 2.31]{MR3444227} and is needed in the proof of \Cref{tfgfingen}.

\begin{lemma}\label{restrictiontoadensesubgroup} Let $\Gamma$ be a topological group and $\Delta \subseteq \Gamma$ a dense subgroup. Then for all Hausdorff $\calO$-algebras $A$ the restriction $$\cPC^{\Gamma}_G(A) \to \cPC^{\Delta}_G(A)$$ defined by composition with $\calC(\Gamma^n,A) \to \calC(\Delta^n,A)$ is injective.
\end{lemma}

\begin{proof} Let $\Theta, \Theta' \in \cPC^{\Gamma}_G(A)$ be such that $\Theta|_{\Delta} = \Theta'|_{\Delta}$. Let $n \geq 0$ and $f \in \calO[G^n]^{G^0}$. Then $\Theta_n(f), \Theta_n'(f) : \Gamma^n \to A$ are continuous maps, that agree on the dense subset $\Delta^n \subseteq \Gamma^n$, hence must be equal.
\end{proof}

\subsection{Reconstruction theorem}
\label{secRecthm}

The goal of this section is to show that when $\calO$ is noetherian a $G$-pseudocharacter over an algebraically closed field comes from an essentially unique $G$-completely reducible representation. We will refer to this result as the \emph{reconstruction theorem} for $G$-pseudocharacters.
We start with the observation that passage from a $G$-valued representation to the associated $G$-pseudocharacter is insensitive to $G$-semisimplification.

\begin{proposition}\label{sshassamePC} Let $\rho : \Gamma \to G(k)$ be a representation over an algebraically closed field $k$ and let $\rho^{\semi}$ be some $G$-semisimplification of $\rho$. Then $\Theta_{\rho} = \Theta_{\rho^{\semi}}$ in $\PC_G^{\Gamma}(k)$.
\end{proposition}

\begin{proof}
Suppose $P$ is a minimal parabolic and contains $\rho(\Gamma)$, suppose $L$ is a Levi of $P$ and let $\lambda : \Gm \to G^0$ be a cocharacter, such that $P = P_{\lambda}$, $L = L_{\lambda}$ and $\rho^{\semi} = c_{P_{\lambda}, L_{\lambda}} \circ \rho = \lim\nolimits_{t \to 0} \lambda(t) \rho \lambda(t)^{-1}$.
The $G$-pseudocharacter $\Theta_{m, \rho}$ attached to $\rho$ satisfies by definition $\Theta_{\rho, m}(f)(\gamma_1, \dots, \gamma_m) = f(\rho(\gamma_1), \dots, \rho(\gamma_m))$ for all $m \geq 1$ and $\rho^{\semi}$ satisfies the analogous formula. Since $f$ is $G$-invariant, the morphism $\mathbb G_m \to \mathbb A^1, ~t \mapsto f(\lambda(t) \rho(\gamma_1) \lambda(t)^{-1}, \dots, \lambda(t) \rho(\gamma_m) \lambda(t)^{-1})$ is constant and equal to $f(\rho(\gamma_1), \dots, \rho(\gamma_m))$. Since the limit $\lim\nolimits_{t \to 0} \lambda(t) \rho \lambda(t)^{-1}$ exists and $f$ is algebraic with separated target $\mathbb A^1$, this is equal to $f(\rho^{\semi}(\gamma_1), \dots, \rho^{\semi}(\gamma_m))$ and $\Theta_{\rho} = \Theta_{\rho^{\semi}}$ follows.
\end{proof}

\begin{lemma}\label{charorbitcr} Let $G$ be a reductive group over an algebraically closed field $k$. Let $g = (g_1, \dots, g_n) \in G^n(k)$ and let $H$ be the smallest Zariski closed subgroup of $G(k)$, containing $\{g_1, \dots, g_n\}$. The following are equivalent:
\begin{enumerate}
    \item The $G^0(k)$-orbit of $g$ is closed.
    \item The $G(k)$-orbit of $g$ is closed.
    \item $H$ is strongly reductive in $G$ in the sense of \cite[Definition 16.1]{Richardson1988ConjugacyCO}.
    \item $H$ is $G$-completely reducible.
\end{enumerate}
\end{lemma}

\begin{proof} Let $x_1, \dots, x_r \in G(k)$ be coset representatives of $G(k)/G^0(k)$. So
$$ G(k) \cdot g = \bigcup_{i=1}^r G^0(k) \cdot x_ig = \bigcup_{i=1}^r x_i \cdot (G^0(k) \cdot g). $$
If $G^0(k) \cdot g$ is closed, then all $x_i \cdot (G^0(k) \cdot g)$ are images of $G^0(k) \cdot g$ under multiplication with $x_i$ and therefore also closed. It follows, that (1) implies (2). If $G(k) \cdot g$ is closed, then it contains a closed $G^0(k)$-orbit, which is necessarily of the form $G^0(k) \cdot x_ig$. But then again $G^0(k) \cdot g$ is closed, so (2) implies (1).
The equivalence of (2) and (3) is \cite[Theorem 16.4]{Richardson1988ConjugacyCO}. The equivalence of (3) and (4) is \cite[Theorem 3.1]{BMR}.
\end{proof}

\begin{lemma}\label{pointsinGnG0} Let $G$ be a generalized reductive group scheme over a noetherian commutative ring $\calO$. Let $k$ be an algebraically closed field over $\calO$. Then there is a bijection between the following sets induced by $\pi : G^n(k) \to (G^n \sslash G^0)(k)$.
\begin{enumerate}
    \item $(G^n \sslash G^0)(k)$
    \item $G^0(k)$-conjugacy classes of tuples $(g_1, \dots, g_n) \in G^n(k)$, such that the smallest Zariski closed subgroup of $G(k)$ that contains $\{g_1, \dots, g_n\}$ is $G$-completely reducible.
\end{enumerate}
\end{lemma}

\begin{proof} Recall, that $G^0$ is a reductive group scheme. By \cite[Theorem 3]{Seshadri}, the map $\pi : G^n(k) \to (G^n \sslash G^0)(k)$ is surjective. By \cite[Proposition 3.2]{BHKT} for each $x \in (G^n \sslash G^0)(k)$, the fiber $\pi^{-1}(x)$ contains a unique closed $G^0(k)$-orbit. The claim follows from \Cref{charorbitcr}.
\end{proof}

\begin{lemma}\label{indepnumbers} Let $\Gamma \subseteq G(k)$ be a subgroup, where $G$ is a reductive group over an algebraically closed field $k$.
Let $P$ and $P'$ be parabolic subgroups of $G$ minimal among those which contain $\Gamma$. Then $\dim P = \dim P'$ and $|\pi_0(P)| = |\pi_0(P')|$.
\end{lemma}

\begin{proof} By \Cref{commonRLevi}, $P$ and $P'$ contain a common Levi $L$. We have $\dim P = \tfrac{1}{2}(\dim G + \dim L) = \dim P'$ and $|\pi_0(P)| = |\pi_0(L)| = |\pi_0(P')|$, since we have surjections $c_{P,L} : P \to L$ and  $c_{P',L} : P' \to L$ with connected kernel: apply \cite[Theorem 4.1.7]{bcnrd} to the conjugation action of the cocharacters defining $P$ and $P'$.
\end{proof}

The proof of the reconstruction theorem below is similar to the proof presented in \cite[Theorem 4.5]{BHKT} in the case that $G$ is split connected reductive over $\bbZ$. The main difference is that we prove the result also for groups $G$ with nontrivial component group $G/G^0$. The validity of \Cref{reconstructiongeneral} has been claimed for $G$ a semidirect product in the proof of \cite[Lemma A.4]{DHKM} without proof. A version of this theorem for $L$-parameters is also proved in \cite[Proposition VIII.3.8]{fargues2021geometrization}.

\begin{theorem}\label{reconstructiongeneral} Let $G$ be a generalized reductive $\calO$-group scheme, assume $\calO$ is noetherian, let $\Gamma$ be a group, let $k$ be an algebraically closed field over $\calO$ and let $\Theta \in \PC_G^{\Gamma}(k)$. Then there is a $G$-completely reducible representation $\rho : \Gamma \to G(k)$ with $\Theta_{\rho} = \Theta$, which is unique up to $G^0(k)$-conjugation.
\end{theorem}

\begin{proof} As we will work with tuples of elements of $\Gamma$ of varying length, we introduce the set $\scrT = \{(n,\gamma) \mid n \geq 1, \gamma \in \Gamma^n\}$.
For every tuple $(n,\gamma) \in \scrT$, the homomorphism $\Theta_n$ determines a homomorphism $\calO[G^n]^{G^0} \to k, ~f \mapsto \Theta_n(f)(\gamma)$, hence an element $\xi_{\gamma} \in (G^n \sslash G^0)(k)$ since $G^n \sslash G^0 = \Spec(\calO[G^n]^{G^0})$. The map $G^n(k) \to (G^n \sslash G^0)(k)$ is surjective by \Cref{pointsinGnG0} and we write $T(\gamma) \in G(k)^n$ for a representative of $\xi_{\gamma}$ contained in the unique closed $G^0(k)$-orbit in $G^n(k)$ over $\xi_{\gamma}$. The representative $T(\gamma)$ shall be chosen and fixed for each $(n,\gamma) \in \scrT$ for the rest of the proof.

Let $H(\gamma)$ be the smallest Zariski closed subgroup of $G(k)$, that contains the entries of $T(\gamma)$. By \Cref{charorbitcr}, since $T(\gamma)$ represents a closed orbit, $H(\gamma)$ is $G$-completely reducible. Let $n(\gamma)$ be the dimension of a parabolic subgroup $P$ of $G_k$ minimal among those with $H(\gamma) \subseteq P(k)$ and let $c(\gamma)$ be the order of the component group $\pi_0(P) = P/P^0$. By \Cref{indepnumbers} these numbers are both independent of the choice of $P$.

Let $N := \sup_{n \geq 1, \gamma \in \Gamma^n} n(\gamma)$ and $C := \sup_{n \geq 1, \gamma \in \Gamma^n; n(\gamma) = N} c(\gamma)$. 
For $(n, \delta) \in \scrT$ we consider the following four conditions:
\begin{enumerate}
    \item $n(\delta) = N$.\label{recone}
    \item $c(\delta) = C$.\label{rectwo}
    \item For any $n' \geq 1$ and $\delta' \in \Gamma^{n'}$ also satisfying \eqref{recone} and \eqref{rectwo}, we have $\dim Z_{G_k}(H(\delta)) \leq \dim Z_{G_k}(H(\delta'))$.\label{recthree}
    \item For any $n' \geq 1$ and $\delta' \in \Gamma^{n'}$ also satisfying \eqref{recone}, \eqref{rectwo} and \eqref{recthree}, we have $|\pi_0(Z_{G_k}(H(\delta)))| \leq |\pi_0(Z_{G_k}(H(\delta')))|$.\label{recfour}
\end{enumerate}

Let us denote by $\scrT_{1,2,3,4}$ the subset of $\scrT$ of pairs $(n,\gamma)$ that satisfy all conditions \eqref{recone}, \eqref{rectwo}, \eqref{recthree} and \eqref{recfour}. Similarly we define subsets $\scrT_{1}$, $\scrT_{1,2}$ and $\scrT_{1,2,3}$. We claim, that $\scrT_{1,2,3,4} \neq \emptyset$.

\underline{Proof of $\scrT_{1,2,3,4} \neq \emptyset$.} 
Condition \eqref{recone} can be established, since $N \leq \dim G_k$, so $\scrT_1 \neq \emptyset$. Condition \eqref{rectwo} can be established, since $G$ has only finitely many conjugacy classes of parabolic subgroups (\cite[Proposition 5.2 (e)]{MartinGeneratingTuples} and \cite[Corollary 6.7]{BMR}) and so $C$ is bounded by the maximal number of components of a parabolic subgroup of $G_k$. So $\scrT_{1,2} \neq \emptyset$. 
There is some $(n,\delta) \in \scrT_{1,2}$ with $\inf_{(n',\delta') \in \scrT_{1,2}} \dim Z_{G_k}(H(\delta')) = \dim Z_{G_k}(H(\delta))$. Hence $\scrT_{1,2,3} \neq \emptyset$. There is some $(n,\delta) \in \scrT_{1,2,3}$ with $\inf_{(n',\delta') \in \scrT_{1,2,3}} |\pi_0(Z_{G_k}(H(\delta')))| \leq |\pi_0(Z_{G_k}(H(\delta)))|$. Hence $\scrT_{1,2,3,4} \neq \emptyset$. \hfill $\Diamond$

We choose and fix $(n,\delta) \in \scrT_{1,2,3,4}$ for the rest of the proof. Let $(g_1, \dots, g_n) := T(\delta)$.

\underline{Claim A.} For all $\gamma \in \Gamma$, there is a unique $g \in G(k)$, such that $(g_1, \dots, g_n, g)$ is $G^0(k)$-conjugate to $T(\delta_1, \dots, \delta_n, \gamma)$.

\underline{Proof of existence of $g$.} Let $(h_1, \dots, h_n, h) := T(\delta_1, \dots, \delta_n, \gamma)$. It follows from substitution property \eqref{subst_1} in the definition of $G$-pseudocharacter with $\zeta : \{1, \dots, n\} \to \{1, \dots, n+1\}$ the inclusion, that $(h_1, \dots, h_n)$ lies over $\xi_{\delta} \in (G^n \sslash G^0)(k)$.

Let $P \subseteq G_k$ be a minimal parabolic among those with $H(\delta_1, \dots, \delta_n, \gamma) \subseteq P(k)$.
Since $H(\delta_1, \dots, \delta_n, \gamma)$ is $G$-completely reducible we find by definition of complete reducibility a Levi subgroup $M_P$ of $P$ with $H(\delta_1, \dots, \delta_n, \gamma) \subseteq M_P(k)$. Let $N_P := R_{\mathrm u}(P)$ be the unipotent radical of $P$ and let $Q \subseteq M_P$ be a parabolic subgroup of $M_P$ minimal among those containing $\{h_1, \dots, h_n\}$. Let $M_Q$ be a Levi subgroup of $Q$ and let $h_1', \dots, h_n' \in M_Q(k)$ be the images of $h_1, \dots, h_n$ in $M_Q(k)$ under the map $Q \to M_Q$ determined by the decomposition $Q = M_Q \ltimes R_{\mathrm u}(Q)$. Then the smallest Zariski closed subgroup of $M_Q$ containing $h_1', \dots, h_n'$ is $M_Q$-irreducible, as the preimage of a parabolic of $M_Q$ in $Q$ is a parabolic \cite[Lemma 6.2 (ii)]{BMR}. Therefore it is $G$-completely reducible by the non-connected version of \cite[Corollary 3.22]{BMR} as explained in \cite[§6.3]{BMR}. The tuples $(h_1, \dots, h_n)$ and $(h_1', \dots, h_n')$ map to the same element in $(M_P^n \sslash M_P^0)(k)$, as $(h_1', \dots, h_n')$ lies in the orbit closure of $(h_1, \dots, h_n)$. Therefore via the map $M_P^n \sslash M_P^0 \to G^n \sslash G^0$ the tuple $(h_1', \dots, h_n')$ maps to $\xi_{\delta}$.
By \Cref{pointsinGnG0}, $(h_1', \dots, h_n')$ is $G^0(k)$-conjugate to $T(\delta)$.

The subgroup $QN_P$ of $G_k$ is parabolic \cite[Lemma 6.2 (ii)]{BMR} and contains the conjugate $(h_1', \dots, h_n')$ of $T(\delta)$. So we have $N = n(\delta) \leq \dim QN_P \leq \dim P \leq N$.
The equality follows, since $\delta$ satisfies \eqref{recone}.
The first inequality follows, since $QN_P$ contains a parabolic minimal among those containing $(h_1', \dots, h_n')$.
The second inequality follows, since $QN_P \subseteq P$.
The third inequality follows by definition of $N$.
We deduce, that $\dim QN_P = \dim P$. 
Since $P = M_P \ltimes N_P$ and $Q \subseteq M_P$, we have $\dim Q = \dim M_P$, $Q^0 = M_P^0$ and $|\pi_0(Q)| \leq |\pi_0(M_P)|$.

We also have $C = c(\delta) \leq |\pi_0(QN_P)| \leq |\pi_0(P)| \leq C$.
The equality follows, since $\delta$ satisfies \eqref{rectwo}.
The first inequality follows, since any parabolic minimal among those containing $(h_1', \dots, h_n')$ has dimension $N = \dim QN_P$.
The second inequality follows from $|\pi_0(Q)| \leq |\pi_0(M_P)|$ and the semidirect product decomposition of $P$.
The third inequality holds, since $N = \dim P = n(\delta_1, \dots, \delta_n, \gamma)$ and $(\delta_1, \dots, \delta_n, \gamma)$ occurs in the supremum in the definition of $C$.

We conclude since $(QN_P)^0 = Q^0N_P = M_P^0N_P = P^0$ and $|\pi_0(QN_P)| = |\pi_0(P)|$ and both groups are smooth (\cite[Theorem 4.1.7 (4)]{bcnrd}), that $QN_P = P$, $Q = M_P$ and $h_i = h_i'$ for all $i = 1, \dots, n$.
So the $G^0(k)$-orbit of $(h_1, \dots, h_n)$ in $G^n(k)$ is closed.
By \Cref{pointsinGnG0}, there is some $x \in G^0(k)$, such that $x(h_1, \dots, h_n)x^{-1} = (g_1, \dots, g_n)$.
We can take $g := xhx^{-1}$ and the proof of existence is finished. \hfill $\Diamond$

\underline{Proof of uniqueness of $g$.} Fix $\gamma \in \Gamma$ and suppose, that $g, g' \in G(k)$ are such that $(g_1, \dots, g_n, g)$ and $(g_1, \dots, g_n, g')$ are $G^0(k)$-conjugate to $T(\delta_1, \dots, \delta_n, \gamma)$. In particular, there is some $y \in G^0(k)$, such that $y(g_1, \dots, g_n, g)y^{-1} = (g_1, \dots, g_n, g')$.
This means, that $y \in Z_{G(k)}(g_1, \dots, g_n) = Z_{G_k}(g_1, \dots, g_n)(k) = Z_{G_k}(H(\delta))(k)$ and our goal is to show that $y \in Z_{G(k)}(g_1, \dots, g_n, g)$, for then $g = g'$. There is an inclusion $Z_{G(k)}(g_1, \dots, g_n, g) \subseteq Z_{G(k)}(g_1, \dots, g_n)$.
Since $\delta$ satisfies \eqref{recone} and \eqref{rectwo}, $(\delta_1, \dots, \delta_n, \gamma)$ also satisfies \eqref{recone} and \eqref{rectwo}. It thus follows from property \eqref{recthree} of $\delta$, that $\dim Z_{G_k}(H(\delta)) = \dim Z_{G_k}(H(\delta_1, \dots, \delta_n, \gamma))$, so $(\delta_1, \dots, \delta_n, \gamma)$ also satisfies \eqref{recthree}. It follows from property \eqref{recfour} of $\delta$, that $|\pi_0(Z_{G_k}(H(\delta)))| = |\pi_0(Z_{G_k}(H(\delta_1, \dots, \delta_n, \gamma)))|$. Hence the groups $Z_{G_k}(H(\delta))$ and $Z_{G_k}(H(\delta_1, \dots, \delta_n, \gamma))$ agree on $k$-points and we conclude that $Z_{G(k)}(g_1, \dots, g_n, g) = Z_{G(k)}(g_1, \dots, g_n)$. \hfill $\Diamond$

So we have proved Claim A and can define a map $\rho : \Gamma \to G(k), ~\gamma \mapsto g$. We have to show, that $\rho$ is a homomorphism.

\underline{Claim B.} For all $\gamma, \gamma' \in \Gamma$, there are unique $g, g' \in G(k)$, such that $(g_1, \dots, g_n, g, g')$ is $G^0(k)$-conjugate to $T(\delta_1, \dots, \delta_n, \gamma, \gamma')$.

The proof of Claim B is similar to the proof of Claim A, compare \cite[Theorem 4.5]{BHKT} for more details.
In fact, we can take an arbitrary tuple in place of a single element for $\gamma$ in the proof of Claim A.

\underline{Claim C.} In the situation of Claim B, the $G^0(k)$-orbits of $(g_1, \dots, g_n, g)$, $(g_1, \dots, g_n, g')$ and $(g_1, \dots, g_n, gg')$ are closed in $G^{n+1}(k)$.

We only show, that the $G^0(k)$-orbit of $(g_1, \dots, g_n, gg')$ is closed in $G^{n+1}(k)$.
The argument for the other two orbits is similar.
Let $P$ be a parabolic minimal among those containing $\{g_1, \dots, g_n, g, g'\}$.
Then $P$ contains $\{g_1, \dots, g_n\}$ and $\dim P = N$ and $|\pi_0(P)| = C$, as before.
It follows, that $P$ is minimal among parabolic subgroups containing $\{g_1, \dots, g_n\}$.
Let $M_P$ be a Levi of $P$ containing $\{g_1, \dots, g_n, g, g'\}$; this exists since the orbit of $(g_1, \dots, g_n, g, g')$ is closed.
As before, the subgroup generated by $\{g_1, \dots, g_n\}$ is $M_P$-irreducible, hence $G$-completely reducible and the same is true for $\{g_1, \dots, g_n, gg'\}$.
It follows, that the $G^0(k)$-orbit of $(g_1, \dots, g_n, gg')$ is closed. \hfill $\Diamond$

By Claim B, Claim C and the substitution property \eqref{subst_1} in the definition of $G$-pseudocharacter applied to the inclusion $\{1, \dots, n+1\} \to \{1, \dots, n+2\}$ and the map $\{1, \dots, n+1\} \to \{1, \dots, n+2\}$ which is the inclusion for elements $\leq n$ and maps $n+1$ to $n+2$, we obtain that $(g_1, \dots, g_n, g)$ is $G^0(k)$-conjugate to $T(\delta_1, \dots, \delta_n, \gamma)$ and $(g_1, \dots, g_n, g')$ is $G^0(k)$-conjugate to $T(\delta_1, \dots, \delta_n, \gamma')$. By Claim B, Claim C and the substitution property \eqref{subst_2} $(g_1, \dots, g_n, gg')$ is $G^0(k)$-conjugate to $T(\delta_1, \dots, \delta_n, \gamma\gamma')$. It follows from the uniqueness part of Claim A, that $\rho(\gamma) = g$, $\rho(\gamma') = g'$ and $\rho(\gamma\gamma') = gg'$.
So $\rho$ is indeed a homomorphism. It can be shown by the same methods, that $\Theta_{\rho} = \Theta$.
By \Cref{sshassamePC}, we can replace $\rho$ by its semisimplification $\rho^{\semi}$, which will be $G$-completely reducible and $\Theta_{\rho^{\semi}} = \Theta$.

We are left to show that we can recover a $G$-completely reducible representation $\rho : \Gamma \to G(k)$ from its associated $G$-pseudocharacter $\Theta_{\rho}$.
For $n \geq 1$ and $\gamma \in \Gamma^n$, let $\xi_{\gamma} \in (G^n \sslash G^0)(k)$ as before and $T(\gamma) := (\rho(\gamma_1), \dots, \rho(\gamma_n)) \in G^n(k)$. By the non-connected version of \cite[Lemma 2.10]{BMR} as explained in \cite[§6.2]{BMR}, we find $\delta_1, \dots, \delta_n \in \Gamma$, such that for every parabolic $P$ and every Levi $L$ of $P$, we have $\rho(\Gamma) \subseteq P(k)$ if and only if $\{\delta_1, \dots, \delta_n\} \subseteq P(k)$ and $\rho(\Gamma) \subseteq L(k)$ if and only if $\{\delta_1, \dots, \delta_n\} \subseteq L(k)$.
In particular, $(g_1, \dots, g_n) := (\rho(\delta_1), \dots, \rho(\delta_n))$ has closed $G^0(k)$-orbit.
After possibly enlarging the tuple $(\delta_1, \dots, \delta_n)$, we may assume that $Z_{G_k}(g_1, \dots, g_n)(k) = Z_{G_k}(\rho(\Gamma))(k)$.

Let $\gamma \in \Gamma$. We know that $(g_1, \dots, g_n, \rho(\gamma)) = T(\delta_1, \dots, \delta_n, \gamma)$. Suppose $g \in G(k)$ is such that $(g_1, \dots, g_n, g)$ is $G^0(k)$-conjugate to $T(\delta_1, \dots, \delta_n, \gamma)$. So we find $x \in Z_{G_k}(g_1, \dots, g_n)(k) = Z_{G_k}(\rho(\Gamma))(k)$, such that $x\rho(\gamma)x^{-1} = g$, but this just means $\rho(\gamma) = g$.
\end{proof}

We prove a continuous version of \Cref{reconstructiongeneral}.
Beware, that a representation $\rho$ which is not $G$-completely reducible and not continuous might give rise to a continuous $G$-pseudocharacter $\Theta_{\rho}$.

\begin{theorem}\label{continuousreconstruction}
    If $k$ in \Cref{reconstructiongeneral} is an algebraic closure of a non-archimedean local field, $\Gamma$ is a profinite group and $\Theta$ is continuous, then $\rho$ is continuous.
\end{theorem}

\begin{proof} Suppose $E$ is a local field and $k$ is an algebraic closure of $E$. We inherit all notations and variables from the proof of \Cref{reconstructiongeneral}. In particular, we fix $\delta \in \Gamma^n$ satisfying conditions \eqref{recone}, \eqref{rectwo}, \eqref{recthree} and \eqref{recfour} and a system of representatives $(g_1, \dots, g_n) := T(\delta) \in G^n(k)$ and we may assume that $(g_1, \dots, g_n) \in G^n(E)$. For $\gamma \in \Gamma$, we write $(\delta, \gamma) := (\delta_1, \dots, \delta_n, \gamma) \in \Gamma^{n+1}$. The homomorphism $\Theta_{n+1} : \calO[G^{n+1}]^{G^0} \to \calC(\Gamma^{n+1}, k)$ can be seen as an element of $(G^{n+1} \sslash G^0)(\calC(\Gamma^{n+1}, k)) = \calC(\Gamma^{n+1}, (G^{n+1} \sslash G^0)(k))$ and thus determines a continuous map $\zeta : \Gamma \to (G^{n+1} \sslash G^0)(k), ~\gamma \mapsto \xi_{(\delta, \gamma)}$. We obtain a unique map $\rho : \Gamma \to G(k)$ from Claim A. The following diagram commutes and $h$ is injective on the image of $(g, \rho)$.
\begin{equation}\label{diagram5}
    \begin{tikzcd}
         & \{(g_1, \dots, g_n)\} \times G(k) \arrow{d}{h} \\
        \Gamma \arrow{ur}{(g,\rho)} \ar[r, swap, "\zeta"] & (G^{n+1} \sslash G^0)(k)
    \end{tikzcd}
\end{equation}
Since $G(k)$ is the inductive limit of $G(E')$ for finite extensions $E'/E$, we are left to show, that $h$ is a topological embedding on $\im(\rho) \cap G(E')$ for every finite extension $E'/E$. Since $\im(\rho)$ is compact, the set $\im(\rho) \cap G(E')$ is contained in a finite disjoint union of $G(\calO_{E'})$-cosets in $G(E')$ for some ring of integers $\calO_{E'}$ of $E'$. Since $(G^{n+1} \sslash G^0)(k)$ is Hausdorff, it follows that $h$ is a topological embedding on $\im(\rho) \cap G(E')$, as desired.
\end{proof}

\subsection{Kernels}
We will need kernels of $G$-pseudocharacters for the proof of \Cref{factoringdiscrete}, which in turn is needed for the construction of the generic fiber in \Cref{secpadicspace}.

\begin{definition}\label{defkernel} Let $\Theta \in \PC_G^{\Gamma}(A)$ be an arbitrary $G$-pseudocharacter as in \Cref{LafPC}. We define the \emph{kernel} $\ker(\Theta)$ of $\Theta$ as the set of all $\delta \in \Gamma$, such that for all $m \geq 1$, all $f \in \calO[G^m]^{G^0}$ and all $\gamma_1, \dots, \gamma_m \in \Gamma$, we have $\Theta_m(f)(\gamma_1, \dots, \gamma_m \delta) = \Theta_m(f)(\gamma_1, \dots, \gamma_m)$.
\end{definition}

\begin{lemma} The set $\ker(\Theta)$ in \Cref{defkernel} is a normal subgroup of $\Gamma$.
\end{lemma}

\begin{proof} It is clear, that $\ker(\Theta)$ is a subgroup of $\Gamma$. Let $\delta \in \ker(\Theta)$, $h \in \Gamma$ and $\gamma_1, \dots, \gamma_m \in \Gamma$ for some $m \geq 1$. Then
\begin{align*}
    \Theta_m(f)(\gamma_1, \dots, \gamma_m h \delta h^{-1}) &= \Theta_{m+1}(\hat f)(\gamma_1, \dots, \gamma_m h \delta, h^{-1}) \\
    &= \Theta_{m+1}(\hat f)(\gamma_1, \dots, \gamma_m h, h^{-1}) = \Theta_{m}(f)(\gamma_1, \dots, \gamma_m)
\end{align*}
so $h \delta h^{-1} \in \ker(\Theta)$.
\end{proof}

If $\delta \in \ker(\Theta)$, then $\Theta_m(f)(\gamma_1, \dots, \gamma_{i-1}, \gamma_i \delta, \gamma_{i+1}, \dots, \gamma_m) = \Theta_m(f)(\gamma_1, \dots, \gamma_m)$
for every $i=1, \dots, m$.
We will refer to the next lemma as the \emph{homomorphisms theorem} for $G$-pseudocharacters.

\begin{lemma}\label{homtheorem} Let $\Theta \in \PC_G^{\Gamma}(A)$ be an arbitrary $G$-pseudocharacter as in \Cref{LafPC}, let $\Delta \leq \Gamma$ be a normal subgroup and assume, that $\Delta \subseteq \ker(\Theta)$. Then there is a unique $G$-pseudocharacter $\Theta' \in \PC_G^{\Gamma/\Delta}(A)$, such that $\Theta$ is the restriction of $\Theta'$ to $\Gamma$.
\end{lemma}

\begin{proof} Uniqueness is clear, since $\Gamma \to \Gamma/\Delta$ is surjective and hence the maps $\map((\Gamma/\Delta)^m, A) \to \map(\Gamma^m, A)$ are injective for all $m \geq 1$. We can define $\Theta'$ as $\Theta_m'(f)(\gamma_1 \Delta, \dots, \gamma_m \Delta) := \Theta_m(f)(\gamma_1, \dots, \gamma_m)$ for all $m \geq 1$, all $f \in \calO[G^m]^{G^0}$ and all $\gamma_1, \dots, \gamma_m \in \Gamma$. This is well-defined, since $\Delta \subseteq \ker(\Theta)$. The axioms of a $G$-pseudocharacter are easily verified.
\end{proof}

If $\Theta_{\rho}$ is a $G$-pseudocharacter which comes from a representation $\rho$, we have $\ker(\rho) \subseteq \ker(\Theta_{\rho})$, but the converse inclusion need not hold.
It follows from \Cref{reconstructiongeneral}, that equality holds when $\rho$ is $G$-completely reducible over an algebraically closed field.

\subsection{Direct sum, dual and tensor product}\label{secsumtensor} Recall from \Cref{subsecLafPC}, that a homomorphism $G \to H$ gives rise to a natural transformation $\PC^{\Gamma}_G \to \PC^{\Gamma}_H$. This provides us with an easy way to define natural operations on pseudocharacters, such as direct sums, duals and tensor products. Defining such operations for determinant laws is more involved; see e.g. \cite[§1.1.11]{MR3167286} for a direct sum operation and \cite[§4.5]{BJ_new} for twisting with a character. It is clear by construction, that these operations will be compatible with the corresponding operations on representations.

Suppose $\Theta \in \PC^{\Gamma}_{\GL_n}(A)$. Then we can define the \emph{dual} $\Theta^*$ by composing with the transpose inverse map $\GL_n \to \GL_n$.

Assume, that $\calO$ is a principal ideal domain. Suppose $\Theta \in \PC^{\Gamma}_{\GL_a}(A)$ and $\Theta' \in \PC^{\Gamma}_{\GL_b}(A)$ for $A \in \CAlg_{\calO}$ and $a+b=n$.
We will define the \emph{direct sum} $\Theta \oplus \Theta' \in \PC^{\Gamma}_{\GL_n}(A)$. For $m \geq 1$, we obtain a map $\Theta_m \otimes \Theta_m' : \calO[\GL_a^m]^{\GL_a} \otimes_{\calO} \calO[\GL_b^m]^{\GL_b} \to \map(\Gamma^m, A)$. It turns out, that since $\calO$ is a principal ideal domain and by the universal coefficient theorem \cite[I.4.18 Proposition (a)]{Jantzen2003}, we have $\calO[\GL_a^m]^{\GL_a} \otimes_{\calO} \calO[\GL_b^m]^{\GL_b} = \calO[(\GL_a \times \GL_b)^m]^{\GL_a \times \GL_b}$. The diagonal embedding $\GL_a \times \GL_b \to \GL_n$ induces a map $\calO[\GL_n^m]^{\GL_n} \to \calO[(\GL_a \times \GL_b)^m]^{\GL_a \times \GL_b}$ and we define $(\Theta \oplus \Theta')_m$ as the composition of this map with $\Theta_m \otimes \Theta_m'$.
The compatibility conditions (1) and (2) in \Cref{LafPC} can be verified directly, but the alternative description of pseudocharacters \Cref{allemorphismen} in the next section provides us with an easier way to see, that $\Theta \oplus \Theta'$ is indeed a pseudocharacter.

As for the direct sum, the \emph{tensor product} $\Theta \otimes \Theta'$ is induced by the dyadic product map $\GL_a \times \GL_b \to \GL_{ab}$.

We shall also need the notion of direct sum of two symplectic pseudocharacters, induced by the natural map $\Sp_{2a} \times \Sp_{2b} \to \Sp_{2n}$ for $a+b=n$, which corresponds to the orthogonal direct sum of symplectic spaces. The procedure for the construction of this direct sum operation is the same as for the general linear group explained above.

There is also a natural map $\GL_n \to \Sp_{2n}$ induced by mapping the standard representation $V$ of $\GL_n$ to $V \oplus V^*$ equipped with the symplectic form, which makes $V$ and $V^*$ totally isotropic subspaces, is the canonical pairing on $V \times V^*$ and the negative of the canonical pairing on $V^* \times V$.

\subsection{${\calC}$-${\calO}$-algebras}

It turns out to be useful to rephrase the definition of $G$-pseudocharacters in terms of functors on a category $\calC$ with values in commutative ${\calO}$-algebras, which we decided to call '${\calC}$-${\calO}$-algebras'. It allows to uniformly describe all substitution properties for $G$-pseudocharacters that arise as consequences of substitution properties \eqref{subst_1} and \eqref{subst_2} in \Cref{LafPC}. We will do this in \Cref{sec_GPC_as_FO}. There is an intrinsic notion for a ${\calC}$-${\calO}$-algebra to be finitely generated, and we will establish this property for some of the relevant ${\calC}$-${\calO}$-algebras in \Cref{secinvthy}. Instances of ${\calC}$-${\calO}$-algebras appear in \cite{Weidner} and \cite{Zhu} under the names F(I)-, FFM- and FFG-algebra. We develop the basic theory of ${\calC}$-${\calO}$-algebras and use them to prove existence and basic properties of an affine moduli scheme of $G$-pseudocharacters in \Cref{secRep}.

\begin{definition}[${\calC}$-$\calO$-algebra]\label{defCOalg} Let $\calC$ be a small category.
\begin{enumerate}
    \item A \emph{${\calC}$-$\calO$-algebra} is a functor $A^{\bullet} : \calC \to \CAlg_{\calO}, ~c \mapsto A^c$ to the category of commutative $\calO$-algebras $\CAlg_{\calO}$.
    \item A \emph{homomorphism} of $\calC$-$\calO$-algebras is a natural transformation $f^{\bullet} : A^{\bullet} \to B^{\bullet}$.
    \item Let $\CAlg_{\calO}^{\calC}$ be the category of $\calC$-$\calO$-algebras together with $\calC$-$\calO$-homomorphisms.
    \item A \emph{${\calC}$-$\calO$-subalgebra} of a $\calC$-$\calO$-algebra $A^{\bullet}$ is a subfunctor $B^{\bullet} \subseteq A^{\bullet}$, such that $B^c$ is an $\calO$-subalgebra of $A^c$ for all objects $c$ of $\calC$.
    \item A \emph{${\calC}$-$\calO$-ideal} is a subfunctor $I^{\bullet}$ of the composition of $A^{\bullet}$ with the forgetful functor $\CAlg_{\calO} \to \Set$, such that $I^c$ is an ideal of $A^c$ for all objects $c$ of $\calC$.
    \item A ${\calC}$-$\calO$-homomorphism $f^{\bullet} : A^{\bullet} \to B^{\bullet}$ is \emph{injective} (\emph{surjective}, \emph{bijective}) if $f^c$ is injective (surjective, bijective) for all objects $c$ of $\calC$.
    \item The \emph{kernel} $\ker(f)^{\bullet}$ of a $\calC$-$\calO$-homomorphism $f^{\bullet} : A^{\bullet} \to B^{\bullet}$ is defined by $\ker(f)^c := \ker(f^c)$. It is a $\calC$-$\calO$-ideal of $A^{\bullet}$.
    \item The \emph{image} $\im(f)^{\bullet}$ of a $\calC$-$\calO$-homomorphism $f^{\bullet} : A^{\bullet} \to B^{\bullet}$ is defined by $\im(f)^c := \im(f^c)$. It is a $\calC$-$\calO$-subalgebra of $B^{\bullet}$. 
\end{enumerate}
\end{definition}

$\calC$-$\calO$-algebras are just commutative $\calO$-algebra objects internal to the topos of $\calC$-sets, i.e. functors $\calC \to \Set$, and \Cref{defCOalg} comes from this perspective.

\subsection{$G$-pseudocharacters as $\calF$-$\calO$-algebra homomorphisms}
\label{sec_GPC_as_FO}

From now on, we will consider two different small categories for $\calC$. We fix an alphabet of symbols $\{x_1, x_2, x_3, \dots\}$.
\begin{enumerate}
    \item Let $\calM$ be the category of free monoids $\FM(m)$ on $m$ generators $x_1, \dots, x_m$ for all $m \geq 1$.
    \item Let $\calF$ be the category of free groups $\FG(m)$ on $m$ generators $x_1, \dots, x_m$ for all $m \geq 1$.
\end{enumerate}
So $\calM$ (resp. $\calF$) contains for each $m$ exactly one object that is free on $m$ generators.
A monoid homomorphism between finitely generated free monoids can be understood as a finite sequence of words.
Such a sequence also defines a homomorphism between free groups and so we get a canonical functor $\calM \to \calF$ which takes a free monoid $\FM(m)$ to a free group $\FG(m)$ on the same set of generators.
In particular, every $\calF$-$\calO$-algebra can be restricted to an $\calM$-$\calO$-algebra.

\begin{example} Here are the two examples of $\calF$-$\calO$-algebras we are interested in.
\begin{enumerate}
    \item If $A$ is an $\calO$-algebra, then the functor 
    $$ \calF \to \CAlg_{\calO}, ~\FG(m) \mapsto \map(\Gamma^m,A) $$
    where $\alpha : \FG(n) \to \FG(m)$ is mapped to $\alpha_* : \map(\Gamma^n,A) \to \map(\Gamma^m,A)$, where $\alpha_*(f)(\gamma_1, \dots, \gamma_m) := f(\phi(\alpha(x_1)), \dots, \phi(\alpha(x_n)))$ and $\phi : \FG(m) \to \Gamma, ~x_i \mapsto \gamma_i$, defines an $\calF$-$\calO$-algebra $\map(\Gamma^{\bullet},A)$.
    \item Similarly 
    $$ \calF \to \CAlg_{\calO}, ~\FG(m) \mapsto \calO[G^m]^{G^0} $$
    defines an $\calF$-$\calO$-algebra: Every homomorphism $\alpha : \FG(n) \to \FG(m)$ induces a morphism of $\calO$-schemes $G^m \to G^n$, which in turn induces the desired map $\alpha_* : \calO[G^n]^{G^0} \to \calO[G^m]^{G^0}$. Note, that since $G^m \to G^n$ is induced by a homomorphism of free groups it is equivariant with respect to diagonal conjugation and hence $\alpha_*$ is well-defined. We will denote this $\calF$-$\calO$-algebra by $\calO[G^{\bullet}]^{G^0}$.
\end{enumerate}
\end{example}

By definition a $G$-pseudocharacter $\Theta$ is a sequence of maps $\Theta_m : \calO[G^m]^{G^0} \to \map(\Gamma^m,A)$, that is natural with respect to two specified types of monoid homomorphisms. 
Our next goal is to understand, that these types of monoid homomorphisms do already generate all morphisms in $\calM$ and make $\Theta_{\bullet} = (\Theta_m)_{m \geq 0}$ an $\calM$-$\calO$-homomorphism.

\begin{definition}\label{defgen} Let $\calC$ be a category and $S$ a system of morphisms $S_{A,B} \subseteq \Hom_{\calC}(A,B)$ for all pairs of objects $A,B$. Let $\tilde S$ be another such system of morphisms.
\begin{enumerate}
    \item $S$ \emph{generates} $\tilde S$, if $\tilde S$ is the smallest system of morphisms, that contains $S$, all identities and for any two composable morphisms $\alpha_1, \alpha_2 \in \tilde S$ their composition $\alpha_2 \circ \alpha_1$.
    \item $S$ \emph{inv-generates} $\tilde S$, if $\tilde S$ is the smallest system of morphisms, that contains $S$, all identities, for any two composable morphisms $\alpha_1, \alpha_2 \in \tilde S$ their composition $\alpha_2 \circ \alpha_1$ and for each invertible morphism $\alpha \in \tilde S$ its inverse $\alpha^{-1}$.
\end{enumerate}
\end{definition}

\begin{remark} \label{structuregeneration} A system of morphisms $S$ always (inv-)generates a unique system of morphisms, since the conditions in \Cref{defgen} are closed under arbitrary intersections.
If $S$ generates $\tilde S$, then $\tilde S$ consists of compositions of morphisms of $S$ and identities. If $S$ inv-generates $\tilde S$, then $\tilde S$ consists of iterated compositions and inversions of morphisms of $S$ that are invertible in $\calC$ and identities.
\end{remark}

Let $F, G : \calC \to \calD$ be functors, let $\eta : F \to G$ be a collection of morphisms $\eta_X : FX \to GX$ and assume that $\eta$ is natural for all $\alpha : X \to Y$ contained in an inv-generating system of morphisms in $\calC$.
Then it follows by structural induction, that $\eta$ is a natural transformation.
Hence it is enough to check naturality on (inv-)generating systems of morphisms.

\begin{lemma} \label{generatingsystem} The morphisms of $\calM$ are generated by the following two types of homomorphisms:
\begin{itemize}
    \item[(1)] $\phi : \FM(n) \to \FM(m)$, where $\phi(x_i) := x_{\zeta(i)}$ for each map $\zeta : \{1, \dots, n\} \to \{1, \dots, m\}$.
    \item[(2)] $\phi : \FM(n) \to \FM(n+1)$ where $\phi(x_i) := x_i$ for all $i < n$ and $\phi(x_n) := x_nx_{n+1}$.
\end{itemize}
The morphisms of $\calF$ are generated by homomorphisms of types (1) and (2) with $\FM$ replaced by $\FG$ and a third type of homomorphism:
\begin{itemize}
    \item[(3)] $\phi : \FG(n) \to \FG(n)$ where $\phi(x_i) := x_i$ for all $i < n$ and $\phi(x_n) := x_n^{-1}$.
\end{itemize}
\end{lemma}

\begin{proof} This is \cite[Lemma 3]{Weidner}.
\end{proof}

The morphisms of $\calF$ are not generated by homomorphisms of type (1) and (2) with $\FM$ replaced by $\FG$: Homomorphisms of type (1) and (2) have the property, that the image of the generators $x_i$ lies in the submonoid spanned by generators. This property is stable under compositions and hence the homomorphism $\FG(1) \to \FG(1), ~x_1 \mapsto x_1^{-1}$ is not a composition of type (1) or (2) homomorphisms. However, $\calF$ is inv-generated by these homomorphisms.

\begin{lemma}\label{invgenF} The morphisms of $\calF$ are inv-generated by homomorphisms of type (1) and (2) in \Cref{generatingsystem} with $\FM$ replaced by $\FG$.
\end{lemma}

\begin{proof} By \Cref{structuregeneration} and \Cref{generatingsystem} it suffices to show, that homomorphisms of type (3) can be written as iterated compositions and inversions of homomorphisms of type (1) and (2). Since $\calM$ is generated by monoid homomorphisms of types (1) and (2), we already know, that all group homomorphisms $\FG(n) \to \FG(m)$, that are induced by monoid homomorphisms $\FM(n) \to \FM(m)$ are generated by group homomorphisms of type (1) and (2).

Let $\phi : \FG(n) \to \FG(n)$ be of type (3). We will use tuple notation for homomorphisms, so $\phi = (\phi(x_1), \dots, \phi(x_n)) = (x_1, \dots, x_{n-1}, x_n^{-1})$. Suppose $n=2$ and $x_1 = x, x_2 = y$, the computation for $n \geq 3$ is analogous. We have $(x,y^{-1}) = (xy^{-1}, y) \circ (xy,x) \circ (x, x^{-1}y)$.
Since $(xy^{-1}, y) = (xy, y)^{-1}$ and $(x, x^{-1}y) = (x, xy)^{-1}$ we see that $\phi$ is inv-generated by homomorphisms of type (1) and (2). For $n=1$ we consider the homomorphisms $(y) : \FG(1) \to \FG(2), ~x \mapsto y$ and $(1,x) : \FG(2) \to \FG(1), ~x \mapsto 1, ~y \mapsto x$ and write $(x^{-1}) = (1,x) \circ (x,y^{-1}) \circ (y)$.
\end{proof}

\begin{proposition}\label{allemorphismen}
    Let $A$ be a commutative $\calO$-algebra. For both $\calE = \CAlg_{\calO}^{\calM}$ and $\calE = \CAlg_{\calO}^{\calF}$, the following map is bijective:
    $$ \Hom_{\calE}(\calO[G^{\bullet}]^{G^0}, ~\map(\Gamma^{\bullet},A)) \to \PC_G^{\Gamma}(A), \quad \Theta \mapsto (\Theta_m)_{m\geq 1} $$
    The same holds when $\Gamma$ is a topological group and $A$ is a topological ring with $\map(\Gamma^{\bullet},A)$ replaced by $\calC(\Gamma^{\bullet},A)$ and $\PC_G^{\Gamma}(A)$ replaced by $\cPC_G^{\Gamma}(A)$.
\end{proposition}

\begin{proof} Suppose $\calE = \CAlg_{\calO}^{\calM}$. We start with a $G$-pseudocharacter $(\Theta_m)_{m \geq 1}$ and define an association $\tilde \Theta : \calO[G^{\bullet}]^{G^0} \to \map(\Gamma^{\bullet},A) $ by setting $\tilde \Theta_{\FM(m)} := \Theta_m$. By definition of $\Theta$ we know, that $\tilde \Theta$ is natural with respect to morphisms $\FM(n) \to \FM(m)$ of type (1) and morphisms $\FM(n) \to \FM(n+1)$ of type (2). By \Cref{generatingsystem} this implies naturality.
Conversely, given a morphism $\tilde \Theta$ of $\calM$-$\calO$-algebras, the associated sequence of algebra maps $\Theta_n := \Theta_{\FM(n)}$ satisfies the required properties by naturality.

The same argument, using \Cref{invgenF}, establishes bijectivity when $\calE = \CAlg_{\calO}^{\calF}$. The continuous case is deduced directly from the definition of continuity of pseudocharacters.
\end{proof}

\subsection{Representability of $\PC^{\Gamma}_G$}
\label{secRep}

\begin{theorem}\label{repofPC} The functor $\PC^{\Gamma}_G : \CAlg_{\calO} \to \Set$ is representable by a commutative $\calO$-algebra $B_G^{\Gamma}$. As an $\calO$-algebra $B_G^{\Gamma}$ is generated by $\{\Theta^u_m(\mu)(\gamma) \mid \mu \in \calO[G^{m}]^{G^0}, ~\gamma \in \Gamma^{m}\}$, where $\Theta^u \in \PC^{\Gamma}_G(B_G^{\Gamma})$ is the universal $G$-pseudocharacter.
\end{theorem}

For all $m \in \bbN$, $\mu \in \calO[G^m]^{G^0}$, $\gamma = (\gamma_1, \dots, \gamma_m) \in \Gamma^m$, for every $A \in \CAlg_{\calO}$ and every $\Theta \in \PC^{\Gamma}_G(A)$, the associated homomorphism $f_{\Theta} : B_{G}^{\Gamma} \to A$ satisfies $f_{\Theta}(\Theta^u_m(\mu)(\gamma)) = \Theta_{m}(\mu)(\gamma)$.
The following argument can be found in \cite[Remark 2.2.5]{Zhu}.

\begin{proof}
    We use the description of pseudocharacters as $\calF$-$\calO$-algebra homomorphisms according to \Cref{allemorphismen}.
    We denote by $\calF/\Gamma$ the slice category of objects of $\calF$ with a fixed homomorphism to $\Gamma$.
    Let
    $$ B_G^{\Gamma} := \colim_{\FG(m) \in \calF/\Gamma} \calO[G^m]^{G^0} $$
    be the colimit in $\CAlg_{\calO}$ indexed over the small category $\calF/\Gamma$.
    Then
    \begin{align*}
        \Hom_{\CAlg_{\calO}}(B_G^{\Gamma}, A) &= \lim_{\FG(m) \in \calF/\Gamma}\Hom_{\CAlg_{\calO}}(\calO[G^m]^{G^0}, A) \\
        &= \Hom_{\CAlg_{\calO}^{\calF/\Gamma}}(\calO[G^{\bullet}]^{G^0}, A) = \Hom_{\CAlg_{\calO}^{\calF}}(\calO[G^{\bullet}]^{G^0}, \map(\Gamma^{\bullet}, A))
    \end{align*}
    for every $A \in \CAlg_{\calO}$, where in the second line $A$ is understood as the constant functor on $\calF/\Gamma$.
    In the last line, we compute the right Kan extension of $A : \calF/\Gamma \to \CAlg_{\calO}$ along the canonical restriction $p : \calF/\Gamma \to \calF$ as $$(\Ran_p A)(\FG(m)) = \lim_{\atopnew{(\FG(n), f) \in \calF/\Gamma}{\varphi \in \Hom(\FG(m), \FG(n))}} A(\FG(n), f) = \lim_{(\FG(n), f) \in \calF/\Gamma} \map(\FG(n)^m, A) = \map(\Gamma^m, A)$$
    using the description of Kan extensions as weighted limits \cite[Theorem 6.2.1]{Riehl}.
\end{proof}

It follows from the proof of \Cref{repofPC}, that $B_G^{\Gamma}$ is finitely generated if $\Gamma$ is and all $\calO[G^m]^{G^0}$ are finitely generated.
 
\begin{proposition}\label{decisivefiniteness} If $\Gamma$ is finitely generated and $\calO$ is noetherian, then $B^{\Gamma}_G$ is a finitely generated $\calO$-algebra.
\end{proposition}

\begin{proof}
    We use the description of $B_G^{\Gamma}$ as a colimit as in the proof of \Cref{repofPC}.
    If $\Gamma$ is finitely generated, then $\calF/\Gamma$ contains a surjection $\pi : \FG(m) \twoheadrightarrow \Gamma$.
    For every $\FG(n) \in \calF$ every homomorphism $f : \FG(n) \to \Gamma$ factors over $\pi$, so the associated map $f_* : \calO[G^n]^{G^0} \to B_G^{\Gamma}$ factors over the map $\pi_* : \calO[G^m]^{G^0} \to B^{\Gamma}_G$ associated to $\pi$, which implies, that $\pi_*$ is surjective. So it suffices to see, that $\calO[G^m]^{G^0}$ is finitely generated. Since the canonical map $[G^m/G^0] \to G^m \sslash G^0$ is an adequate moduli space (see \cite[Theorem 9.1.4]{alper}), it follows from \cite[Theorem 6.3.3]{alper}, that $\calO[G^m]^{G^0}$ is a finitely generated $\calO$-algebra.
\end{proof}

\begin{proposition}\label{basechange} Let $\calO' \in \CAlg_{\calO}$ and assume that one of the following conditions holds.
\begin{enumerate}
    \item $\calO'$ is $\calO$-flat
    \item $G$ is a generalized reductive group over a Dedekind domain $\calO$.
\end{enumerate}
Then for any $\calO'$-algebra $A$, there is a canonical bijection
\begin{align}
    \PC^{\Gamma}_{G_{\calO'}}(A) \cong \PC^{\Gamma}_G(A) \label{PCISO}
\end{align}
induced by a canonical isomorphism $\calO[G^{\bullet}]^{G^0} \otimes_{\calO} \calO' \to \calO'[G^{\bullet}]^{G^0}$ of $\calF$-$\calO'$-algebras.
Moreover, there is a canonical isomorphism $B_G^{\Gamma} \otimes_{\calO} \calO' \cong B_{G_{\calO'}}^{\Gamma}$ of $\calO'$-algebras.
\end{proposition}

\begin{proof} By \Cref{allemorphismen} it is enough to show, that $\calO[G^m]^{G^0} \otimes_{\calO} \calO' = \calO'[G^m]^{G^0}$ for all $m \geq 1$. This follows from \Cref{critchange} and \Cref{gfG}.

We now prove that $B_G^{\Gamma} \otimes_{\calO} \calO' \cong B_{G_{\calO'}}^{\Gamma}$. We apply \Cref{repofPC} twice and the first assertion once:
$$ \Hom_{\calO'}(B_{G_{\calO'}}^{\Gamma}, A) \overset{\eqref{repofPC}}{=} \PC^{\Gamma}_{G_{\calO'}}(A) \overset{\eqref{PCISO}}{\cong} \PC^{\Gamma}_{G}(A) \overset{\eqref{repofPC}}{=} \Hom_{\calO}(B_{G}^{\Gamma}, A) = \Hom_{\calO'}(B_{G}^{\Gamma} \otimes_{\calO} \calO', A) $$
The claim follows by Yoneda.
\end{proof}

\section{Invariant theory}\label{secinvthy} In this section, we compute generators of invariant algebras for $\SL_n$, $\GL_n$, $\Sp_{2n}$, $\GSp_{2n}$, $\SO_{2n+1}$, $\OO_{2n+1}$ and $\GO_{2n+1}$.
This is the key input in proving noetherianity of pseudodeformation rings under the hypothesis that our profinite group satisfies Mazur's condition $\Phi_p$.

The action of $G$ on $G^m$ by diagonal conjugation induces an algebraic action of $G$ on $\calO[G^m]$ by $\calO$-algebra automorphisms.
For each of the groups above, let $\iota : G \hookrightarrow \GL_d$ be the standard representation where $d \in \{n, 2n, 2n+1\}$ in the respective cases.
Let us write $M_d$ for the affine scheme (over $\calO$) which represents the functor that maps a commutative $\calO$-algebra $A$ to the set $M_d(A)$ of $d \times d$ matrices. It is equipped with an algebraic action of $\GL_d$ given on $A$-valued points by $\GL_d(A) \times M_d(A) \to M_d(A), ~(g,X) \mapsto gXg^{-1}$.
Similarly $G$ acts on $M_d^m$ by diagonal conjugation through $\iota$, which induces an action of $G$ on $\calO[M_d^m]$ by $\calO$-algebra automorphisms.
We will use the letter $\bbX^{(k)}$ to denote the projection to the $k$-th factor $G^m \to G$, seen as an element of $G(\calO[G^m])$.
We will also denote by $\bbX^{(k)}$ the projection to the $k$-th factor $M_d^m \to M_d$ as an element of $M_d(\calO[M_d^m])$.

We distinguish between two types of theorems:
\begin{enumerate}
    \item \emph{First fundamental theorem (FFT)}: \\ Determine an explicit set of generators of $\calO[M_d^m]^G$ or $\calO[G^m]^G$.
    \item \emph{Second fundamental theorem (SFT)}: \\ Determine an explicit generating set of relations between given generators of $\calO[M_d^m]^G$ or $\calO[G^m]^G$.
\end{enumerate}

The first results of this kind in characteristic $0$ are due to Frobenius, Sibirskii \cite{Sib67} and Procesi \cite{Pro}. Since $\bbQ[M_n^m]^G \twoheadrightarrow \bbQ[G^m]^G$ is surjective, one can reduce the computation of generators of $\bbQ[G^m]^G$ to $\bbQ[M_n^m]^G$. Donkin proved, that if $K$ is an algebraically closed field, the algebras $K[G^m]^G$ are generated by shifted traces of tilting modules \cite{Donkin1992}. This has since been turned into a concrete description of generators of $K[\GL_n^m]^{\GL_n}$ by Donkin and $K[\Sp_n^m]^{\Sp_n}$ and $K[\OO_n^m]^{\OO_n}$ (under the assumption $\chara(K) \neq 2$ in the orthogonal case) by Zubkov \cite{ZubkovOnTheProcedure, zubkov99}. We can descend generators of invariant algebras to the prime fields $\bbQ$ and $\bbF_p$ and lift them to $\bbZ/p^r$ (see \Cref{invoverfield}, \Cref{genliftlemma}). This is sufficient for our applications to deformation theory. Using results on good filtrations it is possible to descend these generators further to $\bbZ[G^m]^G$ once they are known over fields and defined over $\bbZ$. We include this argument in forthcoming joint work with Mohamed Moakher.

The second fundamental theorem for $\bbQ[M_n^m]^{\GL_n}$ has been proven independently by Procesi \cite{Pro} and Razmyslov \cite{Razmyslov}. In positive characteristic it is due to Zubkov \cite{zubkov99}. In \cite[Theorem 1.13]{MR3726879} de Concini and Procesi prove a second fundamental theorem over $\bbZ$. The work on semi-invariants of quivers over infinite fields was further developed by Domokos and Zubkov \cite{DomokosZubkov}.
While second fundamental theorems are needed to give explicit characterizations of $G$-pseudocharacters and are of notoriously difficult combinatorial nature we have found that for the theory of $G$-pseudocharacters first fundamental theorems are usually sufficient.

\subsection{$\SL_n$ and $\GL_n$}

In \cite[15.2.1]{MR3726879} de Concini and Procesi have determined the generators of $\bbZ[M_n^m]^{\GL_n}$ and $\bbZ[M_n^m]^{\SL_n}$, from which the generators of $\bbZ[\GL_n^m]^{\GL_n}$ and $\bbZ[\SL_n^m]^{\SL_n}$ can be computed. We reprove their result using good filtrations and avoiding usage of the formal character of $\bbZ[M_n^m]$ and the analysis of root subgroups.

Let $\bbX \in M_d(\bbZ[x_{ij} \mid i,j \in \{1, \dots, d\}])$ be a generic $d \times d$ matrix, i.e. $\bbX_{ij} = x_{ij}$ for $1 \leq i,j \leq d$.
Let $\sigma_i \in \bbZ[x_{ij}]$ be up to a sign the $i$-th coefficient of the characteristic polynomial of $\bbX$:
$$ \det(t \cdot I_d - \bbX) = \sum_{i=0}^d (-1)^i \sigma_i(\bbX) t^{d-i} \quad \in \bbZ[x_{ij} \mid i,j \in \{1, \dots, d\}][t] $$
If we evaluate $\bbX$ at a triangular matrix, then $\sigma_i$ is given by the $i$-th elementary symmetric polynomial in the diagonal entries.

Recall the first fundamental theorem on $\GL_n$-invariants of several matrices.
The first fundamental theorem for $\SL_n$-invariants of matrices follows directly.

\begin{theorem}[De Concini, Procesi]\label{DCPFFT}
    Let $K$ be an algebraically closed field.
    Then $K[M_n^m]^{\GL_n}$ and $K[M_n^m]^{\SL_n}$ are generated by elements of the form
    $$ \sigma_i(\bbX^{(j_1)} \cdots \bbX^{(j_s)}) $$
    for $i \in \{1, \dots, n\}$ and $s \geq 0$.
\end{theorem}

\begin{proof}
    See \cite[Theorem 1.10]{MR3726879} for $\GL_d$. The inclusion of the center $\GL_1 \to \GL_n$ and the inclusion $\SL_n \to \GL_n$ combine to a surjection $\SL_n \times \GL_1 \to \GL_n$.
    Therefore $K[M_n^m]^{\GL_d} = K[M_n^m]^{\SL_n \times \GL_1} = K[M_n^m]^{\SL_n}$.
\end{proof}

To descend the first fundamental theorem to $\bbZ$, we need the following lemma.

\begin{lemma}\label{globaldescent} Let $\calO$ be a principal ideal domain and let $M$ and $M'$ be finitely generated free $\calO$-modules. Let $f : M \to M'$ be an $\calO$-module homomorphism, such that for every $\calO$-field $K$ the induced map $M \otimes_{\calO} K \to M' \otimes_{\calO} K$ is an isomorphism. Then $f$ is an isomorphism.
\end{lemma}

\begin{proof} Taking $K$ the field of fractions of $\calO$ shows, that $f$ is injective and that the cokernel $C$ of $f$ is a finitely generated torsion module. For every prime ideal $0 \neq \frakp \subseteq \calO$, the sequence
$$ M \otimes_{\calO} \calO/\frakp \xrightarrow{\sim} M' \otimes_{\calO} \calO/\frakp \to C \otimes_{\calO} \calO/\frakp \to 0 $$
is exact, which shows, that $C \otimes_{\calO} \calO/\frakp = 0$. It follows, that $C=0$.
\end{proof}

\begin{lemma}\label{gfM} For all $m,n \geq 1$, $\bbZ[M_n^m]$ equipped with the action of $G=\GL_n$ (resp. $G=\SL_n$) by conjugation has a good filtration. In particular, for every commutative ring $\calO$ and every $\calO$-algebra $\calO'$, the canonical map $\calO[M_n^m]^G \otimes_{\calO} \calO' \to \calO'[M_n^m]^G$ is an isomorphism.
\end{lemma}

\begin{proof} 
    Let $V$ be the standard representation of $G$.
    Since $M_n \cong V \otimes V^*$ and $V$ is self-dual, we have $M_n^m \cong V^{\otimes 2m}$.
    By \Cref{MathieuTPT} and the formula for symmetric powers of direct sums it is enough to show, that $\Sym^d(V)$ has a good filtration.
    But $\Sym^d(V)$ is a highest weight module, so we are done.
\end{proof}

\begin{theorem}[De Concini, Procesi]\label{DCPoverZ}
    We have $\bbZ[M_n^m]^{\GL_n} = \bbZ[M_n^m]^{\SL_n}$ and this ring is generated by elements of the form $\sigma_i(\bbX^{(j_1)} \cdots \bbX^{(j_s)})$ for $i \in \{1, \dots, n\}$ and $s \geq 0$.
\end{theorem}

\begin{proof}
    As algebraic representations $M_n^m = (V \otimes V^*)^{\oplus m}$, where $V$ is the standard representation of $\GL_n$. The standard representation of $\GL_n$ is self-dual and has a good filtration. We observe, that $\bbZ[M_n^m]$ admits a grading by finitely generated free $\bbZ$-modules $S_d := \Sym^d((M_n^m)^*)$, that is preserved by the action of $\GL_n$. The $S_d$ also have a good filtration, so by \Cref{critchange} $S_d^{\GL_n} \otimes_{\bbZ} k = (S_d \otimes_{\bbZ} k)^{\GL_n}$ for any field $k$.
    
    Let $A$ be a free commutative $\bbZ$-algebra generated by variables $t_{(i, (j_1, \dots, j_s))}$ with $i \in \{1, \dots, d\}$ and $s \geq 0$.
    We let $t_{(i, (j_1, \dots, j_s))}$ have degree $si$ and observe that $A = \bigoplus_{d=0}^{\infty} A_d$ is a graded ring, such that each submodule $A_d$ consisting of homogeneous of degree $d$ elements is a finitely generated free $\bbZ$-module.
    
    The natural map $A \to \bbZ[M_n^m]^{\GL_n}$ sending $t_{(i, (j_1, \dots, j_s))}$ to $\sigma_i(\bbX^{(j_1)} \cdots \bbX^{(j_s)})$ is graded. By \Cref{DCPFFT}, the maps $A_d \otimes_{\bbZ} k \to S_d^{\GL_n} \otimes_{\bbZ} k$ are surjective for every algebraically closed field $k$. Hence by \Cref{globaldescent} the maps $A_d \to S_d$ are surjective and thus $A \to S$ is surjective, proving the first claim. The argument for the $\SL_n$-invariants is the same, using \Cref{DCPFFT}.
\end{proof}

To pass from invariants of $\bbZ[M_n^m]$ to invariants of $\bbZ[\GL_n^m]$ and $\bbZ[\SL_n^m]$, we use the following Lemma.

\begin{lemma}\label{lemsurjtens}
    Let $G$ be a split reductive group over $\bbZ$ and let 
    $$ 0 \to C \to B \to A \to 0 $$
    $$ 0 \to C' \to B' \to A' \to 0 $$
    be two short exact sequences of $G$-modules with good filtration. Then the map $(B \otimes B')^G \to (A \otimes A')^G$ is surjective.
\end{lemma}

\begin{proof}
    Since good filtration modules are free, the sequences
    $$ 0 \to C \otimes A' \to B \otimes A' \to A \otimes A' \to 0 $$
    $$ 0 \to B \otimes C' \to B \otimes B' \to B \otimes A' \to 0 $$
    are exact.
    By Mathieu's tensor product theorem \Cref{MathieuTPT}, the modules $C \otimes A'$ and $B \otimes C'$ have good filtrations, hence by \Cref{critchange} the maps $(B \otimes A')^G \to (A \otimes A')^G$ and $(B \otimes B')^G \to (B \otimes A')^G$ are surjective.
\end{proof}

\begin{theorem}\label{invGLZ} Let $\calO$ be a commutative ring, let $m \geq 1$ and let $n \geq 1$.
\begin{enumerate}
    \item $\calO[M_n^m]^{\GL_n}$ and $\calO[M_n^m]^{\SL_n}$ are generated by elements of the form $\sigma_i(\bbX^{(j_1)} \cdots \bbX^{(j_s)})$ for $i \in \{1, \dots, n\}$ and $s \geq 0$.
    \item $\calO[\GL_n^m]^{\GL_n}$ and $\calO[\SL_n^m]^{\SL_n}$ are generated by elements of the form $\sigma_i(\bbX^{(j_1)} \cdots \bbX^{(j_s)})$ for $i \in \{1, \dots, n\}$ and $s \geq 0$ and $\det^{-1}(X_j)$ for $j \in \{1, \dots, m\}$.
\end{enumerate}
\end{theorem}

\begin{proof} 
    By \Cref{gfG}, it is for both $\GL_n$ and $\SL_n$ sufficient to prove the claim for $\calO = \bbZ$.
    The closed immersion $\GL_n \to M_n \times \bbA^1, ~g \mapsto (g, \det(g)^{-1})$ induces a surjection $\bbZ[(M_n \times \bbA^1)^m] \to \bbZ[\GL_n^m]$ of $\bbZ$-graded $\bbZ$-modules with $\GL_n$-action, where the graded pieces are finitely generated free $\bbZ$-modules. 

    Identifying $\bbZ[\bbA^1] = \bbZ[t]$, we have a short exact sequence
    $$ 0 \to \mathbb{Z}[M_n] \otimes \mathbb{Z}[t] \xrightarrow{\cdot (t \cdot \det - 1)} \mathbb{Z}[M_n] \otimes \mathbb{Z}[t] \to \mathbb{Z}[\mathrm{GL}_n] \to 0 $$
    of $\GL_d$-modules, since $t \cdot \det - 1$ is an invariant element of the integral domain $\bbZ[M_d] \otimes \bbZ[t]$. By \Cref{lemsurjtens} and since $\bbZ[\GL_d]$ and $\bbZ[M_d]$ have good filtrations (\Cref{gfG}, \Cref{gfM}), the maps $\bbZ[M_n^m]^{\GL_n} \otimes \bbZ[t]^{\otimes m} = \bbZ[(M_n \times \bbA^1)^m]^{\GL_n} \to \bbZ[\GL_n^m]^{\GL_n}$ are surjective. The claim follows from \Cref{DCPoverZ}.
    
    The same argument using the closed immersion $\SL_n \to M_n$ and the short exact sequence
    $$ 0 \to \mathbb{Z}[M_n] \xrightarrow{\cdot (\det - 1)} \mathbb{Z}[M_n] \to \mathbb{Z}[\mathrm{SL}_n] \to 0 $$
    implies the claim on $\bbZ[\SL_n^m]^{\SL_n}$.
\end{proof}

\begin{corollary}\label{gensforGLSL}
    Let $\calO$ be a commutative ring.
    \begin{enumerate}
        \item The $\calF$-$\calO$-algebra $\calO[\GL_n^{\bullet}]^{\GL_n}$ is generated by $s_1, \dots, s_n \in \calO[\GL_n]^{\GL_n}$.
        \item The $\calF$-$\calO$-algebra $\calO[\SL_n^{\bullet}]^{\SL_n}$ is generated by $s_1, \dots, s_{n-1} \in \calO[\SL_n]^{\SL_n}$.
    \end{enumerate}
\end{corollary}

\begin{proof}
    This follows by inspection of the generators computed in \Cref{invGLZ} and substitutions.
\end{proof}

\subsection{$\Sp_{2n}$, $\GSp_{2n}$, $\SO_{2n+1}$, $\OO_{2n+1}$ and $\GO_{2n+1}$} We start by recalling a theorem of Zubkov \cite[Theorem 1, Proposition 3.2]{zubkov99} on invariant rings over an algebraically closed field and extend it to general fields.

\begin{theorem}[Zubkov, 1999]\label{zubkovThm1} Let $K$ be an algebraically closed field. Let $G$ be either the symplectic group $\Sp_d$ over $K$ for even $d \geq 2$ or the orthogonal group $\OO_d$ over $K$ for $d \geq 1$ and assume, that $\chara(K) \neq 2$ in the orthogonal case. Let $m \geq 1$. Then:
\begin{enumerate}
    \item\label{zub1} $K[M_d^m]^G$ is generated as a $K$-algebra by elements of the form $\sigma_i(Y_{j_1} \cdots Y_{j_s})$ for $i \in \{1, \dots, d\}$ and $s \geq 0$, where $Y_k \in \{\bbX^{(k)}, (\bbX^{(k)})^*\}$ and in the orthogonal case $* = \top$ is transposition and in the symplectic case $* = \jj$ is symplectic transposition, i.e. $J(-)^{\top}J^{-1}$ for $J = \begin{pmatrix} 0 & \id \\ -\id & 0 \end{pmatrix}$.
    \item\label{zub2} The map $K[M_d^m]^G \to K[G^m]^G$ induced by restriction to $G^m \subseteq M_d^m$ is surjective. In particular, $K[G^m]^G$ is generated by elements of the form $\sigma_i(Y_{j_1} \cdots Y_{j_s})$ for $i \in \{1, \dots, d\}$ and $s \geq 0$, where $Y_k \in \{\bbX^{(k)}, (\bbX^{(k)})^{-1}\}$.
\end{enumerate}
\end{theorem}

\begin{remark}
    Zubkov proves \Cref{zubkovThm1} for algebraically closed fields and then remarks that the claim holds for all infinite fields by a Zariski density argument \cite[Remark 3.2]{zubkov99}.
\end{remark}

We extend Zubkov's \Cref{zubkovThm1} to arbitrary fields and to the groups $\GSp_{2n}$, $\SO_{2n+1}$ and $\GO_n$ when $n \geq 1$.

\begin{proposition}\label{invoverfield} Let $K$ be a field and let $m \geq 1$.
\begin{enumerate}
    \item Suppose $G \in \{\Sp_{2n}, \SO_{2n+1}, \OO_n\}$ and $d \in \{2n, 2n+1, n\}$ respectively. Assume further $\chara(K) \neq 2$ if $G \in \{\SO_{2n+1}, \OO_n\}$. Then $K[G^m]^{G}$ is generated by elements of the form $\sigma_i(Y_{j_1} \cdots Y_{j_s})$ for $i \in \{1, \dots, d\}$ and $s \geq 0$, where $Y_k \in \{\bbX^{(k)}, (\bbX^{(k)})^{-1}\}$.
    \item Suppose $G \in \{\GSp_{2n}, \GO_n\}$ and $d \in \{2n, n\}$ respectively. Assume further $\chara(K) \neq 2$ if $G=\GO_n$. Then $K[G^m]^{G}$ is generated by the symplectic (orthogonal) similitude character $\simil$, its inverse $\simil^{-1}$ and elements as in (1).
\end{enumerate}
\end{proposition}

\begin{proof} Let $K'$ be an algebraically closed overfield of $K$. We first treat the case $G \in \{\Sp_{2n}, \OO_{n}\}$. Let $d=2n$ in the symplectic case and $d=n$ in the orthogonal case. We have a short exact sequence
$$ 0 \to J \to K[M_d^m]^G \to K[G^m]^G \to 0 $$
of $K$-vector spaces, where $J := \ker(K[M_d^m]^G \to K[G^m]^G)$. By faithful flatness it suffices to show, that
$$ 0 \to J \otimes_K K' \to K[M_d^m]^G \otimes_K K' \to K[G^m]^G \otimes_K K' \to 0 $$
is exact. 
By the universal coefficient theorem for algebraic invariants \cite[I.4.18 Proposition]{Jantzen2003}, we have isomorphisms $K[M_d^m]^G \otimes_K K' \cong K'[M_d^m]^G$ and $K[G^m]^G \otimes_K K' \cong K'[G^m]^G$, so $J \otimes_K K'$ is the kernel of $K'[M_d^m]^G \to K'[G^m]^G$.
The claim follows from \Cref{zubkovThm1}. For $\SO_{2n+1}$, we note, that the map $K[\OO_{2n+1}]^{\OO_{2n+1}} \to K[\SO_{2n+1}]^{\SO_{2n+1}}$ is surjective, since $\OO_{2n+1} = \SO_{2n+1} \times \{\pm 1\}$.

For the rest of the proof, we argue as in \cite[Lemma 3.15]{weiss2021images}.
The natural surjection $\Sp_{2n} \times \GL_1 \to \GSp_{2n}$ induces an inclusion $K[\GSp_{2n}^m]^{\GSp_{2n}} \subseteq K[\Sp_{2n}^m]^{\Sp_{2n}} \otimes_{K} K[\GL_1^m]$. Here the second factor is generated by the symplectic similitude character $\simil_i$ of $X_i$ and its inverse. Since all generators on the right hand side are defined on the left hand side, the map is an isomorphism. For $\GO_{n}$ we argue the same way.
\end{proof}

Fix a rational prime $p$ and an integer $r \geq 1$.
We extend \Cref{invoverfield} to $p^r$-torsion coefficients by using the theory of good filtrations over $\bbZ$.
We can lift invariants using the following variant of Nakayama's lemma.

\begin{lemma}\label{prliftlemma} \phantom{a}
\begin{enumerate}
    \item Let $M$ be any $\bbZ/p^r$-module and assume $M/p = 0$. Then $M = 0$.\label{detect}
    \item Let $f : M \to N$ be a homomorphism of $\bbZ/p^r$-modules, such that $\overline f : M/p \to N/p$ is surjective. Then $f$ is surjective.
\end{enumerate}
\end{lemma}

\begin{proof} (1) We have $M = pM$, thus $M = p^rM = 0$. (2) We can apply \eqref{detect} to $\coker(f)$.
\end{proof}

\begin{lemma}\label{genliftlemma} Let $G$ be a Chevalley group and let $S \subseteq \bbZ[G^m]^G$ be a subset, that generates $\bbF_p[G^m]^G$ as a ring. Then $S$ generates $\bbZ/p^r[G^m]^G$.
\end{lemma}

\begin{proof} Let $A \subseteq \bbZ/p^r[G^m]^G$ be the subalgebra generated by $S$. By \Cref{gfG} $\bbZ[G^m]$ has a good filtration. We calculate
$$ \bbZ/p^r[G^m]^G \otimes_{\bbZ/p^r} \bbF_p = (\bbZ[G^m]^G \otimes_{\bbZ} \bbZ/p^r) \otimes_{\bbZ/p^r} \bbF_p = \bbZ[G^m]^G \otimes_{\bbZ} \bbF_p = \bbF_p[G^m]^G $$
applying \Cref{critchange} twice.
Hence the inclusion induces a surjection $A/p \twoheadrightarrow (\bbZ/p^r[G^m]^G)/p$. From \Cref{prliftlemma}, we obtain $A = \bbZ/p^r[G^m]^G$.
\end{proof}

\begin{proposition}\label{genZpr} Let $\calO$ be a commutative ring, such that $p^r \calO = 0$. Let $m \geq 1$ and assume $p > 2$ in the orthogonal cases.
\begin{enumerate}
    \item Let $n \geq 1$. Then $\calO[\Sp_{2n}^m]^{\Sp_{2n}}$ (resp. $\calO[\OO_{2n+1}^m]^{\SO_{2n+1}}$) is generated by elements of the form $\sigma_i(Y_{j_1} \cdots Y_{j_s})$ for $i \in \{1, \dots, d\}$ and $s \geq 0$, where $Y_k \in \{\bbX^{(k)}, (\bbX^{(k)})^{-1}\}$.
    \item Let $n \geq 1$. Then $\calO[\GSp_{2n}^m]^{\GSp_{2n}}$ (resp. $\calO[\GO_{n}^m]^{\GO_{n}}$) is generated by the symplectic (orthogonal) similitude character $\simil$, its inverse $\simil^{-1}$ and elements of the form $\sigma_i(Y_{j_1} \cdots Y_{j_s})$ for $i \in \{1, \dots, d\}$ and $s \geq 0$, where $Y_k \in \{\bbX^{(k)}, (\bbX^{(k)})^{-1}\}$.
\end{enumerate}
\end{proposition}

\begin{proof} Let $G \in \{\Sp_{2n}, \OO_{2n+1}, \GSp_{2n}, \GO_{n}\}$. Since by \Cref{gfG} all $\bbZ/p^r[(G^0)^m]$ have a good filtration, we may assume $\calO = \bbZ/p^r$. In all cases, the expected generators are defined as elements of $\bbZ[G^m]^{G^0}$. The claim now follows from \Cref{genliftlemma} and the generators of $\bbF_p[G^m]^G$ of \Cref{invoverfield}.
\end{proof}

\begin{corollary}\label{prtorsfingen} Let $r \geq 1$ be an integer and let $\calO$ be a commutative ring, such that $p^r \calO = 0$. Let $G \in \{\SL_n, \GL_n, \Sp_{2n}, \GSp_{2n}, \SO_{2n+1}, \OO_{2n+1}, \GO_{2n+1}\}$ and assume that $p > 2$ in the orthogonal cases. Then the maps $\calO[\GL_n^m]^{\GL_n} \to \calO[G^m]^{G^0}$ are surjective for all $m \geq 1$. In particular, the $\calF$-$\calO$-algebras $\calO[G^{\bullet}]^{G^0}$ are finitely generated.
\end{corollary}

\begin{proof} This follows from \Cref{gensforGLSL} and \Cref{genZpr} and substitutions.
\end{proof}

\section{Deformations of $G$-pseudocharacters}
\label{secdefGPC}

Fix a prime $p > 0$ and a profinite group $\Gamma$.
Let $\kappa$ be a field of one of the following three types:
\begin{enumerate}
    \item $\kappa$ is a finite discrete field.\label{kappa1}
    \item $\kappa$ is a finite extension of $\bbQ_p$ equipped with the $p$-adic topology.\label{kappa2}
    \item $\kappa$ is a finite extension of $\bbF_p((t))$ equipped with the $t$-adic topology.\label{kappa3}
\end{enumerate}

We introduce a coefficient ring $\Lambda$ in each of the three cases for $\kappa$.
\begin{enumerate}
    \item In case \eqref{kappa1}, let $\Lambda$ be the ring of integers of a $p$-adic local field with residue field $\kappa$.
    \item In case \eqref{kappa2}, let $\Lambda = \kappa$.
    \item In case \eqref{kappa3}, let $\Lambda = \kappa$.
\end{enumerate}
We will call only such rings $\Lambda$ \emph{coefficient rings} for $\kappa$.
We further fix a generalized reductive $\Lambda$-group scheme $G$ and a continuous pseudocharacter $\Thetabar \in \cPC_G^{\Gamma}(\kappa)$.

Let $\Art_{\Lambda}$ be the category of artinian local $\Lambda$-algebras with residue field $\kappa$.
Every $A$ in $\Art_{\Lambda}$ has a canonical projection $\pi_A : A \to \kappa$ with kernel $\frakm_A$ the maximal ideal of $A$.
Note, that $\Art_{\Lambda}$ admits fiber products \cite[§2.2]{MR1643682}.
Every complete local $\Lambda$-algebra $A$ with residue field $\kappa$ is algebraically isomorphic to the inverse limit $\varprojlim A/\frakm_A^n$.
If $A$ is a complete noetherian local $\Lambda$-algebra, it can be written as $A = \Lambda[[X_1, \dots, X_r]]/I$, where $r$ is the $\kappa$-dimension of the relative cotangent space $t_A^* = \frakm_A/(\frakm_A^2 + \frakm_{\Lambda} A)$ of $A$ \cite[Lem. 5.1]{MR1643682}.

\subsection{The universal deformation ring $\RpsThetabar$}

\begin{definition}\label{pseudodeffunctor} We define the \emph{deformation functor} of $\Thetabar$
$$ \Def_{\Thetabar} : \Art_{\Lambda} \to \Set, ~A \mapsto \{\Theta \in \cPC_G^{\Gamma}(A) \mid \Theta \otimes_A \kappa = \Thetabar\} $$
that sends an object $A \in \Art_{\Lambda}$ to the set of continuous $G$-pseudocharacters $\Theta$ of $\Gamma$ over $A$ with $\Theta \otimes_A \kappa = \Thetabar$.
\end{definition}

If $A$ is an arbitrary local topological $\Lambda$-algebra with residue field $\kappa$, we define $\Def_{\Thetabar}(A)$ analogously.
This is notation for a single $A$ and shall not extend the deformation functor $\Def_{\Thetabar}$.
To prove pro-representability of the deformation functor we need to show, that it preserves projective limits.

\begin{lemma}\label{proextension} Let $A = \varprojlim_i A_i$ be a projective limit of local topological $\Lambda$-algebras with $A_i \in \Art_{\Lambda}$, endowed with the projective limit topology. Then the natural map $\Def_{\Thetabar}(A) \to \varprojlim_i \Def_{\Thetabar}(A_i)$ is bijective.
\end{lemma}

\begin{proof} By definition, $\Def_{\Thetabar}(A) = \{\Thetabar\} \times_{\cPC_G^{\Gamma}(\kappa)} \cPC_G^{\Gamma}(A)$, so it suffices to prove the claim for $\cPC_G^{\Gamma}$ in place of $\Def_{\Thetabar}$. 
By \Cref{allemorphismen} and since $\calC(\Gamma^n, A) = \varprojlim_i \calC(\Gamma^n, A_i)$, we have
\begin{align*}
    \cPC_G^{\Gamma}(A) &= \Hom_{\CAlg_{\Lambda}^{\calF}}(\Lambda[G^{\bullet}]^{G^0}, ~\calC(\Gamma^{\bullet},A)) = \Hom_{\CAlg_{\Lambda}^{\calF}}(\Lambda[G^{\bullet}]^{G^0}, ~\varprojlim\nolimits_i \calC(\Gamma^{\bullet},A_i)) \\
    &= \varprojlim\nolimits_i \Hom_{\CAlg_{\Lambda}^{\calF}}(\Lambda[G^{\bullet}]^{G^0}, ~\calC(\Gamma^{\bullet},A_i)) = \varprojlim\nolimits_i \cPC_G^{\Gamma}(A_i).
\end{align*}
\end{proof}

We adapt the language of \cite[§4.7]{BJ_new} to our setting.

\begin{definition}
    Let $\psi : B \to \kappa$ be a map of commutative $\Lambda$-algebras and let $\Theta \in \PC_G^{\Gamma}(B)$, such that $\Theta \otimes_B \kappa = \Thetabar$. We say that an ideal $I$ of $B$ is \emph{$\Theta$-open}, if the following conditions hold:
    \begin{enumerate}
        \item $I$ is contained in $\ker(\psi)$.
        \item $B/I$ is artinian and local. If $\kappa$ is finite, we equip $B/I$ with the discrete topology. If $\kappa$ is a local field then $B/I$ is a finite-dimensional $\kappa$-vector space and we equip $B/I$ with the product topology of $\kappa$.
        \item The base extension $\Theta^I := \Theta \otimes_B B/I$ is a continuous $G$-pseudocharacter.
    \end{enumerate}
\end{definition}

\begin{theorem}\label{representabledeffunctor} The deformation functor $\Def_{\Thetabar} : \Art_{\Lambda} \to \Set$
is pro-representable by an inverse limit $\RpsThetabar$ of artinian $\Lambda$-algebras with residue field $\kappa$, endowed with the inverse limit topology. If $\kappa$ is finite, then $\RpsThetabar$ is pro-$p$ and in particular complete.
\end{theorem}

If $\Thetabar = \Theta_{\rhobar}$ is associated to a continuous representation $\overline{\rho} : \Gamma \to G(\kappa)$, we write $R_{\rhobar}^{\ps} := \RpsThetabar$. If $\kappa$ is a local field, we only have one choice for $\Lambda$ and we will usually drop it from notations.

\begin{proof} We adapt the proof of \cite[Proposition 3.3]{MR3444227}. Let $B := B^{\Gamma}_{G}$ be the $\Lambda$-algebra from \Cref{repofPC}, that represents $\PC^{\Gamma}_{G} : \CAlg_{\Lambda} \to \Set$. Let $\Theta^u \in \PC^{\Gamma}_{G}(B)$ be the universal $G$-pseudocharacter and $\psi : B \to \kappa$ the morphism, that corresponds to $\Thetabar$ under the identification $\Hom_{\CAlg_{\Lambda}}(B,\kappa) \cong \PC^{\Gamma}_G(\kappa)$. Let $\calI$ be the set of $\Theta$-open ideals of $B$. The set with relation $(\calI, \subseteq)$ is a cofiltered poset: If $I, J \in \calI$, then we have
\begin{enumerate}
    \item $I \cap J \subseteq \frakm$. 
    \item The map $\iota : B/(I \cap J) \to B/I \times B/J$ is injective, hence $B/(I \cap J)$ is artinian. Let $\frakm'$ be a maximal ideal of $B$, that contains $I \cap J$. Then $IJ \subseteq \frakm'$, hence either $I \subseteq \frakm'$ or $J \subseteq \frakm'$. In both cases $\frakm' = \frakm$, since $B/I$ and $B/J$ are local. Hence $B/(I \cap J)$ is local.
    \item Note, that $\iota$ is a topological embedding. Thus, for the reduction $\Theta^{I \cap J}$ of $\Theta^u$ mod $I \cap J$ the homomorphism $\Theta^{I \cap J}_n : B[G^n]^{G^0} \to \map(\Gamma^n, B/(I \cap J))$ has image in $\calC(\Gamma^n, B/(I \cap J))$ for all $n \geq 1$.
\end{enumerate}

Define the topological $\Lambda$-algebra $R := \varprojlim_{I \in \calI} B/I$.

The inverse limit is taken in the category of topological $\Lambda$-algebras.
Let $\pi_{R} : R \to \kappa$ be the map induced by the identification $B/\ker(\psi) \cong \kappa$ and let $\frakm_{R} := \ker(\pi_{R})$. Each $B/I$ is a local ring with residue field $\kappa$, so an element of $R$ is invertible if and only if its reduction to $\kappa$ is. This shows, that $R$ is local with maximal ideal $\frakm_{R}$.
If $\kappa$ is finite, then each $B/I$ is a finite $p$-group and $R$ is pro-$p$ and in particular complete.

We show, that $R$ pro-represents $\Def_{\Thetabar}$ and that $\Theta^u \otimes_B R \in \Def_{\Thetabar}(R)$ is the universal deformation of $\Thetabar$, where $\iota : B \to R$ is the canonical map. Assume for the proof, that $\Def_{\Thetabar}$ is defined on the category of local topological $\Lambda$-algebras with residue field $\kappa$. This way we get uniqueness of $R$ once we show representability. By \Cref{proextension} we have an isomorphism $\Def_{\Thetabar}(R) \cong \varprojlim_{I \in \calI} \Def_{\Thetabar}(B/I)$, so it suffices to show representability for artinian rings.

If $A \in \Art_{\Lambda}$ and $\Theta \in \Def_{\Thetabar}(A)$, then $\Theta$ corresponds to a unique homomorphism $\phi : B \to A$, such that $\Theta^u \otimes_B A = \Theta$ and $\phi \mod \frakm_A = \psi$. We will show, that $\ker(\phi) \in \calI$. We have $\ker(\phi) \subseteq \ker(\psi) = \frakm$ and $B/\ker(\phi) \subseteq A$ is artinian local.
We have to show, that $\Theta^u \otimes_B B/\ker(\phi)$ is continuous.
Indeed $(\Theta^u \otimes_B B/\ker(\phi)) \otimes_{B/\ker(\phi), \overline\phi} A = \Theta^u \otimes_B A = \Theta$ is continuous, where $\overline\phi : B/\ker(\phi) \to A$ is the map induced by $\phi$. Since $\overline\phi$ is a topological embedding $\Theta^u \otimes_B B/\ker(\phi)$ is continuous.
So there is a unique factorization $B \to R \to B/\ker(\phi) \to A$ of $\phi$ over a continuous map $R \to A$.

For the converse suppose, that $\varphi : R \to A$ is a continuous local $\Lambda$-homomorphism compatible with the projections to $\kappa$.
We have to show, that the pseudocharacter $(\Theta^u \otimes_B R) \otimes_{R, \varphi} A$ is continuous.
It is enough to show, that the universal deformation $\Theta^u \otimes_B R$ is continuous.
The pseudocharacters $(\Theta^u \otimes_B R) \otimes_R B/I = \Theta^u \otimes_B B/I$ are continuous by definition of $\calI$.
For fixed $m \geq 1$ and $f \in \Lambda[G^m]^{G^0}$ the map $(\Theta^u \otimes_B R)_m(f) : \Gamma^m \to R$ will be continuous by the universal property of limits.
\end{proof}

\subsection{Noetherianity for topologically finitely generated profinite groups}\label{subsecnoethTFG} In this subsection we will complete the proof of part \eqref{B1} of \Cref{ThmA} by showing, that the universal pseudodeformation ring of a continuous residual $G$-pseudocharacter over a finite field is noetherian when $\Gamma$ is a topologically finitely generated profinite group.

\begin{proposition}\label{hidaslemma} Assume that $\kappa$ is finite and let $A$ be a pro-$p$ local $\Lambda$-algebra with residue field $\kappa$. The following are equivalent:
\begin{enumerate}
    \item $A$ is noetherian. \label{hida_1}
    \item $\frakm_A$ is a finitely generated ideal. \label{hida_2}
    \item $\frakm_A/\frakm_A^2$ is a finite-dimensional $\kappa$-vector space. \label{hida_3}
    \item $\frakm_A/(\frakm_{\Lambda}+ \frakm_A^2)$ is a finite-dimensional $\kappa$-vector space. \label{hida_4}
\end{enumerate}
\end{proposition}

\begin{proof} $\eqref{hida_1} \Rightarrow \eqref{hida_2} \Rightarrow \eqref{hida_3} \Rightarrow \eqref{hida_4}$ is clear. The proof of $\eqref{hida_4} \Rightarrow \eqref{hida_1}$ can be found in Hida's notes \cite[Lemma 2.10]{hidaNotes}.
\end{proof}

\begin{proposition}\label{charcrhoquer} Assume, that $\kappa$ is finite. Then the following are equivalent:
\begin{enumerate}
    \item $\dim_{\kappa}(\Def_{\Thetabar}(\kappa[\varepsilon])) < \infty$.
    \item $\RpsThetabar$ is a noetherian ring.
\end{enumerate}
\end{proposition}

\begin{proof} Since $\RpsThetabar$ represents $\Def_{\Thetabar}$ (\Cref{representabledeffunctor}), the relative tangent space $(\frakm_{\RpsThetabar}/(\frakm_{\Lambda}+\frakm_{\RpsThetabar}^2))^*$ of $\RpsThetabar$ over $\Lambda$ identifies with $\Def_{\Thetabar}(\kappa[\varepsilon])$. 
Since $\RpsThetabar$ is pro-$p$, the claim follows from \Cref{hidaslemma}.
\end{proof}

\begin{theorem}\label{tfgfingen} Assume that $\kappa$ is finite and $\Gamma$ is topologically finitely generated. Then $\RpsThetabar$ is noetherian.
\end{theorem}

\begin{proof} Let $\Delta \subseteq \Gamma$ be a dense and finitely generated subgroup of $\Gamma$. We have a sequence of injections
$$ \Def_{\Thetabar}(\kappa[\varepsilon]) \subseteq \cPC^{\Gamma}_G(\kappa[\varepsilon]) \overset{\eqref{restrictiontoadensesubgroup}}{\subseteq} \cPC^{\Delta}_G(\kappa[\varepsilon]) \subseteq \PC^{\Delta}_G(\kappa[\varepsilon]) \overset{\eqref{repofPC}}{\cong} \Hom_{\Lambda}(B_G^{\Delta}, \kappa[\varepsilon]) $$
By \Cref{decisivefiniteness}, $\Hom_{\Lambda}(B_G^{\Delta}, \kappa[\varepsilon])$ is a finite-dimensional $\kappa$-vector space.
By \Cref{charcrhoquer} we conclude, that $\RpsThetabar$ is noetherian.
\end{proof}

\subsection{Noetherianity for profinite groups satisfying $\Phi_p$}\label{subsecnoethPhip} The idea of establishing noetherianity of the pseudodeformation rings $\RpsThetabar$ in case we only know that $\Gamma$ satisfies Mazur's condition $\Phi_p$ is to prove surjectivity of the map $R^{\ps}_{\iota(\Thetabar)} \to \RpsThetabar$ for a suitable algebraic representation $\iota : G \to \GL_n$, and use noetherianity of $R^{\ps}_{\iota(\Thetabar)}$ proved by Chenevier \cite[Proposition 3.7]{MR3444227}. In this section we establish a criterion in terms of invariant theory, which can be applied to other reductive groups once their invariant theory is understood.

\begin{lemma}\label{invsurjBsurj} Let $\Gamma$ be a group and let $\iota : G \to G'$ be a homomorphism of generalized reductive group schemes over a commutative ring $\calO$. Suppose, that the map $\calO[{G'}^{\bullet}]^{{G'}^0} \to \calO[G^{\bullet}]^{G^0}$ of $\calF$-$\calO$-algebras is surjective. Then the map $\iota^* : B_{G'}^{\Gamma} \to B_{G}^{\Gamma}$ induced by $\iota$ is surjective.
\end{lemma}

\begin{proof} By \Cref{repofPC} it is enough to show, that for each $m \geq 1$, each $\mu \in \calO[G^{m}]^{G^0}$ and each $\gamma = (\gamma_1, \dots, \gamma_{m}) \in \Gamma^{m}$, the element $t_{\mu, \gamma} \in B_{G}^{\Gamma}$ has a preimage in $B_{G'}^{\Gamma}$. By surjectivity of $\calO[{G'}^{m}]^{{G'}^0} \to \calO[G^{m}]^{G^0}$, we find some $\mu' \in \calO[{G'}^{m}]^{{G'}^0}$ mapping to $\mu$. We claim, that $\iota^*(t_{\mu', \gamma}) = t_{\mu, \gamma}$. Let $A \in \CAlg_{\calO}$, $\Theta \in \PC_G^{\Gamma}(A)$ and $f_{\Theta} : B_G^{\Gamma} \to A$ the homomorphism attached to $\Theta$. Let $f_{\iota(\Theta)} : B_{G'}^{\Gamma} \to A$ be the homomorphism attached to $\iota(\Theta)$. By definition $f_{\Theta}(\iota^*(t_{\mu',\gamma})) = f_{\iota(\Theta)}(t_{\mu',\gamma}) = \iota(\Theta)_m(\mu')(\gamma) = \Theta_m(\gamma)$. Since this characterizes $\iota^*(t_{\mu',\gamma})$ uniquely, we have $\iota^*(t_{\mu',\gamma}) = t_{\mu, \gamma}$.
\end{proof}

\begin{lemma}\label{surjectivity} Assume that $\kappa$ is finite and let $\iota : G \to G'$ be a homomorphism of generalized reductive $\Lambda$-group schemes. Let $\Thetabar \in \cPC^{\Gamma}_G(\kappa)$ be a continuous pseudocharacter and we denote by $\iota(\Thetabar)$ its image in $\cPC^{\Gamma}_{G'}(\kappa)$. Assume, that the homomorphism $B_{G'}^{\Gamma}/\varpi \to B_{G}^{\Gamma}/\varpi$ is surjective, where $\varpi$ is the uniformizer of $\Lambda$ and that $R^{\ps}_{\iota(\Thetabar)}$ is noetherian. Then the natural homomorphism $R^{\ps}_{\iota(\Thetabar)} \to \RpsThetabar$ is surjective.
\end{lemma}

\begin{proof}
    By Nakayama's lemma it is enough to show, that the natural map $R^{\ps}_{\iota(\Thetabar)}/\varpi \to \RpsThetabar/\varpi$ induced by $\iota$ is surjective.
    Since $B_{G'}^{\Gamma}/\varpi \to B_{G}^{\Gamma}/\varpi$ is surjective every $\Theta \in \Def_{\iota(\Thetabar)}(A)$ for which $\Theta_m : \Lambda[(G')^m]^{(G')^0} \to \calC(\Gamma^m, A)$ factors over $\Lambda[G^m]^{G^0}$ for all $m \geq 1$ comes from a continuous deformation of $\Thetabar$. So we have $\Def_{\Thetabar}(A) = \Def_{\iota(\Thetabar)}(A) \times_{\PC_{G'}^{\Gamma}(A)} \PC_G^{\Gamma}(A)$ for any $\kappa$-algebra $A \in \Art_{\Lambda}$. Hence $\RpsThetabar/\varpi \cong R^{\ps}_{\iota(\Thetabar)}/\varpi \otimes_{B_{G'}^{\Gamma}/\varpi} B_{G}^{\Gamma}/\varpi$ and the claim follows.
\end{proof}

\begin{theorem}\label{Phipmainthm} Assume that $\kappa$ is finite and that $\Gamma$ satisfies Mazur's condition $\Phi_p$. Let $\iota : G \to \GL_d$ be an algebraic representation of $G$ and assume, that for $m \geq 1$ the natural maps $\Lambda/p[\GL_d^m]^{\GL_n} \to \Lambda/p[G^m]^{G^0}$ are surjective.
Then the canonical map $R^{\ps}_{\iota(\Thetabar)} \to \RpsThetabar$ is surjective and $\RpsThetabar$ is noetherian.
\end{theorem}

\begin{proof} It follows from \Cref{invsurjBsurj}, that the maps $B^{\Gamma}_{\GL_{d,\Lambda/p}} \to B^{\Gamma}_{G_{\Lambda/p}}$ are surjective. By \Cref{basechange}, we have surjections $B^{\Gamma}_{\GL_d}/p \to B^{\Gamma}_G/p$. We use Chenevier's result \cite[Proposition 3.7]{MR3444227} to see, that $R^{\ps}_{\iota(\Thetabar)}$ is noetherian. Hence we can apply \Cref{surjectivity} and see, that the map $R^{\ps}_{\iota(\Thetabar)} \to \RpsThetabar$ is surjective, hence $\RpsThetabar$ is noetherian as well.
\end{proof}

In case $\Thetabar$ comes from a representation, we can control the tangent space of the pseudodeformation ring under a cohomological finiteness condition, using another result of Chenevier \cite[Proposition 2.35]{MR3444227}.

\begin{corollary}\label{H1mainthm} Assume that $\kappa$ is finite and that $\Thetabar = \Theta_{\rhobar}$ for a continuous representation $\rhobar : \Gamma \to G(\kappa)$. Let $\iota : G \to \GL_d$ be an algebraic representation of $G$ and assume, that for $m \geq 1$ the natural maps $\Lambda/p[\GL_d^m]^{\GL_n} \to \Lambda/p[G^m]^{G^0}$ are surjective, that $\dim_{\kappa} H^1(\Gamma, \ad(\iota(\rhobar))) < \infty$ and $p > d$.
Then $\RpsThetabar$ is noetherian.
\end{corollary}

\begin{proof}
    By \cite[Proposition 2.35, Proposition 3.7]{MR3444227} $R^{\ps}_{\iota(\Thetabar)}$ is noetherian.
    We obtain a surjection $R^{\ps}_{\iota(\Thetabar)} \twoheadrightarrow \RpsThetabar$ from \Cref{Phipmainthm}.
\end{proof}

\begin{remark}\label{Rpssurj} Let ${G \in \{\SL_n, \GL_n, \Sp_{2n}, \GSp_{2n}, \SO_{2n+1}, \OO_{2n+1}, \GO_n\}}$ over a coefficient ring $\Lambda$ with finite residue field $\kappa$ and assume $p > 2$ in the orthogonal cases. For the standard representation $\iota : G \hookrightarrow \GL_d$, we have shown in \Cref{prtorsfingen}, that for $m \geq 1$ the natural maps $\Lambda/p[\GL_d^m]^{\GL_n} \to \Lambda/p[G^m]^{G^0}$ are surjective, so \Cref{Phipmainthm} and \Cref{H1mainthm} apply.
\end{remark}

\begin{remark}
    One might ask, whether a representation $\iota$ as assumed in \Cref{Phipmainthm} and \Cref{H1mainthm} always exists. For the purpose of this remark let us assume that $G$ is a Chevalley group over $\bbZ$. We don't expect such a representation $\iota$ to exist for arbitrary $G$: It follows from \Cref{gensforGLSL}, that if there is a representation $\iota : G \to \GL_d$ with $\bbZ[\GL_d^m]^{\GL_d} \to \bbZ[G^m]^G$ surjective for all $m \geq 0$ then the $\calF$-$\bbZ$-algebra $\bbZ[G^{\bullet}]^G$ is generated by $\bbZ[G^1]^G$ (just take $f \in \bbZ[G^m]^G$, lift it to $\bbZ[\GL_d^m]^{\GL_d}$ and write it as as sums and products of invariants coming from $\bbZ[\GL_d^1]^{\GL_d}$ by a suitable substitution). For many $G$ the group $G(\bbC)$ satisfies the following property: We say that $G(\bbC)$ is \emph{acceptable}, if any two homomorphisms $\rho, \rho' : \Gamma \to G(\bbC)$ are conjugate if and only if for each $\gamma \in \Gamma$ the elements $\rho(\gamma), \rho'(\gamma)$ are conjugate by some $A_{\gamma} \in G(\bbC)$ which may depend on $\gamma$.
    If $\bbZ[G^{\bullet}]^G$ is generated by $\bbZ[G^1]^G$, then two pseudocharacters $\Theta_{\rho}, \Theta_{\rho'} \in \PC_G^{\Gamma}(\bbC)$ are determined by the homomorphisms $(\Theta_{\rho})_1, (\Theta_{\rho'})_1 : \bbZ[G]^G \to \map(\Gamma, \bbC)$, which are in turn defined as $(\Theta_{\rho})_1(f)(\gamma) := f(\rho(\gamma))$, $(\Theta_{\rho'})_1(f)(\gamma) := f(\rho'(\gamma))$. If we now assume, that $\rho(\gamma), \rho'(\gamma)$ are conjugate for all $\gamma \in \Gamma$, it follows from invariance of $f$, that $(\Theta_{\rho})_1(f)(\gamma) = (\Theta_{\rho'})_1(f)(\gamma)$ for all $\gamma$ and all $r$ and hence $\Theta_{\rho} = \Theta_{\rho'}$. By the uniqueness part of \Cref{reconstructiongeneral} this implies, that $\rho$ and $\rho'$ are conjugate. So if $\bbZ[G^{\bullet}]^G$ is generated by $\bbZ[G^1]^G$, then $G(\bbC)$ is already acceptable.
    Acceptable Lie groups $G(\bbC)$ have been studied in \cite{Larsen1994} and there are groups $G$ for which it is known that $G(\bbC)$ is unacceptable:
    The groups $\SO_{2n}(\bbC)$ for $n \geq 3$ are unacceptable \cite[§4.3]{Weidner} and recently Chenevier and Gan \cite{chenevier2023rm} showed that also $\Spin(7)$ is unacceptable. We expect a close relationship between acceptability of $G(\bbC)$ and the condition of $\bbZ[G^{\bullet}]^G$ being generated by $\bbZ[G^1]^G$. Based on the current state of knowledge of invariant theory we cannot rule out that these two conditions are equivalent.
    We see this as strong evidence that for some groups $G$, the $\calF$-$\bbF_p$-algebra $\bbF_p[G^{\bullet}]^G$ is not generated by $\bbF_p[G^1]^G$ although we would have to argue with unacceptability of $G(\overline \bbF_p)$ in place of $G(\bbC)$ as above, as we don't know whether the cokernel of $\bbZ[\GL_d^1]^{\GL_d} \to \bbZ[G^1]^G$ has $p$-torsion. To our knowledge these questions have not been studied in characteristic $p$.
    Still we think, that \Cref{Phipmainthm} and \Cref{H1mainthm} do apply to a large class of groups $G$ and a list of possible candidates is apparent from \cite{Larsen1994}.
\end{remark}

\subsection{Comparison with Chenevier's construction}
\label{subsecCompChen}

\subsubsection{Determinant laws}
\label{secdeterminants}

Chenevier's definition of pseudorepresentations relies on the idea of a generalized determinant. He uses the language of polynomial laws. We recall the basic definitions.

Let $A$ be a commutative ring and $M$ and $N$ be $A$-modules. The association $B \mapsto M \otimes_A B$ defines a functor $\underline{M} : \CAlg_A \to \Set$ on the category of commutative $A$-algebras $\CAlg_A$. Any $A$-linear map $f : M \to N$ gives rise to a natural transformation $\underline{M} \to \underline{N}$.

\begin{definition} Let $A$ be a commutative ring and $M$ and $N$ be two $A$-modules.
\begin{enumerate}
    \item[(i)] An \emph{$A$-polynomial law} $P : M \to N$ is a natural transformation of functors $f : \underline{M} \to \underline{N}$.
    \item[(ii)] An $A$-polynomial law $P : M \to N$ is \emph{homogeneous} of degree $n \geq 0$, if $P_B(bx) = b^nP_B(x)$ for all $B \in \CAlg_A$, all $b \in B$ and all $x \in M \otimes_A B$.
    \item[(iii)] When $M = R$ is a unital (not necessarily commutative) associative $A$-algebra, then $P : R \to N$ is \emph{multiplicative}, if $P_B(1) = 1$ and $P_B(xy) = P_B(x)P_B(y)$ for all $B\in \CAlg_A$ and all $x,y \in R \otimes_A B$.
\end{enumerate}
\end{definition}

\begin{definition} Let $A$ be a commutative ring, $R$ be any $A$-algebra and $d \geq 1$ an integer. A \emph{$d$-dimensional $A$-valued determinant law} is a multiplicative $A$-polynomial law $D : R \to A$, that is homogeneous of degree $d$.
\end{definition}

In this text, we are going to consider determinant laws only when $R = A[\Gamma]$ is a group algebra. We denote the set of $A$-valued determinant laws $D : A[\Gamma] \to A$ by $\Det_d^{\Gamma}(A)$.

If $\rho : \Gamma \to \GL_d(A)$ is a representation, then we obtain a determinant law $D_{\rho} \in \Det_d^{\Gamma}(A)$ by first extending $\rho$ to a homomorphism of $A$-algebras $\rho : A[\Gamma] \to M_d(A)$ and the setting $D_{\rho, B}(r) := \det(\rho(r))$ for all $r \in B[\Gamma]$.
This defines a map $\Rep^{\Gamma, \square}_{\GL_d}(A) \to \Det_d^{\Gamma}(A)$, which is natural in $A$ and $\Gamma$.

\begin{definition} Let $D$ be an $A$-linear $d$-dimensional determinant law. We define the coefficients $\Lambda_i : R \to A$ of the characteristic polynomial of $D$ by the expansion
$$ \chi^D(r,t)=D_{B[t]}(t-r)=\sum_{i=0}^{d}(-1)^i \Lambda_{i,B}(r)t^{d-i} \in B[t]$$
for all $B \in \CAlg_A$. 
\end{definition}

One can show, that the coefficients $\Lambda_i$ give rise to $i$-homogeneous $A$-polynomial laws.

\subsubsection{Emerson's bijection}

Kathleen Emerson has proven in her 2018 dissertation \cite{Emerson2018ComparisonOD}, see also \cite{emerson2023comparison}, that there is a bijection between $\GL_d$-valued pseudocharacters and $d$-dimensional determinant laws over any base ring. We summarize her results and adapt them to the continuous case in \Cref{seccontemerson}. In this section we consider $\GL_d$ as a group scheme over $\bbZ$.

\begin{theorem}\label{chenevierlafforguetotal} Let $A$ be a commutative ring, $\Gamma$ a group and $d \geq 1$. Then the map
$$ \PC_{\GL_d}^{\Gamma}(A) \to \Det_{d}^{\Gamma}(A), \quad \Theta \mapsto D_{\Theta} $$
defined in \cite[Theorem 4.1 (ii)]{emerson2023comparison} is bijective.
\end{theorem} 

Emerson's bijection is uniquely determined by the following property: If $\sigma_i$ for $1 \leq i \leq d$ are the coefficients of the characteristic polynomial of a matrix in $\GL_d$ viewed as elements of $\bbZ[\GL_d]^{\GL_d}$, then a $\GL_d$-pseudocharacter $\Theta \in \PC_{\GL_d}^{\Gamma}(A)$ corresponds to a $d$-dimensional determinant law $D \in \Det_d^{\Gamma}(A)$ if and only if $\Lambda_{i, A}(\gamma) = \Theta_1(\sigma_i)(\gamma)$ for all $\gamma \in \Gamma$.

In particular, if $\rho : \Gamma \to \GL_d(A)$ is a representation, then $D_{\Theta_{\rho}} = D_{\rho}$.

\subsubsection{Continuous determinant laws}
\label{seccontemerson}

Let $\Gamma$ be a topological group and let $A$ be a topological ring. We say that a $d$-dimensional $A$-linear determinant law $D \in \Det_d^{\Gamma}(A)$ is \emph{continuous}, if the coefficients $\Lambda_i$ of the characteristic polynomial of $D$ give rise to continuous maps $\Lambda_{i,A}|_{\Gamma} : \Gamma \to A$. This notion of continuity is equivalent to that defined in \cite[§2.30]{MR3444227}. We denote the set of continuous $d$-dimensional $A$-linear determinant laws by $\cDet_d^{\Gamma}(A)$.

If $\rho : \Gamma \to \GL_d(A)$ is a continuous representation, then $D_{\rho}$ is a continuous determinant law. So we have a map $\cRep^{\Gamma, \square}_{\GL_d}(A) \to \cDet^{\Gamma}_d(A)$, which is natural in $A$ and $\Gamma$.

\begin{proposition}\label{chenevierlafforguecont} A $G$-pseudocharacter $\Theta \in \PC_{\GL_d}^{\Gamma}(A)$ is continuous if and only if $D_{\Theta}$ is continuous. In particular, the bijection $\PC_{\GL_d}^{\Gamma}(A) \to \Det_{d}^{\Gamma}(A), ~\Theta \mapsto D_{\Theta}$ in \Cref{chenevierlafforguetotal} restricts to a bijection $\cPC_{\GL_d}^{\Gamma}(A) \to \cDet_{d}^{\Gamma}(A)$.
\end{proposition}

\begin{proof} First suppose, that $\Theta$ is continuous. Then $\Lambda_{i,A}|_{\Gamma} = \Theta_1(\sigma_i)$ is a continuous map by definition of continuity of $\Theta$, hence $D_{\Theta}$ is continuous. 
Conversely, if $D_{\Theta}$ is continuous, then $\Theta_1(\sigma_i)$ is continuous for all $1 \leq i \leq d$. Since the $\calF$-$\bbZ$-algebra $\bbZ[\GL_d^{\bullet}]^{\GL_d}$ is generated by $\{s_1, \dots, s_d\}$ and $\det^{-1} = s_d^{-1}$ (see \Cref{invGLZ}), the image of $\bbZ[\GL_d^{\bullet}]^{\GL_d}$ is contained $\calC(\Gamma^{\bullet}, A)$, as desired.
\end{proof}

When $\Dbar$ is a continuous $\kappa$-valued determinant law, we can define a deformation functor $\Def_{\Dbar} : \Art_{\Lambda} \to \Set$ analogous to \Cref{pseudodeffunctor}. We refer to \cite[§3.1]{MR3444227} for a more extensive discussion of deformations of determinant laws. If $\kappa$ is finite, it is proved in \cite[Proposition 3.3]{MR3444227}, that $\Def_{\Dbar}$ is pro-representable by a local topological ring $R^{\ps}_{\Dbar}$; the result holds for all $\kappa$ we consider by the same arguments. Indeed the deformation theory of determinant laws is equivalent to the deformation theory of the corresponding $\GL_n$-pseudocharacter: 

\begin{corollary}\label{Rpsisom} Let $D_{\Thetabar}$ be the determinant law attached to $\Thetabar$ by \Cref{chenevierlafforguetotal}. Then the natural transformation of \Cref{chenevierlafforguecont} restricts to a natural bijection $\Def_{\Thetabar} \to \Def_{D_{\Thetabar}}$. In particular, there is a canonical isomorphism $\RpsThetabar \cong R^{\ps}_{D_{\Thetabar}}$ of universal pseudodeformation rings.
\end{corollary}

\begin{proof} This follows from \Cref{representabledeffunctor} and \Cref{chenevierlafforguecont}.
\end{proof}

\subsection{\texorpdfstring{Comparing deformations of representations and their $G$-pseudocharacters}{Comparing deformations of representations and their G-pseudocharacters}}
\label{subsecdefComp}

We now fix a continuous representation $\rhobar : \Gamma \to G(\kappa)$ and assume that $\Thetabar = \Theta_{\rhobar}$.
The main purpose of this section is to compare the unframed deformation functor of $\rhobar$ to the deformation functor of $\Thetabar$. We first extend \cite[Theorem 4.10]{BHKT} to local residue fields.

\begin{proposition}\label{compprop} Assume, that $Z_{G_{\kappa}}(\rhobar)=Z_{G_{\kappa}}(G_{\kappa}^0)$ and that $\rhobar$ is $G_{\kappa}$-completely reducible. Then the natural map of deformation functors $\Def_{\rhobar} \to \Def_{\Thetabar}$ is an isomorphism.
\end{proposition}

\begin{proof} Let $A \in \Art_{\Lambda}$ and $\Theta \in \Def_{\Thetabar}(A)$. 
For any $n \geq 1$, we define affine $\Lambda$-schemes of finite type $X_n := G^n$ and $Y_n := G^n \sslash G^0 =  G^n \sslash (G^0 /Z_G(G^0))$ and let $\pi : X_n \to Y_n$ be the projection.

Now fix $n \geq 1$ and $\gamma_1, \dots, \gamma_n \in \Gamma$, such that the scheme-theoretic centralizer $Z_{G_{\kappa}}(x)$ of $x := (\overline g_1, \dots, \overline g_n)$, where $\overline g_i := \overline \rho(\gamma_i)$, in $G_{\kappa}$ coincides with the scheme-theoretic centralizer $Z_{G_{\kappa}}(\rhobar)$ of $\rhobar$ in $G_{\kappa}$. This is possible, as $\kappa[G]$ is a noetherian ring. Thus $Z_{G_{\kappa}}(x)=Z_{G_{\kappa}}(G_{\kappa}^0)$ by assumption. We may assume by \cite[Lemma 9.2]{MartinGeneratingTuples}, that the subgroup generated by $\rhobar(\gamma_1), \dots, \rhobar(\gamma_n)$ has the same Zariski closure as $\rhobar(\Gamma)$, we denote this Zariski closure in $G(\kappa)$ by $H$. Since $\overline{\rho}$ is $G_{\kappa}$-completely reducible, by \Cref{charorbitcr} the orbit of $x$ in $X_{n, \kappa}$ is closed.

The completion of $X_n$ at $x \in X_n(\kappa)$ pro-represents the functor $X_n^{\wedge, x} : \Art_{\Lambda} \to \Set$ defined by $X_n^{\wedge, x}(A) := X_n(A) \times_{X_n(\kappa)} \{x\}$.
Similarly, for fixed $h \in H$, we have a completion of $X_{n+1}$ at $y := (x, h) \in X_{n+1}(\kappa)$ and the respective completions of $Y_n$ and $Y_{n+1}$ at $\pi(x)$ and $\pi(y)$.

In analogy to the completion at a point, we define the functor $X_{n+1}^{\wedge, H} : \Art_{\Lambda} \to \Set$ by $X_{n+1}^{\wedge, H}(A) := X_{n+1}(A) \times_{X_{n+1}(\kappa)} H$, where the map $H \to X_{n+1}(\kappa)$ is given by $h \mapsto (g_1, \dots, g_n,h)$.
Similarly we define $Y_{n+1}^{\wedge, H}(A) := Y_{n+1}(A) \times_{Y_{n+1}(\kappa)} H$.
We take the fiber product topology on the point sets.
We will use these functors to prove continuity of the representation we construct.
One can think of $X_{n+1}^{\wedge, H}$ as putting the completions at single points of $H$ into a continuous family.

The $G$-pseudocharacter $\Theta_{n+1}$ determines a natural map
\begin{align*}
    \Lambda[G^{n+1}]^{G^0} \to \calC(\Gamma, A), \quad f \mapsto (\gamma \mapsto \Theta_{n+1}(f)(\gamma_1, \dots, \gamma_n, \gamma)),
\end{align*}
which is an element $\alpha \in Y_{n+1}(\calC(\Gamma, A)) = \calC(\Gamma, Y_{n+1}(A))$. We obtain a unique continuous map $\beta : \Gamma \to Y_{n+1}^{\wedge, H}(A)$ induced by $\alpha$ and $\rhobar$.

Let $(G^0/Z_G(G^0))^{\wedge}$ be the completion of $G^0/Z_G(G^0)$ at the neutral element.
It is a formal $\Lambda$-scheme, which pro-represents a group functor on $\Art_{\Lambda}$.
By \cite[Proposition 3.13]{BHKT} applied for the action of $G^0/Z_G(G^0)$ on $X_n$, $(G^0/Z_G(G^0))^{\wedge}$ acts freely on $X_n^{\wedge, x}$ and the projection $X_n^{\wedge, x} \to Y_n^{\wedge, \pi(x)}$ factors through an isomorphism $X_n^{\wedge, x}/(G^0/Z_G(G^0))^{\wedge} \cong Y_n^{\wedge, \pi(x)}$. In particular, $X_n^{\wedge, x}(A) \to Y_n^{\wedge, \pi(x)}(A)$ is surjective and we can choose a preimage $(g_1, \dots, g_n) \in X_n^{\wedge, x}(A)$ of the point in $Y_n^{\wedge, \pi(x)}(A)$ determined by $\Lambda[G^n]^{G^0} \to A, ~f \mapsto \Theta_n(f)(\gamma_1, \dots, \gamma_n)$.

For fixed $h \in H$ and $y := (x, h)$, we have two cartesian squares:
\begin{center}
    \begin{tikzcd}
        X_{n+1}^{\wedge, y}(A) \arrow[r] \arrow[d] & Y_{n+1}^{\wedge, \pi(y)}(A) \arrow[r] \arrow[d] & \{h\} \arrow[d] \\
        X_{n+1}^{\wedge, H}(A) \arrow[r] & Y_{n+1}^{\wedge, H}(A) \arrow[r] & H
    \end{tikzcd}
\end{center}
The top left arrow is a $(G^0/Z_G(G^0))^{\wedge}(A)$-torsor of sets, so $X_{n+1}^{\wedge, H}(A) \to Y_{n+1}^{\wedge, H}(A)$ is a $(G^0/Z_G(G^0))^{\wedge}(A)$-torsor.
The square in the following diagram is cartesian, since the horizontal arrows are $(G^0/Z_G(G^0))^{\wedge}(A)$-torsors and the vertical maps are equivariant: 
\begin{center}
    \begin{tikzcd}[row sep=2.5em, column sep=2.5em]
        \Gamma \arrow[r, dashed] \arrow[dr] \arrow[rr, bend left=30, "\beta"] & X_{n+1}^{\wedge, H}(A) \arrow[r] \arrow[d] & Y_{n+1}^{\wedge, H}(A) \arrow[d] \\
        & X_{n}^{\wedge, x}(A) \arrow[r] & Y_{n}^{\wedge, \pi(x)}(A)
    \end{tikzcd}
\end{center}

The map $\Gamma \to X_{n}^{\wedge, x}(A)$ maps constantly to $(g_1, \dots, g_n)$.
By the discussion of the topologies on point sets in the introduction, the diagram is also cartesian in the category of topological spaces.
Again by the universal property, we obtain a continuous map $\Gamma \to X_{n+1}^{\wedge, H}(A)$.

The composition $\Gamma \to X_{n+1}^{\wedge, H}(A) \to X_{n+1}(A) \overset{\pr_{n+1}}{\to} G(A)$ defines the desired continuous $\rho$ with $\Theta_{\rho} = \Theta$ as in \cite[Theorem 4.10]{BHKT}.
\end{proof}

The following Proposition ensures that the reconstructed representations are defined over finite extensions, which will be enough for the proof of \Cref{nspclregular}.

\begin{proposition}\label{recwithfiniteextn} Let $\Gamma$ be a profinite group, let $\kappa$ be a finite or a local field, let $\Lambda$ be a coefficient ring for $\kappa$ and let $G$ be a generalized reductive group over $\Lambda$. Suppose $\Theta \in \cPC^{\Gamma}_G(\calO_{\kappa})$ is a continuous pseudocharacter, where $\calO_{\kappa} = \kappa$ if $\kappa$ is finite. If $\kappa$ is a local field of positive characteristic, assume that $\Gamma$ is topologically finitely generated, $G$ is a Chevalley group over $\Lambda$ and that the reduction $\Thetabar$ of $\Theta$ to the residue field $k$ of $\kappa$ comes from a $G$-completely reducible representation $\rhobar : \Gamma \to G(k')$ for some finite extension $k'/k$, which has scheme-theoretically trivial centralizer in $G^{\ad}_{k'}$. Then there exists a continuous ($G$-completely reducible) representation $\rho : \Gamma \to G(\overline\kappa)$ with $\Theta_{\rho} = \Theta$, which is defined over the ring of integers $\calO_{\kappa'}$ of a finite extension $\kappa'/\kappa$.
\end{proposition}

\begin{proof}
    Existence of a continuous $G$-completely reducible $\rho : \Gamma \to G(\overline\kappa)$ with $\Theta_{\rho} = \Theta$ follows from the continuous reconstruction theorem \Cref{continuousreconstruction}. If $\kappa$ is finite the conclusion is clear by continuity.
    If $\kappa$ is a local field of characteristic $0$, we can use the standard Baire category theorem argument to see that $\rho$ lies in a finite extension of $\bbQ_p$.
    Existence of $\calO_{\kappa'}$ follows from \cite[Lemma 4.3]{cotner}.

    Assume $\kappa$ is a local field of positive characteristic. Let $k$ be the residue field of $\calO_{\kappa}$. By assumption the reduction $\Thetabar$ of $\Theta$ to $k$ comes from a continuous $G$-completely reducible representation $\rhobar : \Gamma \to G(k')$ over a finite extension $k'/k$, which has scheme-theoretically trivial centralizer in $G^{\ad}_{k'}$. Choose a finite extension $\kappa'/\kappa$, such that the residue field of $\calO_{\kappa'}$ is $k'$. So $\Theta \otimes_{\calO_{\kappa}} \calO_{\kappa'}$ is a deformation of $\Thetabar \otimes_k k'$. By \Cref{compprop} $\Theta \otimes_{\calO_{\kappa}} \calO_{\kappa'}$ thus comes from a continuous deformation $\rho : \Gamma \to G(\calO_{\kappa'})$ of $\rhobar$.
\end{proof}

\begin{definition} We say that a prime $p$ is \emph{very good} for a simple algebraic group $G$ over an algebraically closed field of characteristic $p$, if the following conditions hold.
\begin{enumerate}
    \item $p \nmid n+1$, if $G$ is of type $A_n$.
    \item $p \neq 2$, if $G$ is of type $B, C, D, E, F, G$.
    \item $p \neq 3$, if $G$ is of type $E, F, G$.
    \item $p \neq 5$, if $G$ is of type $E_8$.
\end{enumerate}
We say that $p$ is \emph{very good} for a reductive algebraic group $G$, if it is very good for every simple factor of $(G^0)^{\ad}$.
\end{definition}

\begin{lemma}\label{centrtriv} Let $\Gamma$ be a group. Let $G \subseteq \GL_n$ be a reductive group over an algebraically closed field $k$ of characteristic $p \geq 0$ and let $\rho : \Gamma \to G(k)$ be a $G$-completely reducible representation, which is in addition irreducible after embedding into $\GL_n(k)$. 

Assume, that one of the following holds:
\begin{enumerate}
    \item $p$ is very good for $G^{\ad}$ and $G^{\ad}$ is connected.
    \item $(\GL_n, G)$ is a reductive pair, i.e. $\frakg$ is a $G$-module direct summand of $\gl_n$.
\end{enumerate}
Then the scheme-theoretic centralizer $Z_{G^{\ad}}(\rho)$ of $\rho$ in $G^{\ad}$ is trivial.
\end{lemma}

\begin{proof} Recall, that $Z_{G^{\ad}}(\rho)$ is defined as follows. For $A \in \CAlg_k$, the group $Z_{G^{\ad}}(\rho)(A)$ is defined as the kernel of the map $G^{\ad}(A) \to \Hom(\Gamma, G(A)), \quad g \mapsto g \rho g^{-1}$.
By Schur's lemma $Z_{\GL_n}(\rho)(k) = Z(\GL_n)(k)$.
Let $\pi : G \twoheadrightarrow G^{\ad}$ be the canonical projection.
By definition, $Z_{G}(\rho) = \pi^{-1}(Z_{G^{\ad}}(\rho))$ and $Z_G(\rho) = Z_{\GL_n}(\rho) \cap G$.
We get $Z(G)(k) \subseteq \pi^{-1}(Z_{G^{\ad}}(\rho)(k)) = Z_G(\rho)(k) = Z_{\GL_n}(\rho)(k) \cap G(k) \subseteq Z(G)(k)$.
We conclude, that $Z_{G^{\ad}}(\rho)(k)$ is trivial.

Assuming (1), we see by \cite[Theorem 1.2]{BMRT} since $p$ is very good for $G^{\ad}$ and $G^{\ad}$ is connected, that $Z_{G^{\ad}}(\rho)$ is smooth and thus trivial as an algebraic group.

Assuming (2), we obtain from \cite[Corollary 2.13]{BMRT}, that $Z_G(\rho)$ is smooth.
Since $\GL_n$ is separable, $Z_{\GL_n}(\rho)$ is also smooth and we have $Z_{\GL_n}(\rho) = Z(\GL_n)$. We can repeat the above calculation without taking points:
$$ Z(G) \subseteq \pi^{-1}(Z_{G^{\ad}}) = Z_G(\rho) = Z_{\GL_n}(\rho) \cap G = Z(\GL_n) \cap G \subseteq Z(G) $$
Hence $Z_{G^{\ad}} = 1$.
\end{proof}

\begin{proposition}\label{descrdefring} Let $G$ be a Chevalley goup, let $\rhobar : \Gamma_F \to G(\kappa)$ be a continuous representation over a finite or local field $\kappa$ and let $\Lambda$ be a coefficient ring for $\kappa$. Assume, that the unframed deformation functor is representable by $R_{\rhobar}$. We have a presentation $R_{\rhobar} \cong \Lambda[[x_1, \dots, x_r]]/(f_1, \dots, f_s)$, where $r = h^1(\Gamma_F, \ad_{\rhobar})$ and $s = h^2(\Gamma_F, \ad_{\rhobar})$.
\end{proposition}

\begin{proof} This follows from a standard calculation with cocycles. See e.g. \cite{MR1643682}.
\end{proof}

\begin{proposition}\label{pdeflemma} Let $F$ be a $p$-adic local field with absolute Galois group $\Gamma_F$. Let $\kappa$ be a finite or local field of very good characteristic $p \geq 0$ for $G^{\ad}_{\overline\kappa}$, let $\Lambda$ be a coefficient ring for $\kappa$ and let $G \subseteq \GL_n$ be a Chevalley group over $\Lambda$. 
Let $\rhobar : \Gamma_F \to G(\kappa)$ be an absolutely $G$-completely reducible continuous representation with associated $G$-pseudocharacter $\Thetabar \in \cPC_G^{\Gamma_F}(\kappa)$, such that $\rhobar$ is absolutely irreducible after embedding into $\GL_n(\overline\kappa)$ and such that $H^2(\Gamma_F, \frakg_{\kappa}) = 0$.

Assume, that one of the following holds:
\begin{enumerate}
    \item $p$ is very good for $G_{\overline\kappa}^{\ad}$ and $G_{\overline\kappa}^{\ad}$ is connected.
    \item $(\GL_{n, \overline\kappa}, G_{\overline\kappa})$ is a reductive pair, i.e. $\frakg_{\overline\kappa}$ is a $G_{\overline\kappa}$-module direct summand of $\gl_{n, \overline\kappa}$.
\end{enumerate}
Then $\RpsThetabar$ is formally smooth over $\Lambda$ of dimension $\dim \frakg_{\kappa} \cdot [F : \bbQ_p] + h^0(\Gamma_F, \frakg_{\kappa}) + \dim \Lambda$. In particular, $\RpsThetabar \cong \Lambda[[x_1, \dots, x_r]]$.
\end{proposition}

\begin{proof} By \Cref{centrtriv} the scheme-theoretic centralizer of $\rhobar$ in $G_{\overline\kappa}^{\ad}$ is trivial.
We can apply \Cref{compprop} to obtain a canonical isomorphism $R_{\rhobar} \cong \RpsThetabar$.
By \Cref{descrdefring}, the deformation ring $R_{\rhobar}$ is isomorphic to $\Lambda[[x_1, \dots, x_r]]$, where $r = h^1(\Gamma_F, \frakg_{\kappa})$. The Euler characteristic formula \cite[Theorem 3.4.1]{BJ_new} implies, that $\dim R_{\rhobar} = \dim \frakg_{\kappa} \cdot [F : \bbQ_p] + h^0(\Gamma_F, \frakg_{\kappa}) + \dim \Lambda$.
\end{proof}

\section{\texorpdfstring{The rigid analytic space of $G$-pseudocharacters}{The p-adic analytic space of G-pseudocharacters}}
\label{secpadicspace}

We fix notations for this section:
\begin{itemize}
    \item Let $\Gamma$ be a topologically finitely generated profinite group.
    \item Let $L$ be a finite extension of $\bbQ_p$ with ring of integers $\calO$, uniformizer $\varpi$ and residue field $\kappa$.
    \item Let $\Aff_L$ be the category of affinoid $L$-algebras.
    \item Let $\An_L$ be the category of rigid analytic spaces over $L$.
    \item Let $G$ be a generalized reductive group scheme over $\calO$.
\end{itemize}

In \cite[Thm. D]{MR3444227} Chenevier shows, that the functor $X_d : \An_{\bbQ_p}^{\op} \to \Set$ on the category $\An_{\bbQ_p}$ of rigid analytic spaces over $\bbQ_p$, that associates to every $Y \in \An_{\bbQ_p}$ the set $\cDet_{d}^{\Gamma}(\calO(Y))$ of continuous $d$-dimensional determinant laws of $\Gamma$ with values in the global sections $\calO(Y)$, is representable by a quasi-Stein rigid analytic space. Here $\calO(Y)$ carries the topology of uniform convergence on open affinoid subsets. By \Cref{chenevierlafforguecont} the set $\cDet_{d}^{\Gamma}(\calO(Y))$ identifies with $\cPC_{\GL_d}^{\Gamma}(\calO(Y))$.
The goal of this section is to generalize Chenevier's construction to generalized reductive group schemes.

\begin{definition} Define $X_G : \Aff_L \to \Set$ as the functor that associates to every affinoid $L$-algebra $A$ the set of continuous $G$-pseudocharacters $\cPC_{G}^{\Gamma}(A)$.
\end{definition}

All of Chenevier's results carry over when $G=\GL_d$ by base change from $\bbQ_p$ to $L$, in particular we know that $X_{\GL_d}$ is representable by a quasi-Stein rigid analytic space over $L$.
Using invariant theory, it is possible to give a direct construction of $X_G$ as closed subspaces of $X_{\GL_d}$ for the classical groups $\SL_n$, $\Sp_n$, $\GSp_n$, $\OO_n$ or $\GO_n$ and under the weaker assumption, that $\Gamma$ satisfies Mazur's condition $\Phi_p$.
We will not do this, but instead give a functorial construction for generalized reductive $G$, which does not depend on the choice of a faithful representation of $G$.

\subsection{\texorpdfstring{The formal scheme of $G$-pseudocharacters}{The formal scheme of G-pseudocharacters}}

Before we construct the $p$-adic analytic space of $G$-pseudocharacters, we define an auxiliary functor on the level of admissible $\calO$-algebras, which is representable by a disjoint union of formal spectra of deformation rings of residual representations, recovering \cite[Cor. 3.14]{MR3444227} when $G=\GL_d$.
Let $A$ be a complete Hausdorff commutative topological ring. We say that $A$ is \emph{admissible}, if $0$ has a neighborhood basis of ideals, there is an ideal $I \subseteq A$, called \emph{ideal of definition}, such that an ideal $J \subseteq A$ is open if and only if there is some $n \geq 1$, such that $I^n \subseteq J$.

\begin{lemma}\label{admissiblecomp} Let $A$ be a commutative topological ring. The following are equivalent:
\begin{enumerate}
    \item $A$ is complete linearly topologized and has an ideal of definition. This is the notion of admissibility defined in \stackcite{07E8}.
    \item $A$ is, in the category of commutative topological rings, isomorphic to a cofiltered limit of discrete rings $\varprojlim_{\lambda} A_{\lambda}$, where the index category possesses a final object $0$ and the transition maps $A_{\lambda} \to A_0$ are surjective with nilpotent kernel. This is the notion of admissibility defined in \cite[§3.9]{MR3444227}.
\end{enumerate}
\end{lemma}

\begin{proof} This is \cite[Lemme 0.7.2.2]{MR217083}.
\end{proof}

\begin{definition} Let $\frakX_G : \Adm_{\calO} \to \Set$ be the functor, that attaches to an admissible $\calO$-algebra $A$ the set of continuous pseudocharacters $\cPC_G^{\Gamma}(A)$.
\end{definition}

Next, we will define a set which will later on be the index set of a disjoint decomposition of $\frakX_G$ and $X_G$ into open subspaces.

\begin{definition}\label{defindexset} We denote by $|\PC_{G}^{\Gamma}| \subseteq \PC_{G}^{\Gamma}$ the subset of closed points $z$ with finite residue field $k_z$, such that the canonical $G$-pseudocharacter $\Theta_z \in \PC_G^{\Gamma}(k_z)$ attached to $z$ is continuous for the discrete topology on $k_z$.
\end{definition}

\begin{lemma}\label{factoringdiscrete} Let $A$ be a discrete $\calO$-algebra and $\Theta \in \cPC_G^{\Gamma}(A)$. Then $\Theta$ factors over a quotient $\Gamma/\Delta$ for an open normal subgroup $\Delta \leq \Gamma$.
\end{lemma}

\begin{proof} The idea is the same as in the proof of \Cref{decisivefiniteness}. Let $\sigma = (\sigma_1, \dots, \sigma_r) \in \Gamma^r$ be a tuple of topological generators of $\Gamma$ and let $\Sigma$ be the subgroup generated by $\sigma_1, \dots, \sigma_r$. By \cite[Theorem 2 (i)]{Seshadri}, $\calO[G^{r+1}]^{G^0}$ is a finitely generated $\calO$-algebra. Let $f_1, \dots, f_s \in \calO[G^{r+1}]^{G^0}$ be a set of $\calO$-algebra generators.
Since $A$ is discrete and $\Gamma^{r+1}$ is a profinite set, a map $\Theta_{r+1}(f_i) : \Gamma^{r+1} \to A$ is constant on a finite partition of open subsets of $\Gamma^{r+1}$.
Such a partition can be refined to consist of open sets in a topological basis of $\Gamma^{r+1}$.
So we can assume, that the partition of $\Gamma^{r+1}$ consists of products of sets in a topological basis of $\Gamma$.
Refining further, we obtain a partition into products of cosets of an open normal subgroup $\Delta_i$ of $\Gamma$.
We take $\Delta := \bigcap_{i=1}^s \Delta_i$ and observe, that for all $\gamma \in \Gamma^{r+1}$, all $\delta \in \Delta$ and all $f \in \calO[G^{r+1}]^{G^0}$, we have $\Theta_{r+1}(f)(\gamma_1, \dots, \gamma_r, 1) = \Theta_{r+1}(f)(\gamma_1, \dots, \gamma_{r}, \delta)$.

Let $m \geq 0$, $\gamma = (\gamma_1, \dots, \gamma_m) \in \Gamma^m$, $f \in \calO[G^m]^{G^0}$ and $\delta \in \Delta$.
Our goal is to show, that $\Theta_m(f)(\gamma_1, \dots, \gamma_m) = \Theta_m(f)(\gamma_1, \dots, \gamma_m\delta)$ and therefore $\Delta \subseteq \ker(\Theta)$ (see \Cref{defkernel}).
Since $\Theta_m(f)$ is continuous and $A$ is discrete, we can choose $\gamma' = (\gamma_1', \dots, \gamma_m') \in \Sigma^m$ close enough to $\gamma$, such that both $\Theta_m(f)(\gamma_1, \dots, \gamma_m) = \Theta_m(f)(\gamma_1', \dots, \gamma_m')$ and $\Theta_m(f)(\gamma_1, \dots, \gamma_m \delta) = \Theta_m(f)(\gamma_1', \dots, \gamma_m' \delta)$ hold.
There is a homomorphism of free groups $\alpha : \FG(m) \to \FG(r)$, such that the composition with the projection $\FG(r) \twoheadrightarrow \Gamma, ~x_i \mapsto \sigma_i$ maps $x_i$ to $\gamma_i'$.
We extend $\alpha$ to a homomorphism $\tilde \alpha : \FG(m+1) \to \FG(r+1)$, such that $\tilde \alpha(x_{m+1}) = x_{r+1}$.
Let $\eta : \FG(m) \to \FG(m+1)$ be defined by $\eta(x_i) := x_i$ for $i \leq m-1$ and $\eta(x_m) = x_mx_{m+1}$.

Using, what we have just proved, we conclude:
\begin{align*}
    \Theta_m(f)(\gamma_1, \dots, \gamma_m\delta) &= \Theta_m(f)(\gamma_1', \dots, \gamma_m'\delta) = \Theta_{m+1}(f^{\eta})(\gamma_1', \dots, \gamma_m',\delta) = \Theta_{r+1}((f^{\eta})^{\tilde\alpha})(\sigma_1, \dots, \sigma_r,\delta) \\
    &= \Theta_{r+1}((f^{\eta})^{\tilde\alpha})(\sigma_1, \dots, \sigma_r, 1) = \Theta_{m+1}(f^{\eta})(\gamma_1', \dots, \gamma_m', 1) \\
    &= \Theta_m(f)(\gamma_1', \dots, \gamma_m') = \Theta_m(f)(\gamma_1, \dots, \gamma_m)
\end{align*}
By the homomorphisms theorem \Cref{homtheorem}, $\Theta$ factors over a unique pseudocharacter of $\Gamma/\Delta$.
\end{proof}

We have a more explicit description of $|\PC_{G}^{\Gamma}|$:

\begin{lemma}\label{descrindexset} There is a canonical bijection between $|\PC_{G}^{\Gamma}|$ and the set of continuous $G$-completely reducible representations $\Gamma \to G(\kappabar)$ up to $G^0(\kappabar)$-conjugation and the $\kappa$-linear Frobenius action on $G(\kappabar)$ on the coefficients.
\end{lemma}

\begin{proof} Let $\calS$ be the set of continuous $G$-completely reducible representations $\Gamma \to G(\kappabar)$ modulo the action of $G^0(\kappabar)$ by conjugation and modulo the action of the $\kappa$-Frobenius of $\kappabar$ on the entries of $G(\kappabar)$. Let $\calS \to |\PC_{G}^{\Gamma}|$ be the map, that maps an equivalence class $[\rho]$ to the well-defined and unique point in the image of $\Spec(\kappabar) \to \PC_{G}^{\Gamma}$ attached to $\Theta_{\rho}$. Surjectivity follows from the reconstruction theorem \Cref{reconstructiongeneral} together with the fact, that a continuous pseudocharacter over $\kappabar$ factors over a quotient by an open normal subgroup \Cref{factoringdiscrete}. For injectivity suppose $\rho, \rho' : \Gamma \to G(\kappabar)$ are such, that the attached pseudocharacters $\Theta_{\rho}$ and $\Theta_{\rho'}$ are supported on the same point $z \in |\PC_{G}^{\Gamma}|$. Then there are $\kappa$-homomorphisms $f, f' : k_z \to \kappabar$, such that $\Theta_z \otimes_{k_z, f} \kappabar = \Theta_{\rho}$ and $\Theta_z \otimes_{k_z, f'} \kappabar = \Theta_{\rho'}$. We can take a power of the $\kappa$-Frobenius $\varphi : \kappabar \to \kappabar$, such that $\varphi \circ f = f'$, in particular $\Theta_{\rho} \otimes_{\kappabar, \varphi} \kappabar = \Theta_{\rho'}$. The uniqueness part of \Cref{reconstructiongeneral} tells us, that $\rho \otimes_{\kappabar, \varphi} \kappabar$ and $\rho'$ are conjugate.
\end{proof}

\begin{lemma}\label{BGammaGfinite} Let $\Gamma$ be a finite group and let $r \geq 1$. Then $B_{G_{\calO/\varpi^r}}^{\Gamma}$ is finite as a set.
\end{lemma}

\begin{proof} We assume $G$ is defined over $\calO/\varpi^r$. We first show that $B_G^{\Gamma} \otimes \kappabar$ is a finite-dimensional $\kappabar$-vector space. All tensor products are over $\calO/\varpi^r$. By \Cref{decisivefiniteness}, we already know that $B_G^{\Gamma} \otimes \kappabar$ is a finitely generated $\kappabar$-algebra. Choose a surjection $\FG(r) \twoheadrightarrow \Gamma$. In the proof of \Cref{decisivefiniteness} we have found a surjection $\calO/\varpi^r[G^m]^{G^0} \twoheadrightarrow B_G^{\Gamma}$ for some $m \geq 1$. For a faithful representation $G \hookrightarrow \GL_d$, we obtain a map $B_{\GL_d}^{\Gamma} \to B_G^{\Gamma}$. The map $\calO/\varpi^r[\GL_d^m]^{\GL_d} \to \calO/\varpi^r[G^m]^{G^0}$ is finite by \cite[Theorem 1]{cotner} and \cite[Theorem 2 (i)]{Seshadri}. It follows, that $B_{\GL_d}^{\Gamma} \to B_G^{\Gamma}$ is finite, which reduces the claim to the case $G=\GL_d$. Let $B := B_{\GL_d}^{\Gamma}$.

By \Cref{reconstructiongeneral} the canonical map $\Rep^{\Gamma, \square}_{\GL_d}(\kappabar) \twoheadrightarrow \PC_{\GL_d}^{\Gamma}(\kappabar) = \Hom_{\CAlg_{\kappabar}}(B \otimes \kappabar, \kappabar)$ is surjective. But $\Rep^{\Gamma, \square}_{\GL_d}(\kappabar)$ is finite, so $B \otimes \kappabar$ has finitely many $\kappabar$-points and thus its nilreduction $(B \otimes \kappabar)_{\red}$ must be a finite product of $\kappabar$ with itself. The nilradical $N := \Nil(B \otimes \kappabar)$ is finitely generated and hence nilpotent. So by induction each $N^i/N^{i+1}$ is a finitely generated $(B \otimes \kappabar)_{\red}$-module. It follows, that $B \otimes \kappabar$ is a finite-dimensional $\kappabar$-vector space. Hence $B/\varpi$ is finite. Since $B$ is $\varpi^r$-torsion, there is a finite descending sequence
$$ B \supseteq \varpi B \supseteq \varpi^2 B \supseteq \dots \supseteq 0 $$
with quotients $\varpi^i B/\varpi^{i+1} B$. These are finitely generated $B/\varpi$-modules, hence finite and thus $B$ is finite.
\end{proof}

\begin{lemma}\label{Aprimelemma} Let $A$ be an admissible $\calO$-algebra. Let $\Theta \in \cPC_G^{\Gamma}(A)$ be a continuous pseudocharacter.
Let $A' \subseteq A$ be the closure of the $\calO$-subalgebra of $A$ generated by $\Theta_n(f)(\gamma_1, \dots, \gamma_n)$ for all $n \geq 1$, all $f \in \calO[G^n]^{G^0}$ and all $(\gamma_1, \dots, \gamma_n) \in \Gamma^n$. Then $A'$ is an admissible profinite $\calO$-subalgebra of $A$.
\end{lemma}

\begin{proof} Assume, that $A$ is discrete. Then there is some $r \geq 1$, such that $\varpi^rA = 0$. By \Cref{basechange}, $\Theta$ factors over the $G_{\calO/\varpi^r}$-pseudocharacter $\Theta/\varpi^r := \Theta \otimes_{\calO} \calO/\varpi^r$. By \Cref{factoringdiscrete} $\Theta/\varpi^r$ factors through an open subgroup $\Delta \leq \Gamma$. The representing ring $B^{\Gamma/\Delta}_{G_{\calO/\varpi^r}}$ of $\PC^{\Gamma/\Delta}_{G_{\calO/\varpi^r}}$ is finite by \Cref{BGammaGfinite}.
By \Cref{repofPC}, $A'$ is the image of the map $B^{\Gamma/\Delta}_{G_{\calO/\varpi^r}} \to A$ attached to $\Theta/\varpi^r$, in particular $A'$ is finite, hence admissible.

Now let $A = \varprojlim_{\lambda} A_{\lambda}$ be a presentation of $A$ as an inverse limit of discrete rings as in \Cref{admissiblecomp}.
Let $\pi_{\lambda} : A \to A_{\lambda}$ be the canonical projection and let $\Theta_{\lambda} := \pi_{\lambda *} \Theta$.
Since $A_{\lambda}$ is discrete, the image $A_{\lambda}'$ of $A'$ in $A_{\lambda}$ is finite by the previous step.
Since $\ker(A_{\lambda}' \to A_0') \subseteq \ker(A_{\lambda} \to A_0)$, the former kernel is nilpotent for all $\lambda$.
It follows from \Cref{admissiblecomp}, that $A' = \varprojlim_{\lambda} A_{\lambda}'$ is admissible.
\end{proof}

We have just shown, that $\Theta$ can be uniquely descended to a continuous $A'$-valued pseudocharacter.

\begin{definition} If $A'$ in \Cref{Aprimelemma} is local, we say that $\Theta$ is \emph{residually constant}.
\end{definition}

In \Cref{prodoflocal}, we will see that $A'$ is a finite product of local profinite admissible $\calO$-algebras.
So if $\Theta$ is not residually constant it is essentially a finite product of residually constant pseudocharacters, defined over different connected components of $A'$. This picture will be crucial for the description of the functor of points of the generic fiber in \Cref{rigspaceG}.

Suppose $\Theta$ is residually constant. In \Cref{Aprimelemma} the natural map $B_G^{\Gamma} \to A_0'$ (with $A_0'$ as in \Cref{admissiblecomp}) is surjective by definition. The radical of the kernel of this map does not depend on the choice of the presentation of $A'$ as an inverse limit as in \Cref{admissiblecomp}. It is a maximal ideal of $B_G^{\Gamma}$ with finite residue field and therefore determines a closed point $z \in |\PC^{\Gamma}_G|$. The residue field of $A'$ is canonically isomorphic to the residue field $k_z$ of $z$. Therefore $\Theta$ can be reduced to a continuous $k_z$-valued pseudocharacter along the map $A' \to k_z$. This reduction is the pseudocharacter $\Theta_z$ attached to $z$.

\begin{theorem} Let $z \in |\PC_G^{\Gamma}|$ and let $\frakX_{G,z} : \Adm_{\calO} \to \Set$ be the functor, that attaches to an admissible $\calO$-algebra $A$ the set $\frakX_{G,z}(A)$ of continuous pseudocharacters $\Theta \in \cPC_G^{\Gamma}(A)$, such that $\Theta$ is residually constant and equal to $\Theta_z$. Then $\frakX_{G,z}$ is representable by the deformation ring $R^{\ps}_{\Theta_z}$ of $\Theta_z$ over $\calO_{L'}$ for a finite unramified extension $L'/L$ such that the residue field of $\calO_{L'}$ is $k_z$, which is a complete noetherian local $\calO_{L'}$-algebra with residue field $k_z$.
\end{theorem}

\begin{proof} 
    Let $\Theta \in \frakX_{G,z}(A)$. By \Cref{Aprimelemma}, $\Theta$ descends to an $A'$-valued pseudocharacter for some admissible profinite $\calO$-subalgebra $A' \subseteq A$, which we will also denote by $\Theta$. Using a multiplicative section of the reduction map of $(A')^{\times} \to k_z^{\times}$ (which exists and is unique by \stackcite{04GM} and \stackcite{06RR}), we see that there is a finite unramified extension $L'/L$, such that $\calO_{L'}$ has residue field $k_z$ and $A'$ is an $\calO_{L'}$-algebra. By \Cref{basechange} $\Theta$ can be regarded as a $G_{\calO_{L'}}$-pseudocharacter. As such it is a lift of $\Theta_z$ in the pseudodeformation functor $\Def_{\Theta_z} : \Art_{\calO_{L'}} \to \Set$. It follows, that $\Def_{\Theta_z}$ and $\frakX_{G,z}$ are naturally isomorphic as functors on $\Art_{\calO_{L'}}$.
    By \Cref{tfgfingen}, the pseudodeformation functor $\Def_{\Theta_z}$ is representable by a complete noetherian local $\calO_{L'}$-algebra with residue field $k_z$.
\end{proof}

From now on, we denote by $\frakX_{G,z}$ the formal scheme $\Spf(R^{\ps}_{\Theta_z})$.

\begin{lemma}\label{prodoflocal}
    Let $A$ be a profinite admissible $\calO$-algebra. Then $A$ is a finite product of local profinite admissible $\calO$-algebras.
\end{lemma}

We emphasize, that \Cref{prodoflocal} holds independently of any noetherianity hypothesis on $A$.

\begin{proof}
    We only show, that $A$ is a finite product of local rings, the rest of the claim then follows easily.
    Let $\frakm$ be a maximal ideal of $A$ and let $I$ be an ideal of definition of $A$.
    Then $\{(I^n + \frakm)/\frakm\}_{n \geq 1}$ is a system of open subgroups of $A/\frakm$, that induces the quotient topology of $A/\frakm$.
    But $I^n + \frakm$ is either $\frakm$ or $A$, so $A/\frakm$ is either discrete or indiscrete.
    Since $\kappa$ is discrete and there is a continuous injection $\kappa \to A/\frakm$ induced by the natural map $\calO \to A$, we have that $A/\frakm$ is discrete, hence finite.
    So there is some $n \geq 1$, such that $I^n + \frakm = \frakm$, hence $I \subseteq \frakm$.

    We know that $\frakm A/I$ is a maximal ideal of $A/I$ and by \cite[(24.C)]{matsumura1970commutative}, we know that $A/I$ has only finitely many maximal ideals.
    It follows, that $A$ has only finitely many maximal ideals.
    Since $A$ is commutative, it follows that $A$ is semi-local and thus the claim follows from \cite[(24.C)]{matsumura1970commutative}.    
\end{proof}

\begin{lemma}\label{resconst}
    Let $A$ be an admissible local $\calO$-algebra and let $\Theta \in \cPC^{\Gamma}_G(A)$. Then $\Theta$ is residually constant.
\end{lemma}

\begin{proof}
    According to \Cref{Aprimelemma}, there is an admissible profinite subring $A' \subseteq A$, over which $\Theta$ is defined.
    From \Cref{prodoflocal}, we obtain a system of primitive orthogonal idempotents for $A'$, which also leads to a product decomposition of $A$.
    It follows, that the only nonzero idempotent of $A'$ is $1$ and that $A'$ is local.
\end{proof}

\begin{corollary} The functor $\frakX_G : \Adm_{\calO} \to \Set$ is representable by the coproduct $\coprod\nolimits_{z \in |\PC_{G}^{\Gamma}|} \frakX_{G,z}$ in the category of formal schemes over $\calO$.
\end{corollary}

\begin{proof} It is clear, that on the level of Zariski sheaves on $\Adm_{\calO}$, there is an injective natural transformation $\coprod\nolimits_{z \in |\PC_{G}^{\Gamma}|} \frakX_{G,z} \to \frakX_{G}$. We want to show surjectivity. Let $A$ be an admissible $\calO$-algebra. If $\Theta \in \frakX_{G}(A)$, then by \Cref{Aprimelemma} $\Theta$ is defined over a profinite admissible $\calO$-algebra, so we may assume $A$ is profinite. Then by \Cref{prodoflocal} $A$ is a finite product $A = \prod_i A_i$ of local profinite $\calO$-algebras $A_i$.

Since every continuous $G$-pseudocharacter over an admissible local $\calO$-algebra is automatically residually constant (\Cref{resconst}), the map of sets $(\coprod_z \frakX_{G,z})(A_i) = \coprod_z \frakX_{G,z}(A_i) \to \frakX_G(A_i)$ is bijective for all $i$. This will be used in the third equality below.
Recall also, since the decomposition of $A$ is finite, we have $\Spf(A) = \coprod_i \Spf(A_i)$ in the category $\FSch_{\calO}$ of formal $\calO$-schemes.

We calculate
\begin{align*}
    \Hom_{\FSch_{\calO}}(\Spf(A),~\coprod\nolimits_z \frakX_{G,z}) &= \prod\nolimits_i \Hom_{\FSch_{\calO}}(\Spf(A_i),~\coprod\nolimits_z \frakX_{G,z}) = \prod\nolimits_i \coprod\nolimits_z \Hom_{\FSch_{\calO}}(\Spf(A_i),~\frakX_{G,z}) \\
    &= \prod\nolimits_i \Hom_{\FSch_{\calO}}(\Spf(A_i),~\frakX_{G}) = \Hom_{\FSch_{\calO}}(\Spf(A),~\frakX_{G})
\end{align*}
\end{proof}

\subsection{\texorpdfstring{The rigid analytic space of $G$-pseudocharacters}{The rigid analytic space of G-pseudocharacters}} The goal of this subsection is to construct the $p$-adic analytic space of $G$-pseudocharacters, which will be obtained by taking Berthelot's generic fiber (see \cite[(0.2.6)]{BerthCohRig} or \cite[§7]{deJongBerthelot}) of $\frakX_G$.
Let $\FSch_{\calO}^{\lnad}$ be the category of locally noetherian adic formal schemes $\frakX$ over $\Spf(\calO)$ such that the mod $\varpi$ reduction $\frakX_{\red}$ of $\frakX$ is a scheme locally of finite type over $\Spec(\kappa)$.

We briefly recall the features of Berthelot's functor.
It is a functor
$$ (\phantom{-})^{\rig} : \FSch_{\calO}^{\lnad} \to \An_L, ~ \frakX \mapsto \frakX^{\rig} $$
from $\FSch_{\calO}^{\lnad}$ to the category of rigid analytic spaces over $L$.

If $\frakX$ is of the form $\Spf(A)$ for some quotient $A = \calO[[x_1, \dots, x_n]]/(f_1, \dots, f_s)$ of a formal power series ring $\calO[[x_1, \dots, x_n]]$, the space $\frakX^{\rig}$ will be a closed analytic subvariety of the rigid analytic open unit disk $\bbD^n$ of dimension $n$, defined by vanishing of the functions $f_1, \dots, f_s$ interpreted as analytic functions on $\bbD^n$.

If $A$ is an affinoid $L$-algebra, a \emph{model} of $A$ is a continuous open $\calO$-algebra homomorphism $\calA \to A$ for some admissible $\calO$-algebra $\calA$, such that the induced map $\calA[1/\varpi] \to A$ is an isomorphism. For a fixed model $\calA \to A$, there is a canonical map
$$ \iota_{\calA} : \frakX_G(\calA) \to X_G(A), $$
that maps a continuous pseudocharacter with values in $\calA$ to its base change to $A$.

We also have a natural map
\begin{align}
    \iota : \varinjlim_{\calA} \frakX_G(\calA) \to X_G(A), \label{iota}
\end{align}
where the colimit on the left hand side is taken over the category of all models of $A$ with continuous ring homomorphisms over $A$.
The next goal is to show, that $\iota$ is bijective.
For $d$-dimensional determinant laws (i.e. $G=\GL_d$ here by Emerson's isomorphism) and $L=\bbQ_p$ and  this has been shown by Chenevier in \cite[Lemma 3.15]{MR3444227}.

\begin{lemma}\label{calclem} Let $A$ be an affinoid $L$-algebra and let $\Theta \in X_G(A)$.
    \begin{enumerate}
        \item For all $m \geq 1$, all $f \in \calO[G^m]^{G^0}$ and all $\gamma \in \Gamma^m$, we have that $\Theta_m(f)(\gamma)$ is contained in the subring $A^{\circ}$ of power-bounded elements of $A$.
        \item Assume, that $\Gamma$ is topologically finitely generated. Then $\iota$ in \eqref{iota} is bijective.
        \item Assume, that $\Gamma$ is topologically finitely generated. If $A$ is reduced, then $\frakX_G(A^{\circ}) = X_G(A)$.
    \end{enumerate}
\end{lemma}

\begin{proof} \phantom{a}
    \begin{enumerate}
        \item An element of an affinoid $L$-algebra is power-bounded if and only if for every maximal ideal $\frakm \subseteq A$, its image in $A/\frakm$ is power-bounded. This follows from \cite[Proposition 6.2.3/1]{MR746961} and the boundedness of the supremum norm \cite[§6.2.1 and Corollary 3.8.2/2]{MR746961}. We may thus assume, that $A$ is a finite field extension of $L$ and that $A^{\circ}$ the ring of integers of $A$. The claim follows directly from \cite[Theorem 4.8 (i)]{BHKT}.
        
        \item $\iota$ is injective, as every model $\calA$ of $A$ maps to a $\varpi$-torsion-free model (take the image of $\calA$ in $A$) and for a torsion-free model $\calA$, the map $\iota_{\calA}$ is injective. We are left to show surjectivity of $\iota$, so let $\Theta \in X_G(A)$ and let $\calA \subseteq A$ be some (torsion-free) model of $A$. Since we assume, that $\Gamma$ is topologically finitely generated, we can choose a finitely generated dense subgroup $\Sigma \subseteq \Gamma$. Let $\sigma_1, \dots, \sigma_r \in \Sigma$ be group generators of $\Sigma$. Let $f_1, \dots, f_s \in \calO[G^r]^{G^0}$ be $\calO$-algebra generators, which we find by \cite[Theorem 2 (i)]{Seshadri}.
        
        We define a compact subset $C := \bigcup_{i=1}^s \Theta_r(f_i)(\Gamma^r) \subseteq A$. As $\calA$ is an open subset of $A$, $C$ meets only finitely many additive translates of $\calA$ in $A$. So there are $k_i \in C$ with $i=1, \dots, t$, such that $C \subseteq \bigcup_{i=1}^t (k_i + \calA)$.
        
        We claim that the algebra $\calA' := \calA\langle k_1, \dots, k_s \rangle$ (the closure of $\calA[k_1, \dots, k_s]$ in $A$) is a model of $A$ containing $C$. First, since $\calA$ is open in $A$, $\calA'$ is also open. It is also clear, that $\calA'[1/\varpi] = A$. For admissibility of $\calA'$, we note, that by (1) each of the $k_i$ is power-bounded, so there is a continuous surjection by a Tate algebra $\calA\langle T_1, \dots, T_s\rangle \to \calA'$ mapping $T_i \mapsto k_i$, and this map is also open, since after inverting $\varpi$, we obtain a surjection $A\langle T_1, \dots, T_s\rangle \twoheadrightarrow A$, which is open and a quotient map by the open mapping theorem for $p$-adic Banach spaces \cite[§2.8.1]{MR746961}. It follows, that $\calA'$ is a complete Hausdorff ring, which carries the $I$-adic topology for some ideal of definition of $\calA$ and is therefore admissible.

        We claim, that $\Theta$ actually takes values in $\calA'$, so that $\Theta$ is the image of a pseudocharacter in $\frakX_G(\calA')$, as desired.
        Let $m \geq 1$, $f \in \calO[G^m]^{G^0}$ and $\delta \in \Sigma^m$.
        As in the proof of \Cref{factoringdiscrete}, we find a homomorphism $\alpha : \FG(m) \to \FG(r)$, such that $\Theta_m(f)(\delta) = \Theta_r(f^{\alpha})(\sigma)$.
        Since $f^{\alpha}$ is in the $\calO$-algebra span of the $f_i$ and $\Theta_r(f_i)(\sigma) \in \calA'$ by construction, we find that $\Theta_r(f^{\alpha})(\sigma) \in \calA'$.
        Overall, we have shown that $\Theta_m(f)(\Sigma^m) \subseteq \calA'$.
        Since $\Theta_m(f) : \Gamma^m \to A$ is continuous, $\Gamma^m$ is compact, $A$ is Hausdorff and $\calA'$ is closed in $A$, we conclude that $\Theta_m(f)(\Gamma^m) \subseteq \calA'$ and therefore $\Theta$ takes values in $\calA'$.      
        \item This is a direct consequence of (2), since if $A$ is reduced, then $A^0$ is the terminal model of $A$ \cite[§3.14.1]{MR3444227}.
    \end{enumerate}
\end{proof}

\begin{definition} Let $z \in |\PC_{G}^{\Gamma}|$ and define for every affinoid $L$-algebra $A$ the set $X_{G,z}(A)$ as the set of $\Theta \in X_G(A)$, such that there exists a model $\calA \to A$, such that $\Theta$ is the image of a pseudocharacter $\tilde \Theta \in \frakX_{G,z}(\calA)$.
\end{definition}

Suppose $A$ is an affinoid $L$-algebra and $x$ is a point in the maximal spectrum of $A$ with residue field $L'$.
We know, that $L'$ is a finite extension of $L$.

\begin{definition}
    The \emph{reduction map} at $x$ is defined as $\red_x : X_G(A) \to |\PC_{G}^{\Gamma}|$, where for $\Theta \in X_G(A)$, $\red_x(\Theta)$ shall be the reduction of the unique pseudocharacter $\tilde \Theta \in X_G(\calO_{L'})$ (see \Cref{calclem} (3)) mapping to $\Theta \otimes_A L'$.
\end{definition}

\begin{definition} \phantom{a}
\begin{enumerate}
    \item Define $\tilde X_G : \An_L^{\op} \to \Set$ as the functor, that associates to every rigid analytic space $Y \in \An_L$ the set of continuous $G$-pseudocharacters $\cPC_{G}^{\Gamma}(\calO(Y))$.
    \item For $z \in |\PC_G^{\Gamma}|$, let $\tilde X_{G,z}$ be the subset of $\tilde X_G$ of $G$-pseudocharacters $\Theta$, such that for all $x \in \Specmax(A)$, the specialization $\Theta_z$ of $\Theta$ at $z$ defined as the image of $\Theta$ under $\tilde X_G(A) \to \tilde X_G(k_x) \to \tilde X_G(\calO_{k_x})$ is residually equal to $z$.
\end{enumerate}
\end{definition}

The proofs of \Cref{lemdiv} and \Cref{rigspaceG} are the same as the proofs of \cite[Lemma 3.16]{MR3444227} and \cite[Theorem 3.17]{MR3444227}.

\begin{lemma}\label{lemdiv}
    Suppose $A$ is an affinoid $L$-algebra and $z \in |\PC_{G}^{\Gamma}|$. Then
    $$ X_{G,z}(A) = \{\Theta \in X_G(A) \mid \forall x \in \Specmax(A) : \red_x(\Theta) = z\} $$
\end{lemma}

\begin{proof}
    Let $\Theta \in X_G(A)$, so that for all $x \in \Specmax(A)$, we have $\red_x(\Theta) = z$. By \Cref{calclem} (2), there is some model $\calA \to A$ and some $\Theta' \in \frakX_G(\calA)$ that maps to $\Theta$. Let $A' \subseteq \calA$ be the ring attached to $\Theta'$ as in \Cref{Aprimelemma}. We know, that $A'$ is a product of local $\calO$-algebras $\prod_{i=1}^n A_i'$. The idempotents of this decomposition induce a decomposition of $A$ into a product $\prod_{i=1}^n A_i$. Let $x_i \in \Specmax(A_i)$ be a closed point with residue field $L_i$. By assumption, the kernel of the composition $B_G^{\Gamma} \to A_i' \to \calO_{L_i}/\frakm_{\calO_{L_i}}$ is the maximal ideal of $B_G^{\Gamma}$, that corresponds to $z$. By definition of $A'$, the map $B_G^{\Gamma} \to A' \to A'/\Jac(A')$ is surjective and thus $A'$ itself must be local. This shows, that $\Theta'$ is residually constant and residually equal to $\Theta_z$, so $\Theta' \in \frakX_{G,z}(A')$. It follows, that $\Theta \in X_{G,z}(A)$.
\end{proof}

\Cref{lemdiv} in particular implies, that $\tilde X_{G,z}$ is representable by $\frakX_{G,z}^{\rig}$.

\begin{theorem}\label{rigspaceG} The functor $\tilde X_G$ is representable by the quasi-Stein space $\coprod\nolimits_{z \in |\PC_{G}^{\Gamma}|} \frakX_{G,z}^{\rig}$.
\end{theorem}

\begin{proof}
    To verify, that $\frakX_G^{\rig} = \coprod\nolimits_{z \in |\PC_{G}^{\Gamma}|} \frakX_{G,z}^{\rig}$ represents $\tilde X_G$ it is enough to check that the functor of points agree on affinoid analytic spaces $Y \in \An_L$, since $\tilde X_G$ and the functor of points of $\frakX_G^{\rig}$ are sheaves for the Zariski topology on $\An_L$. We have
    $$ \Hom_{\An_L}(Y, \frakX_G^{\rig}) = \varinjlim_{\calY \to Y} \Hom_{\FSch/\calO}(\calY, \frakX_G) = \varinjlim_{\calY \to Y} \frakX_G(\calO(\calY)) = X_G(\calO(Y)) = \tilde X_G(Y). $$
    Here the first equality is the universal property of Berthelot's generic fiber functor \cite[§7.1.7.1]{deJongBerthelot}, the third equality is using \Cref{calclem} (2).
\end{proof}

\begin{remark} It would also have been possible to take the adic generic fiber $\frakX_G^{\ad} \times_{\spa(\calO)} \spa(L)$ of the adic space $\frakX_G^{\ad}$ attached to $\frakX_G$, which is canonically isomorphic to $X_G^{\ad}$. Although we found no advantage in the usage of adic spaces here, this point of view might be more natural for further applications.
\end{remark}

After fixing an embedding $L \hookrightarrow \overline L$ into an algebraic closure $\overline L$ of $L$, the $\overline L$-points are ad hoc defined as $X_G(\overline L) := \bigcup_{L'/L} X_G(L')$ where the union varies over all finite extensions $L'$ of $L$ contained in $\overline L$.

\begin{corollary}\label{Lbarpoints} The $\overline L$-points of $X_G$ are in canonical bijection with $G^0(\overline L)$-conjugacy classes of continuous $G$-completely reducible representations $\Gamma \to G(\overline L)$.
\end{corollary}

\begin{proof}
    By the definition of $X_G$, we have $X_G(\overline L) = \cPC_G^{\Gamma}(\overline L)$ where $\overline L$ carries the direct limit topology of its finite subextensions of $L$. The claim now follows from \Cref{continuousreconstruction}.
\end{proof}

\section{Dimension of $\RpsThetabar$ for $\Sp_{2n}$}
\label{subsecDimRps}

Fix a prime $p>2$, let $F/\bbQ_p$ be a $p$-adic local field and let $\Gamma_F$ be the absolute Galois group of $F$.
Let $\calO$ be the ring of integers of a $p$-adic local field $L$ with uniformizer $\varpi$ and residue field $\kappa$, let $G$ be a Chevalley group over $\calO$ and let $\Thetabar \in \cPC^{\Gamma_F}_G(\kappa)$ be a continuous $G$-pseudocharacter.
Let $\overline X_{\Thetabar} := \Spec(R^{\ps}_{\Thetabar}/\varpi)$, where $R^{\ps}_{\Thetabar}$ is the universal pseudodeformation ring of \Cref{representabledeffunctor}.
We define 
$$\Sp_{2n}(A) := \{M \in \GL_{2n}(A) \mid M^{-1} = JM^{\top}J^{-1}\},$$ 
where $J = \SmallMatrix{0 & I_n \\ -I_n & 0}$ for every commutative ring $A$.
In this section, we use the methods developed in \cite{BJ_new} to estimate the dimension of $\overline X_{\Thetabar}$ for $G = \Sp_{2n}$.
Note, that by \Cref{tfgfingen} $R^{\ps}_{\Thetabar}$ is noetherian, since $\Gamma_F$ is topologically finitely generated (\cite[Satz 3.6]{Jannsen1982}).
Let $\iota : \Sp_{2n} \to \GL_{2n}$ be the standard representation.
By \Cref{chenevierlafforguecont}, the $\GL_{2n}$-pseudocharacter $\iota(\Thetabar)$ corresponds to a unique determinant law $D_{\iota(\Thetabar)}$ of dimension $2n$.
The pseudodeformation ring $R^{\univ}_{\calO, D_{\iota(\Thetabar)}}$ of $D_{\iota(\Thetabar)}$ defined in \cite[Proposition 4.7.4]{BJ_new} is by \Cref{Rpsisom} canonically isomorphic to $R^{\ps}_{\iota(\Thetabar)}$. We shall use this identification without further mention whenever we cite results from \cite{BJ_new}.

\subsection{Symplectic representations} A \emph{symplectic representation} $(V,\beta)$ of a group $\Gamma$ over a field $k$ is a representation $V$, with a fixed $\Gamma$-invariant antisymmetric non-degenerate bilinear form $\beta : V \times V \to k$. When $p > 2$, two semisimple symplectic representations over an algebraically closed field are conjugate over $\Sp_{2n}$ if and only if they are conjugate over $\GL_{2n}$.
This is a consequence of the fact, that when $p > 2$ the notions of $\Sp_{2n}$-semisimplicity and $\GL_{2n}$-semisimplicity coincide \cite[Corollary 16.10]{Richardson1988ConjugacyCO} and the uniqueness part of \Cref{reconstructiongeneral}.
So being symplectic can be seen as a property of $\GL_{2n}$-conjugacy classes of semisimple representations. It is easy to check, that a representation of the form $W \oplus W^*$ for some arbitrary representation $W$ is always symplectic. We call these \emph{representations of pair type}.
In general a semisimple symplectic representation is a direct sum of irreducible symplectic representations and representations of pair type.

\begin{proposition}\label{symplstructure} Every semisimple symplectic representation of a group $\Gamma$ over an algebraically closed field $k$ is a direct sum of representations of one of the following two types.
\begin{enumerate}
    \item An irreducible symplectic representation.
    \item A direct sum $V \oplus V^*$, where $V$ is an arbitrary irreducible representation.
\end{enumerate}
\end{proposition}

\begin{proof} Let $V$ be a symplectic representation equipped with a $\Gamma$-invariant symplectic form $\beta : V \times V \to k$.

We proceed by induction over $\dim V$. If $\dim V = 0$ there is nothing to show. We assume $\dim V > 0$. Let $W$ be an irreducible subrepresentation of $V$ and assume, that $\beta : W \times W \to k$ is non-degenerate. In particular, $W$ is an irreducible symplectic representation. Then $W^{\perp}$ is non-degenerate and $\Gamma$-invariant and we may assume $W^{\perp}$ has the desired form. This implies the claim.

We now assume, that $V$ has no irreducible subrepresentation on which $\beta$ is non-degenerate.
Let $W$ be any irreducible subrepresentation of $V$.
Since $\beta$ is non-degenerate, there is an irreducible subrepresentation $W' \neq W$, such that $\beta : W \times W' \to k$ is non-degenerate.
$\beta$ is non-degenerate on $W \oplus W'$, so $(W \oplus W')^{\perp}$ is non-degenerate and $\Gamma$-invariant. We have $W' \cong W^*$ via $y \mapsto (x \mapsto \beta(x,y))$. As in the previous case, this implies the claim.
\end{proof}

This motivates the following terminology. We say that a symplectic representation $V$ is \emph{symplectically decomposable}, if it can be written as the direct sum of two nonzero symplectic representations, and \emph{symplectically indecomposable} otherwise.
There are exactly two types of symplectically indecomposable semisimple representations: Those which are irreducible under the standard embedding into $\GL_{2n}$ and those which are a direct sum $W \oplus W^*$ for some irreducible representation $W$, which is not symplectic.

\subsection{Subdivision of $\overline X_{\Thetabar}$}\label{secsubdiv} For a point $x \in \overline X_{\Thetabar} = \Spec(R^{\ps}_{\Thetabar}/\varpi)$, there is a natural $G$-pseudocharacter $\Theta_x \in \PC^{\Gamma}_G(\overline{\kappa(x)})$ defined after choice of an algebraic closure $\overline{\kappa(x)}$ of the residue field $\kappa(x)$ of $x$. 

In their analysis of the special fiber of the pseudodeformation space for $\GL_n$, Böckle and Juschka have noticed that points corresponding to irreducible representations need not be unobstructed. They have found a convenient characterization of obstructed irreducible points \cite[Lemma 5.1.1]{BJ_new}, which allows them to find good dimension bounds for the obstructed locus. We recall their definition \cite[Definition 5.1.2]{BJ_new}. It turns out, that for $G=\Sp_{2n}$ the dimension of the locus of special points for $\GL_{2n}$ is still small enough to get the desired estimates.

\begin{definition} Let $k$ be an algebraically closed $\bbZ_p$-field and let $\rho : \Gamma_F \to \GL_{2n}(k)$ be an irreducible representation. We say that $\rho$ is \emph{special}, if one of the following holds.
\begin{enumerate}
    \item $\zeta_p \notin F$ and $\rho \cong \rho(1)$.
    \item $\zeta_p \in F$ and there is some degree $p$ Galois extension $F'/F$, such that $\rho|_{\Gamma_{F'}}$ is reducible.
\end{enumerate}
\end{definition}

\begin{definition}\label{defsubloci} We define the following subsets of $\overline X_{\Thetabar}$.
\begin{enumerate}
    \item $\overline X_{\Thetabar}^{\irr}$ is the subset of irreducible points.
    \item $\overline X_{\Thetabar}^{\nspcl}$ is the subset of non-special irreducible points.
    \item $\overline X_{\Thetabar}^{\spcl}$ is the subset of special points.
    \item $\overline X_{\Thetabar}^{\pair}$ is the subset of points of pair type.
    \item $\overline X_{\Thetabar}^{\dec}$ is the subset of symplectically decomposable points.
    \item For any of the above subsets $\overline Y_{\Thetabar}^? := {\overline X}_{\Thetabar}^? \setminus \{\frakm_{R^{\ps}_{\Thetabar}/\varpi}\}$.
\end{enumerate}
\end{definition}

We use the symbol $\dot{\cup}$ to denote a set-theoretic disjoint union (which need not be a coproduct of topological spaces).

\begin{proposition}\label{subdivision}
$\overline X_{\overline \Theta} = \overline X_{\overline \Theta}^{\nspcl} ~\dot{\cup} ~\overline X_{\overline \Theta}^{\spcl} ~\dot{\cup} ~(\overline X_{\overline \Theta}^{\dec} \cup \overline X_{\overline \Theta}^{\pair})$.
\end{proposition}

\begin{proof} This follows directly from \Cref{symplstructure}.
\end{proof}

\begin{lemma}\label{finmapdec} Suppose $\Thetabar = \Thetabar_1 \oplus \Thetabar_2 \in \cPC^{\Gamma_F}_{\Sp_{2n}}(\kappa)$ with $\Thetabar_1 \in \cPC^{\Gamma_F}_{\Sp_{2a}}(\kappa)$, $\Thetabar_2 \in \cPC^{\Gamma_F}_{\Sp_{2b}}(\kappa)$ and $a+b=n$, where the direct sum is a direct sum of symplectic pseudocharacters as explained in \Cref{secsumtensor}. Then the map $\RpsThetabar \to R^{\ps}_{\Thetabar_1} \widehat\otimes_{\calO} R^{\ps}_{\Thetabar_2}$ induced by $(\Theta_1, \Theta_2) \mapsto \Theta_1 \oplus \Theta_2$ is finite.
\end{lemma}

\begin{proof} Let $\iota : \Sp_{2d} \to \GL_{2d}$ be the canonical embedding and let $\iota(\Thetabar)$ be the associated $\GL_{2d}$-pseudocharacter, similarly for $\iota_i(\Thetabar_i)$. By \Cref{surjectivity} $\RpsThetabar$ is a quotient of $R^{\ps}_{\iota(\Thetabar)}$ and similarly for $R^{\ps}_{\Thetabar_i}$. We know from \cite[Lemma 3.24]{BIP}, that the map $R^{\ps}_{\iota(\Thetabar)} \to R^{\ps}_{\iota_1(\Thetabar_1)} \widehat\otimes_{\calO} R^{\ps}_{\iota_2(\Thetabar_2)}$ is finite. It follows, that the induced map $\RpsThetabar \to (R^{\ps}_{\iota_1(\Thetabar_1)} \widehat\otimes_{\calO} R^{\ps}_{\iota_2(\Thetabar_2)}) \widehat\otimes_{R^{\ps}_{\iota(\Thetabar)}} \RpsThetabar$ is finite. Since there is a natural surjection $R^{\ps}_{\iota_1(\Thetabar_1)} \widehat\otimes_{\calO} R^{\ps}_{\iota_2(\Thetabar_2)} \twoheadrightarrow R^{\ps}_{\Thetabar_1} \widehat\otimes_{\calO} R^{\ps}_{\Thetabar_2}$, the natural map $(R^{\ps}_{\iota_1(\Thetabar_1)} \widehat\otimes_{\calO} R^{\ps}_{\iota_2(\Thetabar_2)}) \widehat\otimes_{R^{\ps}_{\iota(\Thetabar)}} \RpsThetabar \to R^{\ps}_{\Thetabar_1} \widehat\otimes_{\calO} R^{\ps}_{\Thetabar_2}$ is surjective.
\end{proof}

\begin{lemma}\label{finmappair} Let $\Thetabar = \Thetabar_1 \oplus \Thetabar_1^* \in \cPC^{\Gamma_F}_{\Sp_{2n}}(\kappa)$ be an $\Sp_{2n}$-pseudocharacter as explained at the end of \Cref{secsumtensor} with $\Thetabar_1 \in \cPC^{\Gamma_F}_{\GL_n}(\kappa)$. Then the map $\RpsThetabar \to R^{\ps}_{\Thetabar_1}$ induced by $\Theta_1 \mapsto \Theta_1 \oplus \Theta_1^*$ is finite.
\end{lemma}

\begin{proof} As in the proof of \Cref{finmapdec}, the map $R^{\ps}_{\iota(\Thetabar)} \to R^{\ps}_{\iota_1(\Thetabar_1)} \widehat\otimes_{\calO} R^{\ps}_{\iota_1(\Thetabar_1)}$ is finite. By affineness, the map $R^{\ps}_{\iota_1(\Thetabar_1)} \widehat\otimes_{\calO} R^{\ps}_{\iota_1(\Thetabar_1)} \to R^{\ps}_{\iota_1(\Thetabar_1)}$ induced by $\Theta_1 \mapsto (\Theta_1, \Theta_1^*)$ is surjective. So the composition $R^{\ps}_{\iota(\Thetabar)} \to R^{\ps}_{\iota_1(\Thetabar_1)}$ is finite and induced by $\Theta_1 \mapsto \Theta_1 \oplus \Theta_1^*$. Tensoring with $\RpsThetabar$, we obtain a finite map $\RpsThetabar \to R^{\ps}_{\iota_1(\Thetabar_1)} \widehat\otimes_{R^{\ps}_{\iota(\Thetabar)}} \RpsThetabar \cong R^{\ps}_{\Thetabar_1}$. The last isomorphism can be seen to hold by considering the corresponding deformation functors.
\end{proof}

\begin{proposition}\label{openandclosedsubspaces} \phantom{a}
\begin{enumerate}
    \item The natural map $\overline X_{\Thetabar} \to \overline X_{\iota(\Thetabar)}$ is a closed immersion.
    \item $\overline X_{\Thetabar}^{\spcl}$ is closed in $\overline X_{\Thetabar}^{\irr}$.
    \item $\overline X_{\Thetabar}^{\pair}$ is closed in $\overline X_{\Thetabar}$.
    \item $\overline X_{\Thetabar}^{\dec}$ is closed in $\overline X_{\Thetabar}$.
    \item $\overline X_{\Thetabar}^{\nspcl}$ is open in $\overline X_{\Thetabar}$.
\end{enumerate}
\end{proposition}

\begin{proof} \phantom{a}
\begin{enumerate}
    \item By \Cref{Rpssurj}, the map $R^{\ps}_{\iota(\Thetabar)} \to R^{\ps}_{\Thetabar}$ is surjective.
    \item $\overline X_{\Thetabar}^{\spcl}$ is the preimage of $\overline X_{\iota(\Thetabar)}^{\spcl}$ under the closed immersion $\overline X_{\Thetabar} \to \overline X_{\iota(\Thetabar)}$. Since $\overline X_{\iota(\Thetabar)}^{\nspcl}$ is open in $\overline X_{\iota(\Thetabar)}$ by \cite[Theorem 5.5.1 (b)]{BJ_new}, $\overline X_{\iota(\Thetabar)}^{\spcl}$ is closed in $\overline X_{\iota(\Thetabar)}^{\irr}$ and the claim follows.
    \item $\overline X_{\Thetabar}^{\pair}$ is the union of the images of finitely many maps as in \Cref{finmappair}: Suppose $x \in \overline X_{\Thetabar}^{\pair}$ is a point of dimension $1$ in $\overline X_{\Thetabar}$ (i.e. a closed point of the Jacobson scheme $\overline Y_{\Thetabar}$, \stackcite{02IM}). Then $\Theta_x \in \PC_{\Sp_{2n}}^{\Gamma_F}(\kappa(x))$ is continuous for the natural topology on the local field $\kappa(x)$ (see \cite[Lemma 3.17]{BIP}). The continuous reconstruction theorem \Cref{continuousreconstruction}, gives us a continuous representation $\rho_x : \Gamma_F \to \Sp_{2n}(\kappa')$ valued in some finite extension $\kappa'/\kappa(x)$. After possibly enlarging $\kappa'$, $\rho_x$ decomposes into a direct sum $\rho_x \cong W \oplus W^*$ with $W$ a continuous $n$-dimensional representation. Choosing a lattice of $W$ over the ring of integers $\calO_{\kappa(x)}$ of $\kappa(x)$ and passing to pseudocharacters, we get $\Theta_x = \Theta_W \oplus \Theta_{W^*}$ as pseudocharacters over $\calO_{\kappa(x)}$. Since $\calO_{\kappa(x)}$ is local, by \Cref{resconst} $\Theta_W$ is residually constant and corresponds to a homomorphism $R^{\ps}_{\Thetabar_W} \to \calO_{\kappa(x)}$ for some uniquely determined reduction $\Thetabar_W$ defined over the residue field $k_x$ of $\calO_{\kappa(x)}$. Since the set of conjugacy classes of continuous $\GL_n(\overline{k_x})$-valued representations which occur in the semisimple representation associated to $\Thetabar_x$ over $\overline{k_x}$ is finite, all such $x$ are contained in the image of finitely many maps as in \Cref{finmappair}. This is sufficient, since $\overline Y_{\Thetabar}$ is Jacobson.
    \item $\overline X_{\Thetabar}^{\dec}$ is the union of the images of finitely many maps as in \Cref{finmapdec}. We argue similarly as in the previous step.
    \item $\overline X_{\Thetabar}^{\irr}$ is open in $\overline X_{\Thetabar}$, as the complement of $\overline X_{\Thetabar}^{\pair} \cup \overline X_{\Thetabar}^{\dec}$ (see \Cref{subdivision}). The subset $\overline X_{\Thetabar}^{\nspcl} \subseteq \overline X_{\Thetabar}^{\irr}$ is the complement of $\overline X_{\Thetabar}^{\spcl}$, which is closed in $\overline X_{\Thetabar}^{\irr}$. Hence $\overline X_{\Thetabar}^{\nspcl}$ is open in an open subset of $\overline X_{\Thetabar}$.
\end{enumerate}
\end{proof}

\begin{lemma}\label{Rpsbasechange} Let $\bar f : \kappa \to \kappa'$ be a homomorphism between either two finite or two local fields. Let $f : \Lambda \to \Lambda'$ be a local homomorphism of complete noetherian local rings with residue fields $\kappa$ and $\kappa'$ respectively and assume, that $f$ reduces to $\bar f$ on residue fields. Let $\Gamma$ be a profinite group and let $G$ be an affine $\Lambda$-group scheme. Let $\Thetabar \in \cPC^{\Gamma}_G(\kappa)$ and define $\Thetabar' := \Thetabar \otimes_{\kappa} \kappa'$. Then the natural map $R^{\ps}_{\Lambda', \Thetabar'} \to \RpsThetabar \widehat\otimes_{\Lambda} \Lambda'$ induced by
$$ \Def_{\Thetabar}(A) \to \Def_{\Lambda', \Thetabar'}(A \otimes_{\Lambda} \Lambda'), \quad \Theta \mapsto \Theta \otimes_{\Lambda} \Lambda'; \quad\quad A \in \Art_{\Lambda} $$
is an isomorphism. Here we indicate by $\Lambda'$ in the index the respective deformation functor and deformation ring for $\Thetabar'$.
\end{lemma}

\begin{proof} The proof of \cite[Proposition 4.7.6]{BJ_new} carries over in our setting.
\end{proof}

\subsection{Dimension bounds for $G=\Sp_{2n}$}
\label{secboundsforSp}

\begin{lemma}\label{dirsumlemma} Let $k$ be a field with $2 \in k^{\times}$. Then the symplectic Lie algebra $\sp_{2n,k}$ is a direct summand of $\gl_{2n,k}$ and of $\sl_{2n,k}$ and the corresponding projection maps $\gl_{2n,k} \twoheadrightarrow \sp_{2n,k}$ and $\sl_{2n,k} \twoheadrightarrow \sp_{2n,k}$ are equivariant for the adjoint action of the symplectic group $\Sp_{2n}$.
\end{lemma}

\begin{proof} Recall, that $\sp_{2n,k} = \{M \in \gl_{2n,k} \mid JM^{\top} + MJ = 0\}$, where $J = \SmallMatrix{0 & I_n \\ -I_n & 0}$. Right multiplication with $J$ is an isomorphism of $k$-vector spaces $- \cdot J : \gl_{2n,k} \to M_{2n}(k)$ and identifies $\sp_{2n,k}$ with the subspace of symmetric $2n \times 2n$ matrices. The symmetrization map $a : M_{2n}(k) \to M_{2n}(k), ~M \mapsto \tfrac{1}{2}(M+M^{\top})$ shows, that symmetric matrices are a direct summand of $M_{2n}(k)$. The map $\gl_{2n}(k) \to \gl_{2n}(k), ~M \mapsto a(MJ)J^{-1}$ is equivariant for the adjoint action of $\Sp_{2n}$ on $\gl_{2n}(k)$: Suppose $M \in M_{2n}(k)$ and $A \in \Sp_{2n}(k)$: Then
$$ a(AMA^{-1}J)J^{-1} = \frac{1}{2}(AMA^{-1} + J^{-1} (A^{-1})^{\top} M^{\top} A^{\top} J^{-1}) $$
and
$$ A a(MJ) J^{-1} A^{-1} =  \frac{1}{2}(AMA^{-1} + A J^{-1} M^{\top} J^{-1} A^{-1}) = \frac{1}{2}(AMA^{-1} + J^{-1} (A^{-1})^{\top} M^{\top} A^{\top} J^{-1}) $$
using $A \in \Sp_{2n}(k)$, so that $A^{-1} = JA^{\top} J^{-1}$ and $J^{\top} = J^{-1}$. We also obtain, that the projection map $\gl_{2n,k} \twoheadrightarrow \sp_{2n,k}$ is split by the inclusion and equivariant for the adjoint action of $\Sp_{2n}$. Since $\sp_{2n,k} \subseteq \sl_{2n,k}$, the restriction $\sl_{2n,k} \to \sp_{2n,k}$ is still split by the inclusion and $\Sp_{2n}$-equivariant.
\end{proof}

We observe that certain completed local rings at dimension $1$ points $x$ are deformation rings for a deformation problem with residue field $\kappa(x)$. It is for this reason, that we also treat cases \eqref{kappa2} and \eqref{kappa3} in \Cref{secdefGPC}.

\begin{proposition}\label{completionoflocalring} Assume that $\kappa$ is finite, that $\RpsThetabar$ is noetherian and let $x \in \Spec(\RpsThetabar)$ be a dimension $1$ point with residue field $\kappa(x)$. By \cite[Lemma 3.17]{BIP} $\kappa(x)$ is a local field. Let $\varphi_x : \kappa(x) \otimes_{\Lambda} \RpsThetabar \to \kappa(x)$ be the induced map and let $\frakm := \ker(\varphi_x)$. Then the following two rings are canonically isomorphic:
\begin{enumerate}
    \item The universal pseudodeformation ring $R_{\Theta_x}^{\ps}$.
    \item The completion $\widehat R$ of $\kappa(x) \otimes_{\Lambda} \RpsThetabar$ at $\frakm$.
\end{enumerate}
The isomorphism is induced by the map $i_x : \kappa(x) \otimes_{\Lambda} \RpsThetabar \to R_{\Theta_x}^{\ps}$.
\end{proposition}

\begin{proof} 
    The $\widehat R$-valued pseudocharacter $\Theta^u \otimes_{\RpsThetabar} \widehat R$ obtained from the universal deformation $\Theta^u$ of $\Thetabar$ by base extension along the map $\RpsThetabar \to \widehat R$ is a lift of $\Theta_x$, so we get a map $R_{\Theta_x}^{\ps} \to \widehat R$ and $\Theta^u_{\Theta_x} \otimes_{R_{\Theta_x}^{\ps}} \widehat R = \Theta^u \otimes_{\RpsThetabar} \widehat R$, where $\Theta^u_{\Theta_x}$ is the universal deformation of $\Theta_x$.
    Let $\frakm$ be the maximal ideal of $R_{\Theta_x}^{\ps}$ and write $A_i := R_{\Theta_x}^{\ps}/\frakm^i$.
    The universal deformation of $\Theta_x$ induces a continuous $A_i$-valued lift $\Theta^i$ of $\Theta_x$.
    Since $\Theta^1$ takes values in a ring of integers $\calO_{\kappa(x)}$ of $\kappa(x)$, we can apply \Cref{calclem} (3) (which holds in the function field case with the same proof) to find a module-finite $\calO_{\kappa(x)}$-subalgebra $S$ of $A_i$, which surjects onto $\calO_{\kappa(x)}$ and such that $\Theta^i$ takes values in $S$. It follows, that $\Theta^i$ seen as an $S$-valued pseudocharacter is a deformation of $\Thetabar$. We thus get an induced continuous map $\RpsThetabar \to A_i$, hence a map $\widehat R \to R^{\ps}_{\Theta_x}$ such that $(\Theta^u_{\Theta_x} \otimes_{R_{\Theta_x}^{\ps}} \widehat R) \otimes_{\widehat R} R_{\Theta_x}^{\ps} = \Theta^u_{\Theta_x}$, thus the composition $R^{\ps}_{\Theta_x} \to \widehat R \to R^{\ps}_{\Theta_x}$ is the identity. The composition $\widehat R \to R^{\ps}_{\Theta_x} \to \widehat R$ is the identity as well.
\end{proof}

\begin{remark} It is a consequence of \Cref{completionoflocalring}, that $\RpsThetabar$ in \Cref{tfgfingen} is also noetherian when $\kappa$ is a local field.
\end{remark}

The following Proposition is the analog of \cite[Lemma 5.1.6]{BJ_new} for $G=\Sp_{2n}$.

\begin{proposition}\label{nspclregular} Let $\Thetabar \in \cPC_{\Sp_{2n}}^{\Gamma_F}(\kappa)$ with $\kappa$ a finite field of characteristic $p > 2$ and let $\Lambda$ be a coefficient ring for $\kappa$. 
Let $x \in U := \overline Y^{\irr}_{\Thetabar}$ be a closed point. By \cite[Lemma 3.17]{BIP} the residue field $\kappa(x)$ of $x$ is a local field. Let $R_{\Theta_x}^{\ps}$ be the universal pseudodeformation ring of the $\Sp_{2n}$-pseudocharacter $\Theta_x$ attached to $x$. By \Cref{recwithfiniteextn} there is a finite extension $\kappa'$ of $\kappa(x)$, such that $\Theta_x' := \Theta_x \otimes_{\kappa(x)} \kappa'$ is induced by a continuous absolutely irreducible representation $\rhobar : \Gamma_F \to G(\kappa')$.
\begin{enumerate}
    \item \label{lower_bound_nspcl}
    \begin{enumerate}
        \item Suppose, that $x$ is non-special. Then $R_{\Theta_x'}^{\ps}$ is regular of dimension $n(2n+1) \cdot [F : \bbQ_p]$.
        \item If in addition $U^{\nspcl} \neq \emptyset$, then $U^{\nspcl}$ is regular and equidimensional of dimension $n(2n+1) \cdot [F : \bbQ_p] - 1$. 
    \end{enumerate} \label{nr_1}
    \item Suppose, that $\zeta_p \notin F$ and that $x$ is special. Then $\dim R_{\Theta_x'}^{\ps} \in \{n(2n+1) \cdot [F : \bbQ_p], n(2n+1) \cdot [F : \bbQ_p] + 1\}$. \label{nr_2}
    \item If $\zeta_p \notin F$, then $\dim U \leq n(2n+1) \cdot [F : \bbQ_p]$. \label{nr_3}
\end{enumerate}
\end{proposition}

\begin{proof} 
\underline{Proof of \eqref{nr_1}(a).} If $\zeta_p \notin F$, then by \cite[Lemma 5.1.1 Case I]{BJ_new}, we have $H^2(\Gamma_F, \gl_{2n, \kappa'}) = 0$. Since $2$ is invertible in $\kappa'$, by \Cref{dirsumlemma} $\sp_{2n,\kappa'}$ is a direct summand of $\gl_{2n, \kappa'}$ and so $H^2(\Gamma_F, \sp_{2n, \kappa'}) = 0$. If $\zeta_p \in F$, then we have $H^2(\Gamma_F, \sl_{2n}) = 0$ by \cite[Lemma 5.1.1 case II]{BJ_new}. By \Cref{dirsumlemma}, $\sp_{2n}$ is also a direct summand of $\sl_{2n}$. It follows, that $H^2(\Gamma_F, \sp_{2n}) = 0$. Let $R_{\Theta_x'}^{\ps}$ be the universal pseudodeformation ring of $\Theta_x'$ over a coefficient ring $\Lambda' \supseteq \Lambda$ with residue field $\kappa'$. 
By \Cref{pdeflemma} $R_{\Theta_x'}^{\ps}$ is regular of dimension $\dim \sp_{2n, \kappa'} \cdot [F : \bbQ_p] + h^0(\Gamma_F, \sp_{2n, \kappa'})$. By Schur's lemma $h^0(\Gamma_F, \gl_{2n, \kappa'}) = 1$. Clearly $H^0(\Gamma_F, \gl_{2n, \kappa'})$ is spanned by the diagonal matrices in $\gl_{2n, \kappa'}$.
These are not contained in $\sp_{2n,\kappa'}$, hence $h^0(\Gamma_F, \sp_{2n, \kappa'}) = 0$.

\underline{Proof of \eqref{nr_1}(b).} Assume, that $x$ is non-special. By \Cref{completionoflocalring}, the universal pseudodeformation ring $R_{\Theta_x}^{\ps}$ can be identified with the completion of $\RpsThetabar \otimes_{\Lambda} \kappa(x)$ at the kernel of the natural map $\RpsThetabar \otimes_{\Lambda} \kappa(x) \to \kappa(x)$ attached to $x$. Since $x$ is a $1$-dimensional point of $\RpsThetabar$ with residue characteristic $p$, it follows from \cite[Lemma 3.3.5]{BJ_new}, that $x$ is a regular point of dimension $n(2n+1) \cdot [F : \bbQ_p] - 1$ of $U^{\nspcl}$. Let $U^{\sing} \subseteq U^{\nspcl}$ be the closed subscheme of singular points. By \stackcite{02J4} and \stackcite{01TB}, the closed points are dense in $U^{\sing}$. But since all closed points of $U^{\nspcl}$ are regular, $U^{\sing}$ must be empty. Since closed points are dense in $U^{\nspcl}$, it follows that $U^{\nspcl}$ is equidimensional of dimension $n(2n+1) \cdot [F : \bbQ_p] - 1$.

\underline{Proof of \eqref{nr_2}.} As in \eqref{nr_1}(a) $h^0(\Gamma_F, \sp_{2n, \kappa'}) = 0$. Since $x$ is special, we have $\rhobar \cong \rhobar(1)$ by \cite[Lemma 5.1.1 Case (I)]{BJ_new}. We have $H^2(\Gamma_F, \gl_{2n,\kappa'}) \cong \Hom_{\Gamma_F}(\rhobar, \rhobar(1)) \cong \kappa'$ since $\rhobar$ is irreducible, hence $h^2(\Gamma_F, \sp_{2n,\kappa'}) \leq 1$. The case when $h^2(\Gamma_F, \sp_{2n,\kappa'}) = 0$ is already covered in \Cref{pdeflemma}, so we assume $h^2(\Gamma_F, \sp_{2n,\kappa'}) = 1$. By the Euler characteristic formula \cite[Theorem 3.4.1]{BJ_new}
$$ h^1(\Gamma_F, \sp_{2n, \kappa'}) = n(2n+1)[F : \bbQ_p] + 1 $$
and by \Cref{descrdefring}, $R^{\ps}_{\Theta_x'}$ is a quotient of $\kappa'[[x_1, \dots, x_{n(2n+1)[F : \bbQ_p] + 1}]]$ by an ideal generated by at most one element, so the assertion follows.

\underline{Proof of \eqref{nr_3}.} Let $x \in U$ be a closed point. Cases \eqref{nr_1} and \eqref{nr_2} imply, that $\dim R_{\Theta_x}^{\ps} \leq n(2n+1)[F : \bbQ_p] + 1$. As in \eqref{nr_1}(b), identifying $R_{\Theta_x}^{\ps}$ with a completion of $\RpsThetabar \otimes_{\Lambda} \kappa(x)$ and applying \cite[Lemma 3.3.5]{BJ_new}, we see that $U$ has dimension $\leq n(2n+1)[F : \bbQ_p]$.
\end{proof}

\begin{theorem}\label{estspcl} Assume $G= \Sp_{2n}$. Then $\dim \overline X_{\Thetabar}^{\spcl} \leq 2n^2[F : \bbQ_p] + 1$. In particular, if $n[F:\bbQ_p] \geq 3$ and if $\overline X_{\Thetabar}$ contains a non-special point, then $\dim \overline X_{\Thetabar}^{\spcl} \leq \dim \overline X_{\Thetabar} - 2$.
\end{theorem}

\begin{proof} Since $\overline X_{\Thetabar}^{\spcl}$ is a closed subspace of $\overline X_{\iota(\Thetabar)}^{\spcl}$ by \Cref{openandclosedsubspaces} and the latter can be identified with the special locus of the pseudodeformation space of the determinant law $\Dbar$ attached to $\iota(\Thetabar)$ by \Cref{chenevierlafforguetotal}, we can take the estimate \cite[Theorem 5.4.1 (a)]{BJ_new} to obtain $\dim \overline X_{\Thetabar}^{\spcl} \leq 2n^2[F : \bbQ_p] + 1$.
If $\overline X_{\Thetabar}$ contains a non-special point, then $\dim \overline X_{\Thetabar} \geq \dim \overline X_{\Thetabar}^{\nspcl} = n(2n+1)[F : \bbQ_p]$ by \Cref{nspclregular} (1)(b). We get
$\dim \overline X_{\Thetabar} - \dim \overline X_{\Thetabar}^{\spcl} \geq n [F : \bbQ_p] - 1 \geq 2$.
\end{proof}

\begin{theorem}\label{estred} Assume $G = \Sp_{2n}$.
\begin{enumerate}
    \item $\dim \overline X_{\Thetabar}^{\dec} \leq n(2n+1)[F : \bbQ_p] - 4(n-1)[F : \bbQ_p]$. \\
    In particular, if $\overline X_{\Thetabar}$ contains a non-special point, then $\dim \overline X_{\Thetabar}^{\dec} \leq \dim \overline X_{\Thetabar} - 4$. \label{induction_1}
    \item $\dim \overline X_{\Thetabar}^{\pair} \leq n^2[F : \bbQ_p] + 1$. \\
    In particular, if $\overline X_{\Thetabar}$ contains a non-special point and $n[F:\bbQ_p] \geq 2$, then $\dim \overline X_{\Thetabar}^{\pair} \leq \dim \overline X_{\Thetabar} - 3$. \label{induction_2}
    \item $\dim \overline X_{\overline \Theta} \leq n(2n+1)[F : \bbQ_p]$. \\
    In particular, if $\overline X_{\Thetabar}$ contains a non-special point then equality holds. \label{induction_3}
\end{enumerate}
\end{theorem}

\begin{proof} We begin with case \eqref{induction_2}. There are finitely many ways to write $\Thetabar = \Thetabar_1 \oplus \Thetabar_1^*$ for some $\GL_n$-pseudocharacter $\Thetabar_1$ and we may assume, that there is at least one way. The sum yields an $\Sp_{2n}$-pseudocharacter, as explained in \Cref{secsumtensor}. By \Cref{finmappair}, the map $\iota^{\pair}_{\Thetabar_1} : \overline X_{\Thetabar_1} \to \overline X_{\Thetabar}$ induced by $\Theta_1 \mapsto \Theta_1 \oplus \Theta_1^*$ is finite.
We have an inclusion
$$ \overline X^{\pair}_{\Thetabar} \subseteq \bigcup_{\Thetabar_1 \oplus \Thetabar_1^* = \Thetabar} \iota^{\pair}_{\Thetabar_1}(\overline X_{\Thetabar_1}) $$
and the estimate
$$ \dim \overline X^{\pair}_{\Thetabar} \leq \dim \overline X_{\Thetabar_1} = n^2[F : \bbQ_p] + 1, $$
where the last equality follows from \cite[Theorem 5.5.1 (a)]{BJ_new} after applying the bijection \Cref{Rpsisom}.
If $\overline X_{\overline \Theta}$ contains a non-special point, then by \Cref{nspclregular} \eqref{lower_bound_nspcl}(b), we have a lower bound $\dim \overline X_{\overline \Theta} \geq n(2n+1)[F : \bbQ_p]$.
Since $n(n+1)[F : \bbQ_p] - 1\geq 3$, this completes the proof of \eqref{induction_2}.

For \eqref{induction_1} and \eqref{induction_3}, we proceed by induction over $n$. Assume that $n=1$. Then the decomposable locus $\overline X_{\Thetabar}^{\dec}$ is empty, so \eqref{induction_1} holds.
Using the stratification $\overline X_{\overline \Theta} = \overline X_{\overline \Theta}^{\nspcl} \cup \overline X_{\overline \Theta}^{\spcl} \cup \overline X_{\overline \Theta}^{\pair}$ from \Cref{subdivision} and the bounds $\dim \overline X_{\overline \Theta}^{\nspcl} \leq 3[F:\bbQ_p]$ from \Cref{nspclregular} \eqref{lower_bound_nspcl}(b), $\dim \overline X_{\Thetabar}^{\spcl} \leq 2[F : \bbQ_p] + 1$ from \Cref{estspcl} and $\dim \overline X_{\overline \Theta}^{\pair} \leq [F : \bbQ_p] + 1$ from the previous step, we get $\dim \overline X_{\overline \Theta} \leq 3[F:\bbQ_p]$.

We now assume that $n \geq 2$ and the entire theorem has been proved for all $n' < n$. Since our assertions are only about dimensions, by \Cref{Rpsbasechange} we may assume that $\iota(\Thetabar)$ comes from a representation $\Gamma_F \to \GL_{2n}(\kappa)$ and that the irreducible constituents are absolutely irreducible.

\underline{Proof of \eqref{induction_1}.} There are up to isomorphism only finitely many ways to write $\Thetabar$ as a direct sum of two symplectic pseudocharacters $\Thetabar_1 \oplus \Thetabar_2$. Here the notion of direct sum is that for symplectic pseudocharacters, introduced in \Cref{secsumtensor}.
By \Cref{finmapdec}, the map $\iota^{\dec}_{\Thetabar_1, \Thetabar_2} : \overline X_{\Thetabar_1} \widehat\times_{\calO} \overline X_{\Thetabar_2} \to \overline X_{\Thetabar}$ is finite. We have an inclusion
$$ \overline X^{\dec}_{\Thetabar} \subseteq \bigcup_{\Thetabar_1 \oplus \Thetabar_2 = \Thetabar} \iota^{\dec}_{\Thetabar_1, \Thetabar_2}(\overline X_{\Thetabar_1} \widehat\times_{\calO} \overline X_{\Thetabar_2}) $$
where the right hand side is a closed subset of $\overline X_{\Thetabar}$.
Suppose $\Thetabar = \Thetabar_1 \oplus \Thetabar_2$ is a decomposition into an $\Sp_{2a}$-pseudocharacter $\Thetabar_1$ and an $\Sp_{2b}$-pseudocharacter $\Thetabar_2$ for $a+b = n$ with $a,b \geq 1$.
Then since $\iota^{\dec}_{\Thetabar_1, \Thetabar_2}$ is finite and by part \eqref{induction_3} of the inductive hypothesis, we have
\begin{align*}
    \dim \iota^{\dec}_{\Thetabar_1, \Thetabar_2}(\overline X_{\Thetabar_1} \widehat\times \overline X_{\Thetabar_2}) &\leq \dim (\overline X_{\Thetabar_1} \widehat\times \overline X_{\Thetabar_2}) \\
    &\leq a(2a+1)[F : \bbQ_p] + b(2b+1)[F : \bbQ_p]
\end{align*}
Calculating
\begin{align*}
    &\phantom{=} n(2n+1)[F : \bbQ_p] - a(2a+1)[F : \bbQ_p] - b(2b+1)[F : \bbQ_p] \\
    &= 4ab[F : \bbQ_p] \geq (\min_{\substack{a+b=n \\ a,b \geq 1}} 4ab) \cdot [F : \bbQ_p] = 4(n-1)[F : \bbQ_p]
\end{align*}
we obtain the desired bound $\dim \overline X^{\dec}_{\Thetabar} \leq n(2n+1)[F : \bbQ_p] - 4(n-1)[F : \bbQ_p]$.
If $\overline X_{\overline \Theta}$ contains a non-special point, then by \Cref{nspclregular} \eqref{lower_bound_nspcl}(b), we have a lower bound $\dim \overline X_{\overline \Theta} \geq n(2n+1)[F : \bbQ_p]$.
Since $4(n-1)[F : \bbQ_p] \geq 4$, this implies the assertion.

\underline{Proof of \eqref{induction_3}.} Let us recollect all upper bounds, we have established.
\begin{align*}
    &\dim \overline X_{\overline \Theta}^{\nspcl} \overset{\eqref{nspclregular}\text{ \eqref{lower_bound_nspcl}(b)}}{\leq} n(2n+1) \cdot [F : \bbQ_p] \\
    &\dim \overline X_{\overline \Theta}^{\spcl} \overset{\eqref{estspcl}}{\leq} 2n^2 \cdot [F : \bbQ_p] + 1 \\
    &\dim \overline X_{\overline \Theta}^{\dec} \overset{\eqref{induction_1}}{\leq} n(2n+1)[F : \bbQ_p] - 4(n-1)[F : \bbQ_p] \\
    &\dim \overline X_{\Thetabar}^{\pair} \overset{\eqref{induction_2}}{\leq} n^2[F : \bbQ_p] + 1
\end{align*}
Using the stratification $\overline X_{\overline \Theta} = \overline X_{\overline \Theta}^{\nspcl} \cup \overline X_{\overline \Theta}^{\spcl} \cup \overline X_{\overline \Theta}^{\dec} \cup \overline X_{\overline \Theta}^{\pair}$ from \Cref{subdivision}, we obtain the desired dimension bound for $\overline X_{\overline \Theta}$. If $\overline X_{\overline \Theta}$ contains a non-special point, we obtain equality from \Cref{nspclregular} \eqref{lower_bound_nspcl}(b).
\end{proof}

\begin{corollary}\label{endcor} Assume $G = \Sp_{2n}$ and that $\Thetabar$ comes from a residual representation $\rhobar : \Gamma_F \to \Sp_{2n}(\kappa)$, which is absolutely irreducible under the standard embedding into $\GL_{2n}(\kappa)$. Then $\dim \overline X_{\overline \Theta} = n(2n+1)[F : \bbQ_p]$ and in particular $\overline X_{\Thetabar}$ contains a non-special point.
\end{corollary}

\begin{proof} By \Cref{compprop} and \Cref{centrtriv} $\overline X_{\Thetabar}$ identifies with the deformation functor of $\rhobar$. From \cite[Proposition 5.7]{MR1643682} and the Euler characteristic formula \cite[Theorem 3.4.1]{BJ_new}, we know, that $\overline X_{\Thetabar} \geq h^1(\Gamma_F, \sp_{2n}) - h^2(\Gamma_F, \sp_{2n}) = h^0(\Gamma_F, \sp_{2n}) + n(2n+1)[F:\bbQ_p]$. By absolute irreducibility and Schur's lemma $h^0(\Gamma_F, \sp_{2n}) = 0$. So from \Cref{estspcl}, we see, that the special locus $\overline X_{\Thetabar}^{\spcl}$ is strictly contained in $\overline X_{\Thetabar}$ and there must be a non-special point in $\overline X_{\Thetabar}$.
\end{proof}

\begin{remark} It is likely that the arguments of \Cref{secboundsforSp} carry over to $G = \GSp_{2n}$ with minor modifications. It is also likely that in future work we will be able to deduce the existence of non-special points for arbitrary residual $\Sp_{2n}$- and $\GSp_{2n}$-pseudocharacters, so that in \Cref{estred} (3) equality holds.
\end{remark}

\textbf{Competing interests.} The author declares none.

\bibliographystyle{alpha}
\bibliography{literature}

\end{document}